\documentclass[a4paper,12pt]{amsart}
\usepackage{mathrsfs}
\usepackage{appendix}
\usepackage{amssymb}
\usepackage{fancyhdr}
\usepackage{charter}
\usepackage{typearea}
\usepackage{color}
\usepackage{amsmath,amstext,amsthm,amscd}
\usepackage{hyperref}

\usepackage[a4paper,top=3cm,bottom=3cm,left=2cm,right=2cm]{geometry}

\makeatletter
      \def\@setcopyright{}
      \def\serieslogo@{}
      \makeatother

\newcommand{\Complex}{\mathbb C}
\newcommand{\Real}{\mathbb R}
\newcommand{\R}{\mathbb R}
\newcommand{\N}{\mathbb N}
\newcommand{\ddbar}{\overline\partial}
\newcommand{\pr}{\partial}
\newcommand{\ol}{\overline}
\newcommand{\Td}{\widetilde}
\newcommand{\norm}[1]{\left\Vert#1\right\Vert}
\newcommand{\abs}[1]{\left\vert#1\right\vert}

\newcommand{\set}[1]{\left\{#1\right\}}
\newcommand{\To}{\rightarrow}

\theoremstyle{plain}
\newtheorem{thm}{Theorem}[section]
\newtheorem{cor}[thm]{Corollary}

\newtheorem{lem}[thm]{Lemma}

\theoremstyle{definition}
\newtheorem{defn}[thm]{Definition}

\theoremstyle{remark}

\newtheorem{ex}[thm]{Example}

\numberwithin{equation}{section}

\begin{document}
\title[Equivariant Kodaira embedding for CR manifolds with circle action]
{Equivariant Kodaira embedding for CR manifolds with circle action}
\author[Chin-Yu Hsiao]{Chin-Yu Hsiao}
\address{Institute of Mathematics, Academia Sinica and National Center for Theoretical Sciences, Astronomy-Mathematics Building, No. 1, Sec. 4, Roosevelt Road, Taipei 10617, Taiwan}
\thanks{The first author was partially supported by Taiwan Ministry of Science of Technology project 104-2628-M-001-003-MY2 and the Golden-Jade fellowship of Kenda Foundation}
\email{chsiao@math.sinica.edu.tw or chinyu.hsiao@gmail.com}
\author[Xiaoshan Li]{Xiaoshan Li}
\address{School of Mathematics
and Statistics, Wuhan University, Hubei 430072, China \& Institute of Mathematics, Academia Sinica, 6F, Astronomy-Mathematics Building,
No.1, Sec.4, Roosevelt Road, Taipei 10617, Taiwan}
\thanks{The second author was  supported by Central university research Fund 2042015kf0049, Postdoctoral Science Foundation of China 2015M570660 and NSFC No. 11501422}
%\thanks{The second-named author was partially supported by Central university research Fund 2042015kf0049.}
\email{xiaoshanli@whu.edu.cn or xiaoshanli@math.sinica.edu.tw}

\author[George Marinescu]{George Marinescu}
\address{Universit{\"a}t zu K{\"o}ln,  Mathematisches Institut,
    Weyertal 86-90,   50931 K{\"o}ln, Germany\\
    \& Institute of Mathematics `Simion Stoilow', Romanian Academy,
Bucharest, Romania}
\thanks{The third author gratefully acknowledges the support of the Academia Sinica at Taipei,
where part of this paper was written.}
\email{gmarines@math.uni-koeln.de}

\begin{abstract}
%Let $X$ be a compact CR manifold with a transversal CR locally free $S^1$-action and let $L^k$ be the $k$-th tensor power of a rigid positive CR line bundle $L$ over $X$. Without any Levi curvature assumption, we prove that the partial Szeg\"o kernel admits a full asymptotic expansion and by using these asymptotics, we establish Kodaira embedding theorem on CR manifolds with transversal CR locally free $S^1$-actions.
We consider  a compact CR manifold with a transversal CR locally free circle action
endowed with  a rigid positive CR line bundle.
We prove that a certain weighted Fourier-Szeg\H{o} kernel of the CR sections
in the high tensor powers admits a full asymptotic expansion.
As a consequence, we establish an equivariant Kodaira embedding theorem.
\end{abstract}

\maketitle \tableofcontents

\section{Introduction and statement of the main results} \label{s-gue150806}

The goal of this paper is to study the equivariant embedding of CR manifolds
with circle action. The embedding of CR manifolds in general is a subject with
long tradition. One paradigm is the embedding theorem of compact strictly pseudoconvex
CR manifolds. A famous theorem of Louis Boutet de Monvel~\cite{BdM1:74b}
asserts that such manifolds can be embedded
by CR maps into the complex Euclidean space, provided the dimension of the manifold is
greater than or equal to five.

In dimension three there are non-embeddable
compact strictly pseudoconvex CR manifolds (see e.\,g.\ Burns \cite{Bu:77}, where the boundary of
the non-fillable example of strictly pseudoconcave manifold by Grauert \cite{G94}, 
Andreotti-Siu \cite{AS70} and Rossi \cite{Ro65}
is shown to be non-embeddable).
However, if the manifold admits a circle action, then it is embeddable, by a theorem
of Lempert \cite{Lem92}.
In the study of CR functions, which would eventually provide an embedding,
it is natural to look to the orthogonal projector on the space of square integrable
CR functions, called Szeg\H{o} projector. The Schwartz kernel of this projector
is called Szeg\H{o} kernel.
In this spirit,  a proof based on the Szeg\H{o} kernel of Lempert's embedding theorem was given
in \cite{HM14}. Using the Szeg\H{o} kernel of the Fourier components it was recently
shown in \cite[Theorem 1.2]{HHL15} that there exists an \emph{equivariant} embedding
of strictly pseudoconvex CR manifolds with circle action.

Leaving the territory of strictly pseudoconvex CR manifolds, the natural idea arrises
to embed CR manifolds into the projective space by means of CR sections of a
CR line bundle of positive curvature \cite{Ge89,HHL15,Hsiao12,Hsiao14,HL15,HM09,HM14,HM15,Ma96,OS00}.
This is the analogue of the \emph{Kodaira embedding theorem} from complex geometry.
In the case of CR manifolds we have to use an analytic method, while Kodaira's
original proof relied on cohomology vanishing theorems.
Analytic proofs of the Kodaira embedding theorem for K\"ahler and symplectic manifolds,
based on the Bergman kernel asymptotics,
were given in \cite{BU,MM07,SZ02,Zel98}. 
In this paper we will use Szeg\H{o} kernel analogues on CR manifolds
of the Bergman kernel asymptotics on K\"ahler or symplectic manifolds
\cite{Cat:97,HM12,MM07,MM08,SZ02,Zel98}.
A motivating example is the quadric
\[
\big\{[z]\in\Complex\mathbb{P}^{N-1};
\, |z_1|^2+\ldots+|z_q|^2-|z_{q+1}|^2-\ldots-|z_N|^2=0\big\}
\]
which is a CR manifold possessing a positive line bundle and a circle action.
In \cite{BU,SZ02,Zel98} the Szeg\H{o} kernel on a strictly
pseudoconvex CR manifold with trivial line bundle \cite{BouSj76} (see also \cite{HM14})
was used to study the Bergman kernel
on a K\"ahler manifold, while here we study the Szeg\H{o} kernel
for tensor powers of a CR line bundle. 

We are thus led to the problem of \emph{equivariant Kodaira embedding}
of CR manifolds with circle action, which will be the subject of this paper.
We will prove that a certain weighted Fourier-Szeg\H{o} kernel
admits a full asymptotic expansion and by using these asymptotics, we will show that
if $X$ admits a transversal CR locally free $S^1$-action and there is a rigid positive
CR line bundle $L$ over $X$, then $X$ can be CR embedded into projective space
\emph{without any assumption} of the Levi form.
In particular, when $X$ is Levi-flat, we improve to $C^\infty$
the regularity in the Kodaira embedding theorem of Ohsawa and Sibony (see Corollary \ref{C:Lf}). 
%It should be mentioned that by using Szeg\H{o} kernel, we established in this paper and 
%in~\cite{HHL15},~\cite{Hsiao14},~\cite{HM14} and~\cite{HM15} embedding theorems in CR geometry. 
%The main inspiration of all these works come from the works of 
%Zelditch~\cite{Zel98} and Shiffman-Zelditch~\cite{SZ02}. 

%\subsection{Main results}\label{s-gue150806I}

Let us now formulate our main results. We refer to Section~\ref{s:prelim} for some standard
notations and terminology used here. Let $(X,T^{1,0}X)$ be a compact CR manifold
of dimension $2n-1$, $n\geqslant2$, endowed with a locally free $S^1$-action
$S^1\times X\to X$, $(e^{i\theta},x)\mapsto e^{i\theta}  x$ and we let $T$
be the infinitesimal generator of the $S^1$-action. %(see \eqref{e-gue150808}).

We assume that this $S^1$-action is transversal CR, that is, $T$ preserves the CR structure $T^{1,0}X$,
and $T$ and $T^{1,0}X\oplus\overline{T^{1,0}X}$ generate the complex tangent bundle to $X$.
In our paper we will make systematic use of appropriate coordinates introduced by
Baouendi-Rothschild-Treves~\cite{BRT85}.
Namely, if $X$ admits a transversal CR locally free $S^1$-action, then for each point $p\in X$
there exist a coordinate neighborhood $U$ with coordinates $(x_1,\ldots,x_{2n-1})$,
centered at $p=0$, and $\eta>0$, $\varepsilon_0>0$,
such that, by setting $z_j=x_{2j-1}+ix_{2j}$, $j=1,\ldots,n-1$, $x_{2n-1}=\theta$
and $D=\{(z, \theta)\in U: \abs{z}<\eta, |\theta|<\varepsilon_0\}\subset U$, we have
\begin{equation}\label{e-can1}
T=\frac{\partial}{\partial\theta}\:\:\text{on $D$},\\
\end{equation}
and the vector fields
\begin{equation}\label{e-can2}
Z_j=\frac{\partial}{\partial z_j}-i\frac{\partial\phi}{\partial z_j}(z)\frac{\partial}{\partial\theta},
\:\:j=1,\ldots,n-1,
\end{equation}
form a basis of $T_x^{1,0}X$ for each $x\in D$, where $\phi\in C^\infty(D,\mathbb R)$
is independent of $\theta$.
We call $(x_1,\ldots,x_{2n-1})$ canonical coordinates, $D$ canonical coordinate patch and
$(D,(z,\theta),\phi)$ a BRT trivialization.
%see Theorem \ref{t-gue150514}.

%We assume that this action is CR, that is, $T$ preserves the CR structure $T^{1,0}X$ on $X$,
%and the $S^1$-action is transversal to the CR structure, that is,
%$T$ and $T^{1,0}X\oplus\overline{T^{1,0}X}$ generate the complex tangent bundle to $X$.
%$e^{i\theta}$, We assume that $X$ admits a transversal
%CR locally free $S^1$-action $e^{i\theta}$, $0\leq\theta<2\pi$ and we let $T$
%be the infinitesimal generator of the $S^1$-action (see \eqref{e-gue150808}).  ?

A vector bundle is called rigid (resp.\ rigid CR) vector bundle if there is a
family of trivializations which cover $X$ such that the entries of the transition
matrices are functions annihilated by $T$ (resp.\ CR functions annihilated by $T$),
see Definition~\ref{d-gue150508dI}. The corresponding frames are called rigid frames. 
A Hermitian metric on a rigid vector bundle is called rigid if for every rigid frame 
$\set{f_1,\ldots,f_r}$, the inner products
of $f_j$ and $f_\ell$  are annihilated by $T$ for any $j,\ell$.

Let $L$ be a rigid CR line bundle over $X$ and let $L^k$ be the $k$-th power of $L$.
Let
\[\ddbar_b:\Omega^{0,q}(X,L^k)\To\Omega^{0,q+1}(X,L^k),\]
be the tangential Cauchy-Riemann operator with values in $L^k$.
For every $m\in\mathbb Z$, put
\begin{equation}\label{e-gue150806I}
C^\infty_m(X,L^k):=\set{u\in C^\infty(X,L^k);\, Tu=imu},
\end{equation}
and
\begin{equation}\label{e-gue150806II}
\mathcal{H}^0_{b,m}(X,L^k):=\set{u\in C^\infty_m(X,L^k);\, \ddbar_bu=0}.
\end{equation}
%The starting point of this paper is that without any Levi curvature assumption,
Since $X$ is a compact manifold we have
for every $m\in\mathbb Z$ (see \cite[Theorem 1.23]{HL15} and also
Theorem \ref{t-gue150517aI})
\begin{equation}\label{e-gue150806III}
{\rm dim\,}\mathcal{H}^{0}_{b,m}(X,L^k)<\infty.
\end{equation}
For $\lambda>0$, put
\begin{equation}\label{e-gue150806IV}
\mathcal{H}^0_{b,\leq\lambda}(X,L^k):=\bigoplus_{\abs{m}\leq\lambda}\mathcal{H}^0_{b,m}(X,L^k).
\end{equation}

We assume further that $L$ is endowed with a rigid Hermitian metric $h$.
%(see Definition~\ref{d-gue150808gII} and Lemma~\ref{l-rigid}).
The curvature of $(L,h)$ at a point $x\in X$ is denoted by $R^L_x$ (cf.\ Definition \ref{d-gue150808g}), and
$(L,h)$ is called positive if $R^L_x$ is positive definite at any point $x\in X$.
%is positive (see Definition~\ref{d-gue150808gI}) and fix a positive rigid Hermitian
%fiber metric $h^L$ of $L$ (see Definition~\ref{d-gue150808gII} and Lemma~\ref{l-rigid}).
The Hermitian metric on $L^k$ induced by $h$ is denoted by $h^k$.

The bundle $\Complex TX$ is rigid and we can take a rigid Hermitian metric
$\langle\,\cdot\,|\,\cdot\,\rangle$ on $\Complex TX$
such that
$$T^{1,0}X\perp T^{0,1}X,\:\: T\perp (T^{1,0}X\oplus T^{0,1}X),\:\:
\langle\,T\,|\,T\,\rangle=1$$
and $\langle\,u\,|v\,\rangle$ is real if $u, v$ are real tangent vectors
(see Theorem~\ref{t-gue150514fa}).
We denote by $dv_X$ the volume form induced by $\langle\,\cdot\,|\,\cdot\,\rangle$.

Let $\omega_0\in C^\infty(X,T^*X)$ be the real $1$-form of unit length
annihilating $T^{1,0}X\oplus T^{0,1}X$ and satisfying $\omega_0(T)=-1$.
The Levi form $\mathcal L_x$ at a point $x\in X$
is the Hermitian quadratic form on $T^{1,0}_xX$ given by
$\mathcal{L}_x(U,\ol V)=-\frac{1}{2i}\langle\,d\omega_0(x)\,,\,U\wedge\ol V\,\rangle$,
$U, V\in T^{1,0}_xX$.

Let $(\,\cdot\,|\,\cdot\,)_k=(\,\cdot\,|\,\cdot\,)$ be the $L^2$ inner product on $C^\infty(X,L^k)$ induced by $h^k$ and
$dv_X$. Let $L^2(X,L^k)$ be the completion of $C^\infty(X,L^k)$ with respect
to $(\,\cdot\,|\,\cdot\,)$. We extend $(\,\cdot\,|\,\cdot\,)$ to $L^2(X,L^k)$.

For every $m\in\mathbb Z$, let $L^2_m(X,L^k)\subset L^2(X,L^k)$
be the completion of $C^\infty_m(X,L^k)$ with respect to $(\,\cdot\,|\,\cdot\,)$.
Let
\begin{equation}\label{e-gue150807}
Q^{(0)}_{m,k}:L^2(X,L^k)\To L^2_m(X,L^k)
\end{equation}
be the orthogonal projection with respect to $(\,\cdot\,|\,\cdot\,)$.
We have the Fourier decomposition
\[L^2(X,L^k)=\bigoplus_{m\in\mathbb{Z}} L^2_m(X,L^k).\]
We first construct a bounded operator on $L^2(X,L^k)$
by putting a weight on the components of the Fourier
decomposition with the help of a cut-off function.
Fix $\delta>0$ and a function
\begin{equation}\label{e-gue160105}
\tau_\delta\in C^\infty_0((-\delta,\delta)),\:\:
0\leq\tau_\delta\leq1, \:\:\text{$\tau_\delta=1$ on
$\left[-\frac{\delta}{2},\frac{\delta}{2}\:\right]$}.
\end{equation}
Let $F_{k,\delta}:L^2(X,L^k)\To L^2(X,L^k)$
be the bounded operator given by
\begin{equation}\label{e-gue150807I}
\begin{split}
F_{k,\delta}:L^2(X,L^k)&\To L^2(X,L^k),\\
u&\mapsto\sum_{m\in\mathbb Z}\tau_\delta\left(\frac{m}{k}\right)Q^{(0)}_{m,k}(u).
\end{split}
\end{equation}
For every $\lambda>0$, we consider the partial Szeg\H{o} projector
\begin{equation}\label{e-gue150806V}
\Pi_{k,\leq\lambda}:L^2(X,L^k)\To \mathcal{H}^0_{b,\leq\lambda}(X,L^k)
\end{equation}
which is the orthogonal projection on the space of equivariant CR functions
of degree less than $\lambda$. Finally, we consider the weighted Fourier-Szeg\H{o}
operator
\begin{equation}\label{e-gue150807II}
P_{k,\delta}:=F_{k,\delta}\circ\Pi_{k,\leq k\delta}\circ F_{k,\delta}:L^2(X,L^k)
\To \mathcal{H}^0_{b,\leq k\delta}(X,L^k).
\end{equation}
The Schwartz kernel of $P_{k,\delta}$ with respect to $dv_X$ is the smooth function
%$P_{k,\delta}(\cdot,\cdot)\in C^\infty(X\times X, L^k\boxtimes (L^k)^*)$
$P_{k,\delta}(x,y)\in L^k_x\otimes(L^k_y)^*$ satisfying
(cf.\ Section \ref{s-ssnI}, \cite[B.2]{MM07})
%-----------
\begin{equation}\label{sk}
(P_{k,\delta}u)(x)=\int_X P_{k,\delta}(x,y)u(y)\,dv_X(y)\,,\:\:u\in L^2(X,L^k).
\end{equation}
%-----------
Let $f_j=f_j^{(k)}$, $j=1,\ldots,d_k,$ be an orthonormal basis of
$\mathcal{H}^0_{b,\leq k\delta}(X,L^k)$.
Then
%-----------
\begin{equation}\label{sk1}
\begin{split}
P_{k,\delta}(x,y)&=\sum_{j=1}^{d_k} (F_{k,\delta}f_j)(x)\otimes
\big((F_{k,\delta}f_j)(y)\big)^*,\\
P_{k,\delta}(x,x)&=\sum_{j=1}^{d_k}\big| (F_{k,\delta}f_j)(x)\big|^2_{h^k},
\end{split}
\end{equation}
%-----------
(see Lemma \ref{L:sk}) and these representations are independent of the chosen
orthonormal basis. If should be noticed that the full Szeg\H{o} kernel
$\sum_{j=1}^{d_k}|f_j(x)|^2_{h^k}$ doesn't admit 
an asymptotic expansion in general, hence the necessity
of using the cut-off function $F_{k,\delta}$, see the discussion after Corollary~\ref{c-gue150811}. 
In order to describe the Fourier-Szeg\H{o} kernel $P_{k,\delta}(x,y)$ we will localize $P_{k,\delta}$
with respect to a local rigid CR frame $s$ of $L$ on an open set $D\subset X$.
We define the weight of the metric $h$ on $L$ with respect to $s$ to be the function
$\Phi\in C^\infty(D)$ satisfying $\abs{s}^2_{h}=e^{-2\Phi}$.
We have an isometry
\begin{equation}\label{e-gue150806VI}
U_{k,s}:L^2(D)\to L^2(D,L^k),\:\: u\longmapsto ue^{k\Phi}s^k,
\end{equation}
 with
inverse $U_{k,s}^{-1}:L^2(D,L^k)\to L^2(D)$, $\alpha\mapsto e^{-k\Phi}s^{-k}\alpha$.
The localization of $P_{k,\delta}$ with respect to the trivializing rigid CR section $s$ is given by
\begin{equation}\label{e-gue150806VII}
P_{k,\delta,s}:L^2_{\mathrm{comp}}(D)\To L^2(D),\:\:
P_{k,\delta,s}= U_{k,s}^{-1}P_{k,\delta}U_{k,s},
\end{equation}
where $L^2_{\mathrm{comp}}(D)$ is the subspace of elements of $L^2(D)$
with compact support in $D$. Let $P_{k,\delta,s}(x,y)\in C^\infty(D\times D)$
be the Schwartz kernel of $P_{k,\delta,s}$ with respect to $dv_X$, defined as in \eqref{sk}.
The first main result of this work describes the structure of the localized
Fourier-Szeg\H{o} kernel $P_{k,\delta,s}(x,y)$.
%-----------
\begin{thm}\label{t-gue150807}
Let $X$ be a compact CR manifold with a transversal CR locally free $S^1$-action and
let $L$ be a positive rigid CR line bundle on $X$.
Consider a point $p\in X$ and a canonical coordinate neighborhood
$(D,x=(x_1,\ldots,x_{2n-1}))$ centered at $p=0$.
Let $s$ be a local rigid CR frame of $L$ on $D$ and set $\abs{s}^2_{h}=e^{-2\Phi}$.
Fix $\delta>0$ small enough and $D_0\Subset D$. Then
\begin{equation}\label{sk2}
P_{k,\delta,s}(x,y)=\int_\R e^{ik\varphi(x,y,t)}g(x,y,t,k)dt+O(k^{-\infty})\:\:
\text{on $D_0\times D_0$},
\end{equation}
where $\varphi\in C^\infty( D\times D\times(-\delta,\delta))$ is a phase function such that
for some constant $c>0$ we have
\begin{equation}\label{e-gue150807b}
\begin{split}
%&\varphi(x,y,t)\in C^\infty( D\times D\times(-\delta,\delta)),\\
&d_x\varphi(x,y,t)|_{x=y}=-2{\rm Im\,}\ddbar_b\Phi(x)+t\omega_0(x),\ \
d_y\varphi(x,y,t)|_{x=y}=2{\rm Im\,}\ddbar_b\Phi(x)-t\omega_0(x),\\
&{\rm Im\,}\varphi(x,y,t)\geq c|z-w|^2,\ \ (x,y,t)\in D\times D\times(-\delta,\delta), x=(z, x_{2n-1}), y=(w, y_{2n-1}),\\
&\mbox{${\rm Im\,}\varphi(x,y,t)+\abs{\frac{\pr\varphi}{\pr t}(x,y,t)}^2\geq c\abs{x-y}^2$,
$(x,y,t)\in D\times D\times(-\delta,\delta)$},\\
&\mbox{$\varphi(x,y,t)=0$ and $\frac{\pr\varphi}{\pr t}(x,y,t)=0$ if and only if $x=y$},
\end{split}
\end{equation}
%(see Theorem~\ref{t-gue140121I} and Theorem~\ref{t-gue140121II},
%for more properties for the phase $\varphi$),
and $g(x,y,t,k)\in S^{n}_{{\rm loc\,}}(1;D\times D\times(-\delta,\delta))
\cap C^\infty_0(D\times D\times(-\delta,\delta))$ is a symbol with expansion
\begin{equation}\label{e-gue150807bI}
\begin{split}
%&g(x,y,t,k)\in S^{n}_{{\rm loc\,}}(1;D\times D\times(-\delta,\delta))
%\cap C^\infty_0(D\times D\times(-\delta,\delta)),\\
&g(x,y,t,k)\sim\sum^\infty_{j=0}g_j(x,y,t)k^{n-j}\text{ in }S^{n}_{{\rm loc\,}}
(1;D\times D\times(-\delta,\delta)), \\
%&g_j(x,y,t)\in C^\infty_0(D\times D\times(-\delta,\delta)),\ \ j=0,1,2,\ldots,
\end{split}\end{equation}
and for
$x\in D_0$ and $|t|<\delta$ we have
\begin{equation}\label{e-gue150807a}
g_0(x,x,t)=(2\pi)^{-n}\abs{\det\bigr(R^L_x-2t\mathcal{L}_x\bigr)}\abs{\tau_\delta(t)}^2.
\end{equation}
%Here $\omega_0\in C^\infty(X,T^*X)$ is the global real $1$-form of unit length orthogonal
%to $T^{*1,0}X\oplus T^{*0,1}X$, see \eqref{e-gue150808I},
%$\abs{\bigr(R^L_x-2t\mathcal{L}_x\bigr)}=\abs{\lambda_1(x,t)\cdots\lambda_{n-1}(x,t)}$,
%where $\lambda_j(x,t)$, $j=1,\ldots,n-1$, are the eigenvalues of the Hermitian quadratic
%form $R^L_x-2t\mathcal{L}_x$ with respect to $\langle\,\cdot\,|\,\cdot\,\rangle$,
%$R^L_x$ and $\mathcal{L}_x$ denote the curvature two form of $L$ and the Levi form
%of $X$ respectively (see Definition~\ref{d-gue150808g} and Definition~\ref{d-gue150808}).
\end{thm}
%-----------
We refer the reader to Section~\ref{s-ssnI} for the notations in semi-classical analysis used
in Theorem~\ref{t-gue150807}.
The determinant of a Hermitian quadratic form $\mathcal{R}_x$
on $T^{1,0}_xX$ is defined by $\det\mathcal{R}_x=\lambda_1\ldots \lambda_{n-1}$
where $\lambda_1,\ldots, \lambda_{n-1}$ are the eigenvalues of $\mathcal{R}_x$  with respect to
$\langle\,\cdot\,|\,\cdot\,\rangle$.

%With the notations in Theorem~\ref{t-gue150807}, since
%${\rm Im\,}\varphi(x,y,t)+\abs{\frac{\pr\varphi}{\pr t}(x,y,t)}\geq c\abs{x-y}^2$,
%where $c>0$ is a constant, we can integrate by parts with respect to $t$
%and conclude that the integral $\int e^{ik\varphi(x,y,t)}g(x,y,t,k)dt$ is $k$-negligible away $x=y$.

From Theorem~\ref{t-gue150807}, we deduce the asymptotics of
the kernel $P_{k,\delta}(x,y)$ on the diagonal.
Note that $P_{k,\delta}(x,x)=P_{k,\delta,s}(x,x)$.
%---------------
\begin{cor}\label{c-gue150811}
In the conditions of Theorem ~\ref{t-gue150807} we have as $k\to\infty$,
\begin{equation}\label{e-gue150811b}
P_{k,\delta}(x,x)\sim\sum\limits^\infty_{j=0}k^{n-j}b_j(x)~\text{in}~ S^n_{\rm loc}(1; X)
\end{equation}
where $b_j(x)\in C^\infty(X)$, $j=0,1,2,\ldots,$ and
\begin{equation}\label{e-gue150811c}
b_0(x)=(2\pi)^{-n}\int_\R \abs{\det\bigr(R^L_x-2t\mathcal{L}_x\bigr)}\abs{\tau_\delta(t)}^2dt,
\end{equation}
with $\tau_\delta(t)\in C^\infty_0(\Real)$ introduced in \eqref{e-gue160105}.
\end{cor}
%---------------

%\begin{cor}\label{c-gue150811}
%With the notations and assumptions above, let $\delta>0$ be a small constant and let
%\[f_1\in\mathcal{H}^0_{b,\leq k\delta}(X,L^k),\ldots,f_{d_k}\in\mathcal{H}^0_{b,\leq k\delta}(X,L^k)\]
%be an orthonormal basis for $\mathcal{H}^0_{b,\leq k\delta}(X,L^k)$
%with respect to $(\,\cdot\,|\,\cdot\,)$. Then,
%\begin{equation}\label{e-gue150811b}
%\begin{split}
%&\mbox{$\sum\limits^{d}_{j=1}\abs{(F_{k,\delta}f_j)(x)}^2_{h^k}
%\equiv\sum\limits^\infty_{j=0}k^{n-j}b_j(x)\mod O(k^{-\infty})$ on $X$},\\
%&b_j(x)\in C^\infty(X),\ \ j=0,1,2,\ldots,\\
%&b_0(x)=(2\pi)^{-n}\int \abs{\det\bigr(R^L_x-2t\mathcal{L}_x\bigr)}\abs{\tau_\delta(t)}^2dt,
%\end{split}\end{equation}
%where $\tau_\delta(t)\in C^\infty_0(\Real)$ was introduced in \eqref{e-gue160105}.
%\end{cor}
%Corollary~\ref{c-gue150811} will be proved in Section~\ref{s-gue150816}. 

We now give a simple example to show why we need the $L^2$ cut-off 
function $F_{k,\delta}$. Let $(L,h)$ be a positive holomorphic line bundle over 
a compact complex manifold $M$ of dimension $n-1$.  
Then $X:=M\times S^1$ is a Levi-flat CR manifold of dimension $2n-1$ with transversal CR $S^1$ 
action $e^{i\theta}$ and the pull-back of $(L,h)$ by the projection $M\times S^1\to X$
is a positive CR line bundle over $X$, denoted again $(L,h)$. 
For $k>0$, let $g_1^{(k)},\ldots,g_{r_k}^{(k)}$ be an orthonormal basis of  
the space $H^0(M,L^k)$ of global holomorphic sections with values in $L^k$. 
By the asymptotic expansion of the Bergman kernel of $L^k$
\cite{Cat:97,Zel98} (see also \cite{MM07,MM08}) we have for $x\in M$,
%The pioneer work of Zelditch~\cite{Zel98} tells us that 
\[\sum^{r_k}_{j=1}\big|g_j^{(k)}(x)\big|^2_{h^{k}}\sim 
k^{n-1}b_0(x)+k^{n-2}b_1(x)+\ldots,\;\;k\to\infty.\]
For each $m\in\mathbb Z$, $\big\{f_{j,m}^{(k)}(x,\theta):=
\frac{1}{\sqrt{2\pi}}g_j^{(k)}(x)e^{im\theta}\big\}_{j=1}^{r_k}$
%$f_j^{(k)}(x,\theta)=
%$\big\{\frac{1}{\sqrt{2\pi}}g_1(x)e^{im\theta},\ldots,
%\frac{1}{\sqrt{2\pi}}g_{r_k}(x)e^{im\theta}\big\}$ 
is an orthonormal basis 
of $H^0_{b,m}(X,L^k)$. Hence, 
$\big\{f_{j,m}^{(k)}(x,\theta);\, m\in\mathbb Z, \abs{m}\leq k\delta\big\}$
%
%$\big\{\frac{1}{\sqrt{2\pi}}g_1(x)e^{im\theta},\ldots,
%\frac{1}{\sqrt{2\pi}}g_{r_k}(x)e^{im\theta};\, m\in\mathbb Z, \abs{m}\leq k\delta\big\}$
is an orthonormal basis of the space $H^0_{b,\leq k\delta}(X,L^k)$, 
whose cardinal is denoted by $d_k$. Thus, the Szeg\H{o}
kernel of $H^0_{b,\leq k\delta}(X,L^k)$ is given by
\[
\sum\limits^{d_{k}}_{j=1,\abs{m}\leq k\delta}\big|f_{j,m}^{(k)}(x)\big|^2_{h^{k}}=
\frac{1}{\pi}[k\delta]\sum^{r_k}_{j=1}\big|g_j^{(k)}(x)\big|^2_{h^{k}},
\]
where $[k\delta]$ denotes  Gauss' symbol of $k\delta$. 
The difficulty comes from the fact that the function $\delta\mapsto[k\delta]$ does not admit an 
asymptotic expansion in $k$. To get asymptotic expansion, we consider 
\begin{equation}\label{e:exp}
\sum\limits^{d_{k}}_{j=1,\abs{m}\leq k\delta}\big|(F_{\delta,k}f_{j,m}^{(k)})(x)\big|^2_{h^{k}}=
\frac{1}{2\pi}\sum\limits_{m\in\mathbb Z}
\Big|\tau_\delta\Big(\frac{m}{k}\Big)\Big|^2\sum\limits^{r_k}_{j=1}\big|g_j^{(k)}(x)\big|^2_{h^{k}}.
\end{equation} 
 From Fourier analysis, we can check that 
\[\sum\limits_{m\in\mathbb Z}\abs{\tau_\delta\Big(\frac{m}{k}\Big)}^2\sim 
k\int_\R \abs{\tau_\delta(t)}^2dt+a_0+a_{-1}k^{-1}+\ldots,\:\: k\to\infty.\]
We have therefore an asymptotic expansion of \eqref{e:exp} in $k$.
This is the idea of introducing the $L^2$ cut-off function $F_{\delta,k}$.

We define now the Kodaira map. Consider an open set $D\subset X$ with 
\begin{equation}\label{e:eqc}
\bigcup_{-\pi\leq\theta\leq\pi}e^{i\theta}D\subset D\,,
\end{equation}
and let $s:D\to L$ be a local rigid CR trivializing section on $D$.
For any $u\in C^\infty(X,L^k)$ 
we write $u(x)=s^k(x)\otimes\widetilde u(x)$ on $D$, 
with $\widetilde u\in C^\infty(D)$.
Let $\{f_j\}_{j=1}^{d_k}$ be an orthonormal basis of 
$\mathcal H^0_{b, \leq k\delta}(X, L^k)$ 
with respect to $(\,\cdot\,|\,\cdot\,)$ such that 
$f_j\in\mathcal H^0_{b, m_j}(X, L^k)$ and
set $g_j=F_{k,\delta}f_j$, $1\leq j\leq d_k$.
The Kodaira map is defined on $D$ by
\begin{equation}\label{e-gue150807h0}
\begin{split}
\Phi_{k,\delta}:D&\longrightarrow\mathbb C\mathbb P^{d_k-1},\\
x&\longmapsto\big[F_{k,\delta}f_1,\ldots,F_{k,\delta}f_{d_k}\big]:=
\big[\widetilde g_1(x), \ldots, \widetilde g_{d_k}(x)\big],\;\;\text{for $x\in D$}.
\end{split}
\end{equation}
By the proof of \cite[Lemma 1.22]{HL15} there exist
an open cover of $X$ with sets $D$ satisfying \eqref{e:eqc}.
Thus we have a well-defined global map 
\begin{equation}\label{e-gue150807h0}
\Phi_{k,\delta}:X\longrightarrow\mathbb C\mathbb P^{d_k-1},\quad
x\longmapsto\big[F_{k,\delta}f_1,\ldots,F_{k,\delta}f_{d_k}\big].
\end{equation}
Since $g_j\in{\mathcal H}^0_{b, m_j}(X, L^k)$ we have 
$T\,\widetilde g_j=im_j\widetilde g_j$ hence
%$u(e^{i\theta}x)=s^k(e^{i\theta}x)\otimes\widetilde u(e^{i\theta}x)
%=s^k(e^{i\theta}x)\otimes e^{im\theta}\widetilde u(x), \forall~ x\in D.$
\[g_j(e^{i\theta}x)=s^k(e^{i\theta}x)\otimes\widetilde g_j(e^{i\theta}x)=
s^k(e^{i\theta}x)\otimes e^{im_j\theta}\widetilde g_j(x).\] Thus
\begin{equation}
\begin{split}
\Phi_{k, \delta}(e^{i\theta}x)&=[\widetilde g_1(e^{i\theta}x), \cdots, \widetilde g_{d_k}(e^{i\theta}x)]=
[e^{im_1\theta}\widetilde g_1(x), \cdots, e^{im_{d_k}\theta}\widetilde g_{d_k}(x)]\\
&=\big[e^{im_1\theta}\Phi^1_{k,\delta}(x),\ldots,
e^{im_{d_k}\theta}\Phi^{d_k}_{k,\delta}(x)\big]
\end{split}
\end{equation}
We are thus led to consider \emph{weighted diagonal} $S^1$-actions on 
$\mathbb C\mathbb P^N$, that is, actions for which there exists 
$(m_1,\ldots,m_{N},m_{N+1})\in\N_0^{N+1}$
such that for all $ \theta\in[0,2\pi)$,
\begin{equation}\label{e:equi}
e^{i\theta} [z_1,\ldots,z_{N+1}]=
\big[e^{im_1\theta}z_1,\ldots,e^{im_{N+1}\theta}z_{N+1}\big],\:\:
[z_1,\ldots,z_{N+1}]\in\mathbb C\mathbb P^N.
\end{equation}

\begin{thm}\label{t-gue150807I}
Let $(X, T^{1,0}X)$ be a compact CR manifold with a transversal CR locally free $S^1$-action.
Assume there is a rigid positive CR line bundle $L$ over $X$.
Then there exists $\delta_0>0$ such that for all $\delta\in(0,\delta_0)$ there exists
$k(\delta)$ so that for $k>k(\delta)$ and any orthonormal basis 
$\{f_j\}_{j=1}^{d_k}$ of 
$\mathcal H^0_{b, \leq k\delta}(X, L^k)$ 
with respect to $(\,\cdot\,|\,\cdot\,)$ such that 
$f_j\in\mathcal H^0_{b, m_j}(X, L^k)$,
the map $\Phi_{k,\delta}$ introduced in \eqref{e-gue150807h0}
is a smooth CR embedding which is $S^1$-equivariant with respect to
the weighted diagonal action defined by $(m_1,\ldots,m_{d_k})\in\N_0^{d_k}$ 
as in \eqref{e:equi}, that is, 
\[\Phi_{k,\delta}(e^{i\theta}  x)=
e^{i\theta}\Phi_{k,\delta}(x),\:\:x\in X,\;\theta\in[0,2\pi).\]
%and there exist a weighted diagonal $S^1$-action on
%$\mathbb C\mathbb P^{d_k-1}$ such that $\Phi_{k,\delta}$ is $S^1$-equivariant, that is,
%there exists  $(m_1,\ldots,m_{d_k})\in\N_0^{d_k}$ such that
%\[\Phi_{k,\delta}(e^{i\theta}  x)=
%\big[e^{im_1\theta}\Phi^1_{k,\delta}(x),\ldots,
%e^{im_{d_k}\theta}\Phi^{d_k}_{k,\delta}(x)\big]=:e^{i\theta}\Phi_{k,\delta}(x):,\:\:
%x\in X,\;\theta\in[0,2\pi).\]
In particular, the image $\Phi_{k,\delta}(X)\subset\mathbb C\mathbb P^{d_k-1}$ is a CR submanifold
with an induced
weighted diagonal locally free  $S^1$-action.
\end{thm}
In~\cite{Hsiao14}, it was proved that if $X$ admits a transversal CR locally free
$S^1$-action and there is a rigid positive CR line bundle $L$ over $X$,
then $X$ can be CR embedded into projective space under the assumption that condition
$Y(0)$ holds on $X$. In Theorem \ref{t-gue150807I} we remove the Levi curvature assumption
$Y(0)$ used in~\cite{Hsiao14}. Moreover, the embedding theorems 
established in~\cite{Hsiao14} are not $S^1$-equivariant.

As a consequence of Theorem \ref{t-gue150807I} we obtain an embedding result
for Levi-flat CR manifolds.
\begin{cor}\label{C:Lf}
Let $X$ be a compact Levi-flat CR manifold. Assume that $X$ admits a
transversal CR locally free $S^1$-action and a positive rigid CR line bundle.
Then there exists $\delta_0>0$ such that for all $\delta\in(0,\delta_0)$ there exists
$k(\delta)$ so that for $k>k(\delta)$ the map $\Phi_{k,\delta}$ introduced in \eqref{e-gue150807h}
is a $C^\infty$ CR embedding of $X$ in $\mathbb C\mathbb P^{d_k-1}$
which is $S^1$-equivariant with respect to weighted diagonal actions.
%Then as in Theorem \ref{t-gue150807I} there exists $C^\infty$ CR 
%embeddings $\Phi_{k,\delta}$ of $X$ in $\mathbb C\mathbb P^{d_k-1}$
%which are $S^1$-equivariant with respect to weighted diagonal actions.
\end{cor}
Ohsawa and Sibony~\cite{Oh12,OS00} %, cf.\ also \cite{Oh12},
constructed for every $\kappa\in\mathbb N$ a CR projective embedding of class $C^\kappa$
of a Levi-flat CR manifold by using $\ol\partial$-estimates.
The first and third authors~\cite{HM15} gave a Szeg\H{o} kernel proof of Ohsawa and Sibony's result.
A natural question is whether we can improve the regularity to $\kappa=\infty$.
Adachi~\cite{Ad13} showed that the answer is no, in general.
The analytic difficulty of this problem comes from the fact that
the Kohn Laplacian is not hypoelliptic on Levi flat manifolds.
Corollary \ref{C:Lf} shows that one can find $C^\infty$ CR embeddings
of Levi flat manifolds in the equivariant setting.

When $X$ is strongly pseudoconvex, it is known \cite{OV07} 
that there is a rigid positive CR line bundle over $X$.
We deduce from Theorem~\ref{t-gue150807I}:
\begin{cor}\label{t-gue151230I}
Let $(X, T^{1,0}X)$ be a compact strongly pseudoconvex CR manifold with
a transversal CR locally free $S^1$-action.
Then there exists smooth CR embeddings $\Phi_{k,\delta}$ of $X$ in $\mathbb C\mathbb P^{d_k-1}$
which are $S^1$-equivariant with respect to weighted diagonal actions (cf.\ Theorem \ref{t-gue150807I}).
%Then we can find a CR embedding $\Phi:x\in X\To(\Phi_1(x),\ldots,\Phi_{N}(x))\in\mathbb C^N$,
%for some $N\in\mathbb N$ such that $\Phi(X)$ is a CR submanifold of $\mathbb C^N$
%with a transversal CR locally free simple $S^1$-action $e^{i\theta}$ and we have
%\[\Phi(e^{i\theta}  x)=e^{i\theta} \Phi(x)=
%(e^{im_1\theta}\Phi_1(x),\ldots,e^{im_{N}\theta}\Phi_{N}(x)),\] $\forall x\in X$, $\forall\theta\in[0,2\pi[$,
%where $(m_1,\ldots,m_N)\in\N_0^N$.
\end{cor}
We illustrate Corollary \ref{t-gue150807I} in Example \ref{ex_t-gue150807I}.

This paper is organized as follows. In Section \ref{s:prelim} we recall
the necessary notions and results from semiclassical analysis and theory
of CR manifolds with circle action.
In Section \ref{e-gue150811} we prove the asymptotics of the Fourier-Szeg\H{o}
kernel (Theorem \ref{t-gue150807} and Corollary \ref{c-gue150811}).
Section \ref{s-gue150819} deals with the Kodaira embedding theorem.

\section{Preliminaries}\label{s:prelim}

\subsection{Some standard notations}\label{s-gue150508b}
We use the following notations: $\mathbb N=\set{1,2,\ldots}$,
$\mathbb N_0=\mathbb N\cup\set{0}$, $\Real$
is the set of real numbers, $\ol\Real_+:=\set{x\in\Real;\, x\geq0}$.
For a multiindex $\alpha=(\alpha_1,\ldots,\alpha_m)\in\mathbb N_0^m$
we set $\abs{\alpha}=\alpha_1+\ldots+\alpha_m$. For $x=(x_1,\ldots,x_m)\in\Real^m$ we write
\[
\begin{split}
&x^\alpha=x_1^{\alpha_1}\ldots x^{\alpha_m}_m,\quad
 \pr_{x_j}=\frac{\pr}{\pr x_j}\,,\quad
\pr^\alpha_x=\pr^{\alpha_1}_{x_1}\ldots\pr^{\alpha_m}_{x_m}=\frac{\pr^{\abs{\alpha}}}{\pr x^\alpha}\,,\\
&D_{x_j}=\frac{1}{i}\pr_{x_j}\,,\quad D^\alpha_x=D^{\alpha_1}_{x_1}\ldots D^{\alpha_m}_{x_m}\,,
\quad D_x=\frac{1}{i}\pr_x\,.
\end{split}
\]
Let $z=(z_1,\ldots,z_m)$, $z_j=x_{2j-1}+ix_{2j}$, $j=1,\ldots,m$, be coordinates of $\Complex^m$,
where
$x=(x_1,\ldots,x_{2m})\in\Real^{2m}$ are coordinates in $\Real^{2m}$.
Throughout the paper we also use the notation
$w=(w_1,\ldots,w_m)\in\Complex^m$, $w_j=y_{2j-1}+iy_{2j}$, $j=1,\ldots,m$, where
$y=(y_1,\ldots,y_{2m})\in\Real^{2m}$.
We write
\[
\begin{split}
&z^\alpha=z_1^{\alpha_1}\ldots z^{\alpha_m}_m\,,\quad\ol z^\alpha=\ol z_1^{\alpha_1}\ldots\ol z^{\alpha_m}_m\,,\\
&\pr_{z_j}=\frac{\pr}{\pr z_j}=
\frac{1}{2}\Big(\frac{\pr}{\pr x_{2j-1}}-i\frac{\pr}{\pr x_{2j}}\Big)\,,\quad\pr_{\ol z_j}=
\frac{\pr}{\pr\ol z_j}=\frac{1}{2}\Big(\frac{\pr}{\pr x_{2j-1}}+i\frac{\pr}{\pr x_{2j}}\Big),\\
&\pr^\alpha_z=\pr^{\alpha_1}_{z_1}\ldots\pr^{\alpha_m}_{z_m}=\frac{\pr^{\abs{\alpha}}}{\pr z^\alpha}\,,\quad
\pr^\alpha_{\ol z}=\pr^{\alpha_1}_{\ol z_1}\ldots\pr^{\alpha_m}_{\ol z_m}=
\frac{\pr^{\abs{\alpha}}}{\pr\ol z^\alpha}\,.
\end{split}
\]

Let $X$ be a $C^\infty$ orientable paracompact manifold.
We let $TX$ and $T^*X$ denote the tangent bundle of $X$ and the cotangent bundle of $X$ respectively.
The complexified tangent bundle of $X$ and the complexified cotangent bundle of $X$
will be denoted by $\Complex TX$ and $\Complex T^*X$ respectively. We write $\langle\,\cdot\,,\cdot\,\rangle$
to denote the pointwise duality between $TX$ and $T^*X$.
We extend $\langle\,\cdot\,,\cdot\,\rangle$ bilinearly to $\Complex TX\times\Complex T^*X$.

Let $E$ be a $C^\infty$ vector bundle over $X$. The fiber of $E$ at $x\in X$ will be denoted by $E_x$.
Let $F$ be another vector bundle over $X$. We write
$F\boxtimes E^*$ to denote the vector bundle over $X\times X$ with fiber over $(x, y)\in X\times X$
consisting of the linear maps from $E_x$ to $F_y$.

Let $Y\subset X$ be an open set. The spaces of
smooth sections of $E$ over $Y$ and distribution sections of $E$ over $Y$ will be denoted by $C^\infty(Y, E)$ and $\mathscr D'(Y, E)$ respectively.
Let $\mathscr E'(Y, E)$ be the subspace of $\mathscr D'(Y, E)$ whose elements have compact support in $Y$.
For $m\in\Real$, we let $H^m(Y, E)$ denote the Sobolev space
of order $m$ of sections of $E$ over $Y$. Put
\begin{gather*}
H^m_{\rm loc\,}(Y, E)=\big\{u\in\mathscr D'(Y, E);\, \varphi u\in H^m(Y, E),
      \,\forall\varphi\in C^\infty_0(Y)\big\}\,,\\
      H^m_{\rm comp\,}(Y, E)=H^m_{\rm loc}(Y, E)\cap\mathscr E'(Y, E)\,.
\end{gather*}

\subsection{Definitions and notations from semi-classical analysis} \label{s-ssnI}

We recall the Schwartz kernel theorem \cite[Theorems 5.2.1, 5.2.6]{Hor03}, \cite[p. 296]{Tay1:96},
\cite[B.2]{MM07}.
Let $E$ and $F$ be smooth vector bundles over $X$. Let $Y$ be an open set of $X$.
Let $A(\cdot,\cdot)\in \mathscr{D}'(Y\times Y,F\boxtimes E^*)$. For any fixed
$u\in C^\infty_0 (Y,E)$, the linear map
$C^\infty_0 (Y,F^*)\ni v\mapsto  (A(\cdot,\cdot),v\otimes u)\in\mathbb{C}$ defines a distribution $Au\in \mathscr D'(Y,F)$.
The operator $A:C^\infty_0(Y, E)\To \mathscr D'(Y,F)$, $u\mapsto Au$, is linear and continuous.

The Schwartz kernel theorem asserts that, conversely, for any continuous linear operator
$A:C^\infty_0(Y, E)\To \mathscr D'(Y,F)$ there exists a unique distribution
$A(\cdot,\cdot)\in \mathscr{D}'(Y\times Y,F\boxtimes E^*)$ such that
$(Au,v)=(A(\cdot,\cdot),v\otimes u)$ for any $u\in C^\infty_0(Y,E)$, $v\in C^\infty_0(Y,F^*)$.
The distribution $A(\cdot,\cdot)$ is
called the Schwartz distribution kernel of $A$.
We say that $A$ is properly supported if the canonical projections on the two factors restricted to
${\rm Supp\,}A(\cdot,\cdot)\subset Y\times Y$ are proper.
%i.\,e., the pjections $t_x:(x,y)\in{\rm Supp\,}A(\cdot,\cdot)\To x\in M$,
%$t_y:(x,y)\in{\rm Supp\,}A(\cdot,\cdot)\To y\in M$ are proper.

The following two statements are equivalent:
\begin{enumerate}
\item $A$ can be extended to a continuous operator $A:\mathscr E'(Y, E)\To C^\infty(Y, F)$,
\item $A(\cdot,\cdot)\in C^\infty(Y\times Y, F\boxtimes E^*)$.
\end{enumerate}
If $A$ satisfies (a) or (b), we say that $A$ is a \emph{smoothing operator}.
Furthermore, $A$ is smoothing if and only if for all $N\geq0$ and $s\in\Real$,
$A: H^s_{\rm comp\,}(Y, E)\To H^{s+N}_{\rm loc\,}(Y, F)$
is continuous.

Let $A$ be a smoothing operator.
Then for any volume form $d\mu$,  the Schwartz kernel of $A$ is represented by a smooth kernel
$K\in C^\infty(Y\times Y,F\boxtimes E^*)$, called the Schwartz kernel of $A$ with respect to $d\mu$,
such that
\begin{gather}
(Au)(x)=\int_M K(x,y)u(y)\,d\mu(y)\,,
\quad \text{for any $u\in C^\infty_0(Y,E)$}\,.
\end{gather}
Then $A$ can be extended as a linear continuous operator
$A:\mathscr E'(Y,E)\to C^\infty(Y,F)$ by setting
$(Au)(x)=\big(u(\cdot),K(x,\cdot)\big)$, $x\in Y$,
for any $u\in\mathscr E^{\prime}(Y,E)$.

Let $W_1$, $W_2$ be open sets in $\Real^N$ and let $E$ and $F$ be complex
Hermitian vector bundles over $W_1$ and $W_2$ respectively.
Let $s, s'\in\Real$ and $n_0\in\mathbb\Real$.
For a $k$-dependent continuous function $F_k:H^s_{{\rm comp\,}}(W_1,E)\To H^{s'}_{{\rm loc\,}}(W_2,F)$
we write
\[F_k=O(k^{n_0}):H^s_{{\rm comp\,}}(W_1,E)\To H^{s'}_{{\rm loc\,}}(W_2,F),\]
if for any $\chi_0\in C^\infty(W_2), \chi_1\in C^\infty_0(W_1)$, there is a positive constant $c>0$ independent of $k$, such that
\begin{equation} \label{e-gue13628II}
\norm{(\chi_0F_k\chi_1)u}_{s'}\leq ck^{n_0}\norm{u}_{s},\ \ \forall u\in H^s_{{\rm loc\,}}(W_1,E),
\end{equation}
where $\norm{\cdot}_s$ denotes the usual Sobolev norm of order $s$. We write
\[F_k=O(k^{-\infty}):H^s_{{\rm comp\,}}(W_1,E)\To H^{s'}_{{\rm loc\,}}(W_2,F),\]
if $F_k=O(k^{-N}):H^s_{{\rm comp\,}}(W_1,E)\To H^{s'}_{{\rm loc\,}}(W_2,F)$, for every $N>0$.
%Similarly, let $\ell\in\mathbb N$, for a $k$-dependent continuous function $G_k: C^\infty_0(W_1,E)\To C^\ell(W_2,F)$
%we write
%\[G_k=O(k^{-\infty}):C^\infty_0(W_1,E)\To C^\ell(W_2,F),\]
%if for any $\chi_0\in C^\infty(W_2), \chi_1\in C^\infty_0(W_1)$ and $N>0$, there are positive constants $c>0$ and $M\in\mathbb N_0$ independent of $k$, such that
%\begin{equation} \label{e-gue13628IIbis}
%\norm{(\chi_0G_k\chi_1)u}_{C^\ell(W_2,F)}\leq ck^{-N}\norm{u}_{C^M(W_1,E)},\ \ \forall u\in C^\infty_0(W_1,E),
%\end{equation}

A $k$-dependent continuous operator
$A_k:C^\infty_0(W_1,E)\To\mathscr D'(W_2,F)$ is called $k$-negligible on $W_2\times W_1$
if for $k$ large enough $A_k$ is smoothing  and for any $K\Subset W_2\times W_1$, any
multi-indices $\alpha$, $\beta$ and any $N\in\mathbb N$ there exists $C_{K,\alpha,\beta,N}>0$
such that
\begin{equation}\label{e-gue13628III}
\abs{\pr^\alpha_x\pr^\beta_yA_k(x, y)}\leq C_{K,\alpha,\beta,N}k^{-N}\:\: \text{on $K$}.
\end{equation}
We write in this case
\[A_k(x,y)=O(k^{-\infty})\:\:\text{on $W_2\times W_1$,}\]
or
\[A_k=O(k^{-\infty})\:\:\text{on $W_2\times W_1$.}\]
If $A_k, B_k:C^\infty_0(W_1,E)\To\mathscr D'(W_2,F)$ are $k$-dependent continuous operators,
we write $A_k= B_k+O(k^{-\infty})$ if $A_k-B_k=O(k^{-\infty})$ on $W_2\times W_1$.

%Let $C_k:C^\infty_0(W_1,E)\To\mathscr D'(W_2,F)$
%be another $k$-dependent continuous operator.
%We write $A_k\equiv C_k\mod O(k^{-\infty})$ (on $W_2\times W_1$) or
%$A_k(x,y)\equiv C_k(x,y)\mod O(k^{-\infty})$ (on $W_2\times W_1$)
%if $A_k-C_k$ is $k$-negligible on $W_2\times W_1$.

Let $A_k:L^2(X,L^k)\To L^2(X,L^k)$ be a continuous operator.
Let $s$, $s_1$ be local rigid CR frames of $L$ on open sets $D_0\Subset M$, $D_1\Subset M$
respectively, $\abs{s}^2_{h}=e^{-2\Phi}$, $\abs{s_1}^2_{h}=e^{-2\Phi_1}$.
The localization of $A_k$ (with respect to the trivializing rigid CR sections $s$ and $s_1$)  is given by
\begin{equation} \label{e-gue140824II}
\begin{split}
A_{k,s,s_1}:L^2(D_1)\cap \mathscr E'(D_1)\To L^2(D),\:\:
u\longmapsto e^{-k\Phi}s^{-k}A_k(s^k_1e^{k\Phi_1}u)=U^{-1}_{k,s}A_kU_{k,s_1},
\end{split}
\end{equation}
and let $A_{k,s,s_1}(x,y)\in\mathscr D'(D\times D_1)$ be the distribution kernel of $A_{k,s,s_1}$.  Let $\sigma, \sigma', n_0\in\Real$. We write
\[A_k=O(k^{n_0}):H^{\sigma}(X,L^k)\To H^{\sigma'}(X,L^k),\]
if for all local rigid CR frames $s, s_1$ on $D$ and $D_1$ respectively, we have
\[A_{k,s,s_1}=O(k^{n_0}):H^{\sigma}_{{\rm comp\,}}(D_1)\To H^{\sigma'}_{{\rm loc\,}}(D).\]
We write
\[A_k=O(k^{-\infty}):H^\sigma(X,L^k)\To H^{\sigma'}(X,L^k),\]
if for all local rigid CR frames $s, s_1$ on $D$ and $D_1$ respectively, we have
\[A_{k,s,s_1}=O(k^{-\infty}):H^\sigma_{{\rm comp\,}}(D_1)\To H^{\sigma'}_{{\rm loc\,}}(D).\]
We write
\[A_k=O(k^{-\infty})\]
if for all local rigid CR frames $s, s_1$ on $D$ and $D_1$ respectively, we have
\[A_{k,s,s_1}(x,y)= O(k^{-\infty})\:\: \text{on $D\times D_1$}.\]
When $s=s_1$, $D=D_1$, we write $A_{k,s}:=A_{k,s,s}$, $A_{k,s}(x,y):=A_{k,s,s}(x,y)$.

We recall the definition of the semi-classical symbol spaces \cite[Chapter 8]{DiSj99}:
\begin{defn} \label{d-gue140826}
Let $W$ be an open set in $\Real^N$. Let
%$S(1;W)=S(1)$ be the set of
%$a\in \cC^\infty(W)$ such that for every $\alpha\in\mathbb N^N_0$, there
%exists $C_\alpha>0$, such that $\abs{\pr^\alpha_xa(x)}\leq
%C_\alpha$ on $W$.
\begin{gather*}
S(1;W):=\Big\{a\in C^\infty(W)\,|\, \forall\alpha\in\mathbb N^N_0:
\sup_{x\in W}\abs{\pr^\alpha a(x)}<\infty\Big\},\\
S^0_{{\rm loc\,}}(1;W):=\Big\{(a(\cdot,k))_{k\in\Real}\,|\, \forall\alpha\in\mathbb N^N_0,
\forall \chi\in C^\infty_0(W)\,:\:\sup_{k\in\Real, k\geq1}\sup_{x\in W}\abs{\pr^\alpha(\chi a(x,k))}<\infty\Big\}\,.
\end{gather*}
%The space $S_{{\rm loc\,}}(1;W)$ is the set of sequences $a=(a(\cdot,k))$ in $\cC^\infty(W)$
%with the property that for any $\chi\in\cC^\infty_0(W)$ we have
%%\[
%$\sup_{k\in\N}\sup_{x\in W}\abs{\pr^\alpha a(x,k)}<\infty$\,.
%%\]
%If $a=a(x,k)$ depends on $k\in]1,\infty[$, we say that
%$a(x,k)\in S_{{\rm loc\,}}(1;W)=S_{{\rm loc\,}}(1)$ if $\chi(x)a(x,k)$ is uniformly bounded
%in $S(1)$ when $k$ varies in $]1,\infty[$, for any $\chi\in\cC^\infty_0(W)$.
For $m\in\Real$ let
\[
S^m_{{\rm loc}}(1):=S^m_{{\rm loc}}(1;W)=\Big\{(a(\cdot,k))_{k\in\Real}\,|\,(k^{-m}a(\cdot,k))\in S^0_{{\rm loc\,}}(1;W)\Big\}\,.
\]
%For $m\in\Real$, we put $S^m_{{\rm loc}}(1;W)=\{(a(\cdot,k)):(k^{-m}a(\cdot,k))\in S_{{\rm loc\,}}(1;W)\}$.
%For $m\in\Real$, we put $S^m_{{\rm loc}}(1;W)=S^m_{{\rm loc}}(1)=k^mS_{{\rm loc\,}}(1)$.
Hence $a(\cdot,k)\in S^m_{{\rm loc}}(1;W)$ if for every $\alpha\in\mathbb N^N_0$ and $\chi\in C^\infty_0(W)$, there
exists $C_\alpha>0$ independent of $k$, such that $\abs{\pr^\alpha (\chi a(\cdot,k))}\leq C_\alpha k^{m}$ on $W$.

Consider a sequence $a_j\in S^{m_j}_{{\rm loc\,}}(1)$, $j\in\N_0$, where $m_j\searrow-\infty$,
and let $a\in S^{m_0}_{{\rm loc\,}}(1)$. We say that
\[
a(\cdot,k)\sim
\sum\limits^\infty_{j=0}a_j(\cdot,k)\:\:\text{in $S^{m_0}_{{\rm loc\,}}(1)$},
\]
if for every
$\ell\in\N_0$ we have $a-\sum^{\ell}_{j=0}a_j\in S^{m_{\ell+1}}_{{\rm loc\,}}(1)$.
For a given sequence $a_j$ as above, we can always find such an asymptotic sum
$a$, which is unique up to an element in
$S^{-\infty}_{{\rm loc\,}}(1)=S^{-\infty}_{{\rm loc\,}}(1;W):=\cap _mS^m_{{\rm loc\,}}(1)$.

We say that $a(\cdot,k)\in S^{m}_{{\rm loc\,}}(1)$ is a classical symbol on $W$ of order $m$ if
\begin{equation} \label{e-gue13628I}
a(\cdot,k)\sim\sum\limits^\infty_{j=0}k^{m-j}a_j\: \text{in $S^{m}_{{\rm loc\,}}(1)$},\ \ a_j(x)\in
S^0_{{\rm loc\,}}(1),\ j=0,1\ldots.
\end{equation}
The set of all classical symbols on $W$ of order $m$ is denoted by
$S^{m}_{{\rm loc\,},{\rm cl\,}}(1)=S^{m}_{{\rm loc\,},{\rm cl\,}}(1;W)$.
\end{defn}
%--------------
\begin{defn} \label{d-gue13628I}
Let $W$ be an open set in $\Real^N$. A semi-classical pseudodifferential
operator on $W$ of order $m$ with classical symbol is a $k$-dependent continuous operator
$A_k:C^\infty_0(W)\To C^\infty(W)$ such that the distribution kernel $A_k(x,y)$
is given by the oscillatory integral
\begin{equation}\label{psk-def}
\begin{split}
A_k(x,y)=&\frac{k^N}{(2\pi)^N}\int e^{ik\langle x-y,\eta\rangle}a(x,y,\eta,k)d\eta
+O(k^{-\infty}),\\
&a(x,y,\eta,k)\in S^m_{{\rm loc\,},{\rm cl\,}}(1;W\times W\times\Real^N).\end{split}
\end{equation}
We shall identify $A_k$ with $A_k(x,y)$. It is clear that $A_k$
has a unique continuous extension $A_k:\mathscr E'(W)\To\mathscr D'(W)$. We have
\begin{equation}\label{ps-def}
A_k(x,y)=\frac{k^N}{(2\pi)^N}\int e^{ik\langle x-y,\eta\rangle}\alpha(x,\eta,k)d\eta
+O(k^{-\infty})
\end{equation}
with symbol
\begin{equation}\label{ps-symb}
\alpha(x,\eta,k)\in S^m_{{\rm loc\,},{\rm cl\,}}(1;W\times\Real^N)=
S^m_{{\rm loc\,},{\rm cl\,}}(1;T^*W).
\end{equation}
\end{defn}

\subsection{CR manifolds with circle action}\label{s-gue150808}

Let $(X, T^{1,0}X)$ be a compact CR manifold of dimension $2n-1$, $n\geq 2$, where $T^{1,0}X$ is a CR structure of $X$. That is $T^{1,0}X$ is a subbundle of rank $n-1$ of the complexified tangent bundle $\mathbb{C}TX$, satisfying $T^{1,0}X\cap T^{0,1}X=\{0\}$, where $T^{0,1}X=\overline{T^{1,0}X}$, and $[\mathcal V,\mathcal V]\subset\mathcal V$, where $\mathcal V=C^\infty(X, T^{1,0}X)$. We assume that $X$ admits a $S^1$-action: $S^1\times X\rightarrow X$. We write $e^{i\theta}$ to denote the $S^1$-action. Let $T\in C^\infty(X, TX)$ be the global real vector field induced by the $S^1$-action given by
\begin{equation}\label{e-gue150808}
(Tu)(x)=\frac{\partial}{\partial\theta}\left(u(e^{i\theta}x)\right)|_{\theta=0},\ \ u\in C^\infty(X).
\end{equation}

\begin{defn}
We say that the $S^1$-action $e^{i\theta}$ is CR if
$[T, C^\infty(X, T^{1,0}X)]\subset C^\infty(X, T^{1,0}X)$ and the $S^1$-action is transversal if for each $x\in X$,
$\Complex T(x)\oplus T_x^{1,0}(X)\oplus T_x^{0,1}X=\mathbb CT_xX$. Moreover, we say that the $S^1$-action is locally free if $T\neq0$ everywhere.
\end{defn}

%We assume throughout that $(X, T^{1,0}X)$ is a CR manifold with a transversal CR locally free $S^1$-action $e^{i\theta}$ and we let $T$ be the global vector field induced by the $S^1$-action. Let $\omega_0\in C^\infty(X,T^*X)$ be the global real one form determined by
%\begin{equation}\label{e-gue150808I}
%\begin{split}
%&\langle\,\omega_0\,,\,u\,\rangle=0,\ \ \forall u\in T^{1,0}X\oplus T^{0,1}X,\\
%&\langle\,\omega_0\,,\,T\,\rangle=-1.
%\end{split}
%\end{equation}
%
%\begin{defn}\label{d-gue150808}
%For $p\in X$, the Levi form $\mathcal L_p$ is the Hermitian quadratic form on $T^{1,0}_pX$ given by
%$\mathcal{L}_p(U,\ol V)=-\frac{1}{2i}\langle\,d\omega_0(p)\,,\,U\wedge\ol V\,\rangle$, $U, V\in T^{1,0}_pX$.
%\end{defn}

Denote by $T^{*1,0}X$ and $T^{*0,1}X$ the dual bundles of
$T^{1,0}X$ and $T^{0,1}X$ respectively. Define the vector bundle of $(0,q)$ forms by
$T^{*0,q}X=\Lambda^q(T^{*0,1}X)$.
Let $D\subset X$ be an open subset. Let $\Omega^{0,q}(D)$
denote the space of smooth sections of $T^{*0,q}X$ over $D$ and let $\Omega_0^{0,q}(D)$
be the subspace of $\Omega^{0,q}(D)$ whose elements have compact support in $D$. Similarly, if $E$ is a vector bundle over $D$, then we let $\Omega^{0,q}(D, E)$ denote the space of smooth sections of $T^{*0,q}X\otimes E$ over $D$ and let $\Omega_0^{0,q}(D, E)$ be the subspace of $\Omega^{0,q}(D, E)$ whose elements have compact support in $D$.

Fix $\theta_0\in]-\pi, \pi[$, $\theta_0$ small. Let
$$d e^{i\theta_0}: \mathbb CT_x X\rightarrow \mathbb CT_{e^{i\theta_0}x}X$$
denote the differential map of $e^{i\theta_0}: X\rightarrow X$. Since the $S^1$-action is CR, we can check that
\begin{equation}\label{e-gue150508fa}
\begin{split}
de^{i\theta_0}:T_x^{1,0}X\rightarrow T^{1,0}_{e^{i\theta_0}x}X,\\
de^{i\theta_0}:T_x^{0,1}X\rightarrow T^{0,1}_{e^{i\theta_0}x}X,\\
de^{i\theta_0}(T(x))=T(e^{i\theta_0}x).
\end{split}
\end{equation}
Let $(e^{i\theta_0})^*:\Lambda^r(\Complex T^*X)\To\Lambda^r(\Complex T^*X)$ be the pull-back map of $e^{i\theta_0}$, $r=0,1,\ldots,2n-1$. From \eqref{e-gue150508fa}, it is easy to see that for every $q=0,1,\ldots,n$,
\begin{equation}\label{e-gue150508faI}
(e^{i\theta_0})^*:T^{*0,q}_{e^{i\theta_0}x}X\To T^{*0,q}_{x}X.
\end{equation}
Let $u\in\Omega^{0,q}(X)$. Define
\begin{equation}\label{e-gue150508faII}
Tu:=\frac{\pr}{\pr\theta}\bigr((e^{i\theta})^*u\bigr)|_{\theta=0}\in\Omega^{0,q}(X).
\end{equation}
(See also \eqref{lI}.) For every $\theta\in[0, 2\pi)$ and every $u\in C^\infty(X,\Lambda^r(\Complex T^*X))$, we write $u(e^{i\theta}  x):=(e^{i\theta})^*u(x)$. It is clear that for every $u\in C^\infty(X,\Lambda^r(\Complex T^*X))$, we have
\begin{equation}\label{e-gue150510f}
u(x)=\sum_{m\in\mathbb Z}\frac{1}{2\pi}\int^{\pi}_{-\pi}u(e^{i\theta}  x)e^{-im\theta}d\theta.
\end{equation}

Let $\ddbar_b:\Omega^{0,q}(X)\rightarrow\Omega^{0,q+1}(X)$ be the tangential Cauchy-Riemann operator.
Since the $S^1$-action is CR, it is straightforward to see that (see also \eqref{e-gue150514f})
\[T\ddbar_b=\ddbar_bT\ \ \mbox{on $\Omega^{0,q}(X)$}.\]

\begin{defn}\label{d-gue50508d}
Let $D\subset U$ be an open set. We say that a function $u\in C^\infty(D)$ is rigid if $Tu=0$. We say that a function $u\in C^\infty(X)$ is Cauchy-Riemann (CR for short)
if $\ddbar_bu=0$. We say that $u\in C^\infty(X)$ is rigid CR if  $\ddbar_bu=0$ and $Tu=0$.
\end{defn}

\begin{defn} \label{d-gue150508dI}
Let $F$ be a complex vector bundle over $X$. We say that $F$ is rigid (resp.\ CR, resp.\ rigid CR) if there exists
an open cover $(U_j)_j$ of $X$ and trivializing frames $\set{f^1_j,f^2_j,\dots,f^r_j}$ on $U_j$,
such that the corresponding transition matrices are rigid (resp.\ CR, resp.\ rigid CR). The frames $\set{f^1_j,f^2_j,\dots,f^r_j}$ are called rigid (resp.\ CR, resp.\ rigid CR) frames.
\end{defn}
If $F$ is a rigid vector bundle, we can define the operator $T$ on $\Omega^{0,q}(X,F)$.
Indeed, every $u\in\Omega^{0,q}(X,F)$ can be written on $U_j$ as
$u=\sum u_\ell\otimes f^\ell_j$ and we set $Tu=\sum Tu_\ell\otimes f^\ell_j$.
Then $Tu$ is well defined as element of $\Omega^{0,q}(X,F)$,
since the entries of the transition matrices between different frames $\set{f^1_j,f^2_j,\dots,f^r_j}$
are annihilated by $T$.

\begin{ex}
Let $X$ be a compact CR manifold  with a locally free transversal CR $S^1$ action. Let $\{Z_j\}_j$ be a trivializing frame of $T^{1, 0}X$ defined in (\ref{e-can2}) in the BRT trivialization. It is easy to check that the transition functions of such frames are rigid CR  and thus $T^{1, 0}X$ is a rigid CR vector bundle. Moreover, $\det{T^{1, 0}X}$ the determinant bundle of $T^{1, 0}X$ is a rigid CR line bundle.
\end{ex}

\begin{ex}
Let $(L, h)\overset{\pi}{\rightarrow}M$ be a Hermitian line bundle over a complex manifold $M$. Consider the circle bundle $X=\{v\in L: h(v)=1\}$ over $M$. Then $X$ is a compact CR manifold with a globally free transversal CR $S^1$ action. Let $E$ be a holomorphic vector bundle over $M$. Then the restriction of the pull back $\pi^\ast E|_X$ on $X$ is a rigid CR vector bundle over $X$.
\end{ex}

\medskip

From now on, let $L$ be a rigid CR line bundle over $X$.
We fix an open covering $(U_j)_j$ and a family $(s_j)_j$ of rigid CR frames $s_j$ on
$U_j$.
Let $L^k$ be the $k$-th tensor power of $L$.
Then $(s_j^{\otimes k})_j$ are rigid CR frames for $L^k$.

The tangential Cauchy-Riemann operator
$\overline\partial_b:\Omega^{0,q}(X, L^k)\rightarrow\Omega^{0,q+1}(X, L^k)$ is well defined.
Since $L^k$ is rigid, we can also define $Tu$  for every $u\in\Omega^{0,q}(X,L^k)$ and we have
\begin{equation}\label{e-gue150508d}
T\ddbar_b=\ddbar_bT\ \ \mbox{on $\Omega^{0,q}(X,L^k)$}.
\end{equation}
For every $m\in\mathbb Z$, let
\begin{equation}\label{e-gue150508dI}
\Omega^{0,q}_m(X,L^k):=\set{u\in\Omega^{0,q}(X,L^k);\, Tu=imu}.
\end{equation}
For $q=0$, we write $C^\infty_m(X,L^k):=\Omega^{0,0}_m(X,L^k)$.

Let $h$ be a Hermitian metric on $L$. The local weight of $h$
with respect to a local rigid CR frame $s$ of $L$ over an open subset $D\subset X$
is the function $\Phi\in C^\infty(D, \mathbb R)$ for which
\begin{equation}\label{e-gue150808g}
|s(x)|^2_{h^L}=e^{-2\Phi(x)}, x\in D.
\end{equation}
We denote by $\Phi_j$ the weight of $h$ with respect to $s_j$.
\begin{defn}\label{d-gue150808g}
Let $L$ be a rigid CR line bundle and let $h$ be a Hermitian metric on $L$.
The curvature of $(L,h)$ is the the Hermitian quadratic form $R^L=R^{(L,h)}$ on $T^{1,0}X$
defined by
\begin{equation}\label{e-gue150808w}
R_p^L(U,\overline V)=\,\big\langle d(\overline\partial_b\Phi_j-\partial_b\Phi_j)(p),
U\wedge\overline V\,\big\rangle,\:\: U, V\in T_p^{1,0}X,\:\: p\in U_j.
\end{equation}
\end{defn}
Due to \cite[Proposition 4.2]{HM09}, $R^L$ is a well-defined global Hermitian form,
since the transition functions between different frames $s_j$ are annihilated by $T$.
\begin{defn}\label{d-gue150808gI}
We say that $(L,h)$ is positive if the associated curvature $R^L_x$ is positive definite
at every $x\in X$.
\end{defn}

\begin{defn}\label{d-gue150808gII}
Let $L$ be a rigid line bundle. A Hermitian metric $h$ on $L$
is said to be rigid if for every $j$ we have
$T\Phi_j=0$ on $U_j$.
\end{defn}
The definition does not depend on the choice of covering and rigid frames.
The following is well-known (see Lemma 1.20 in~\cite{HL15}).

\begin{lem}\label{l-rigid}
Let $L$ be a rigid CR line bundle. Then
there is a rigid Hermitian fiber metric on $L$. Moreover, for any Hermitian metric $\Td h$ on $L$,
there is a rigid Hermitian metric $h$ of $L$ such that $R^{(L,\widetilde{h})}=R^{(L,h)}$ on $X$.
% where $\Td{\mathcal{R}}^L$ and $R^L$ denote the curvatures induced by $\Td h$ and $h$ respectively.
\end{lem}

From now on, we assume that $(L,h)$ is positive. In view of Lemma \ref{l-rigid}
we can and will assume without loss of generality that $h$ is rigid.
%
%and we take a rigid Hermitian fiber metric $h$ on $L$ such that the associated curvature $R^L_x$ is positive definite at every $x\in X$. In view of Lemma~\ref{l-rigid}, we see that this is always possible.
There is also a notion of rigid Hermitian metric for vector bundles of arbitrary
rank, which we recall now.
\begin{defn}\label{d-gue150514f}
Let $F$ be a rigid vector bundle over $X$. Let $\langle\,\cdot\,|\,\cdot\,\rangle_F$ be a Hermitian metric on $F$. We say that $\langle\,\cdot\,|\,\cdot\,\rangle_F$ is a rigid Hermitian metric if for every local rigid frame $f_1,\ldots, f_r$ of $F$, we have $T\langle\,f_j\,|\,f_k\,\rangle_F=0$, for every $j,k=1,2,\ldots,r$.
\end{defn}

For the following result we refer to \cite[Theorem 2.10]{CHT15}.

\begin{thm}\label{t-gue150514fa}
On every rigid vector bundle $F$ over $X$ there is a rigid Hermitian metric $\langle\,\cdot\,|\,\cdot\,\rangle_F$.
\end{thm}

Baouendi-Rothschild-Treves~\cite{BRT85} proved that $T^{1,0}X$ is a rigid complex vector bundle over $X$. By Theorem~\ref{t-gue150514fa}, there is a rigid Hermitian metic on $T^{1,0}X$. From now on, we take a rigid Hermitian metric $\langle\,\cdot\,|\,\cdot\,\rangle$ on $\Complex TX$ such that $T^{1,0}X\perp T^{0,1}X$, $T\perp (T^{1,0}X\oplus T^{0,1}X)$, $\langle\,T\,|\,T\,\rangle=1$. The Hermitian metric $\langle\,\cdot\,|\,\cdot\,\rangle$ on
$\mathbb CTX$ induces by duality a Hermitian metric on $\mathbb CT^*X$ and also on the bundles of $(0,q)$ forms $T^{*0,q}X, q=0,1\cdots,n-1$. We shall also denote all these induced
metrics by $\langle\,\cdot\,|\,\cdot\,\rangle$. For every $v\in T^{*0,q}X$, we write
$\abs{v}^2:=\langle\,v\,|\,v\rangle$.

The Hermitian metrics on $T^{*0,q}X$ and $L$ induce
Hermitian metrics on $T^{*0,q}X\otimes L^k$, $q=0,1,\ldots,n$. We shall also denote these induced metrics by $\langle\,\cdot\,|\,\cdot\,\rangle_{h^k}$. For $f\in\Omega^{0,q}(X,L^k)$, we denote the pointwise norm $\abs{f(x)}^2_{h^k}:=\langle\,f(x)\,|\,f(x)\rangle_{h^k}$.
Let $dv_X=dv_X(x)$ the volume form on $X$ induced by the fixed
Hermitian metric $\langle\,\cdot\,|\,\cdot\,\rangle$ on $\Complex TX$.
Then we get natural global $L^2$ inner products $(\,\cdot\,|\,\cdot\,)$
on $\Omega^{0,q}(X,L^k)$ and $\Omega^{0,q}(X)$ respectively.
We denote by $L^2(X,T^{*0,q}X
\otimes L^k)$ and $L^2(X,T^{*0,q}X)$ the completions of
$\Omega^{0,q}(X,L^k)$ and $\Omega^{0,q}(X)$ with respect to $(\,\cdot\,|\,\cdot\,)$.

Similarly, for each $m\in\mathbb Z$, we denote by $L^2_m(X,T^{*0,q}X\otimes L^k)$
and $L^2_m(X,T^{*0,q}X)$ the completions of $\Omega^{0,q}_m(X,L^k)$
and $\Omega^{0,q}_m(X)$ with respect to $(\,\cdot\,|\,\cdot\,)$.
We extend $(\,\cdot\,|\,\cdot\,)$ and $(\,\cdot\,|\,\cdot\,)$ to $L^2(X,T^{*0,q}X\otimes L^k)$
and $L^2(X,T^{*0,q}X)$ in the standard way.  For $f\in\Omega^{0,q}(X,L^k)$
or $f\in\Omega^{0,q}(X)$,
we denote $\norm{f}^2:=(\,f\,|\,f\,)$.
%Similarly, for $f\in\Omega^{0,q}(X)$, we denote $\norm{f}^2:=(\,f\,|\,f\,)$.

\subsection{Expression of $T$ and $\ddbar_b$ in BRT trivializations}\label{s-gue150514}
%Let $(D,(z,\theta),\phi)$ be a BRT trivialization on $X$, cf.\ \eqref{e-can1}, \eqref{e-can2}.
%In this work, we need the following result due to Baouendi-Rothschild-Treves~\cite{BRT85}.
%
%\begin{thm}\label{t-gue150514}
%For every point $x_0\in X$, we can find local coordinates $x=(x_1,\cdots,x_{2n-1})=(z,\theta)=(z_1,\cdots,z_{n-1},\theta), z_j=x_{2j-1}+ix_{2j},j=1,\cdots,n-1, x_{2n-1}=\theta$, defined in some small neighborhood $D=\{(z, \theta): \abs{z}<\delta, -\varepsilon_0<\theta<\varepsilon_0\}$ of $x_0$, $\delta>0$, $0<\varepsilon_0<\pi$, such that $(z(x_0),\theta(x_0))=(0,0)$ and
%\begin{equation}\label{e-can}
%\begin{split}
%&T=\frac{\partial}{\partial\theta}\\
%&Z_j=\frac{\partial}{\partial z_j}-i\frac{\partial\phi}{\partial z_j}(z)\frac{\partial}{\partial\theta},j=1,\cdots,n-1
%\end{split}
%\end{equation}
%where $Z_j(x), j=1,\cdots, n-1$, form a basis of $T_x^{1,0}X$, for each $x\in D$ and $\phi(z)\in C^\infty(D,\mathbb R)$ independent of $\theta$. We call $x=(x_1,\ldots,x_{2n-1})$ canonical coordinates, $D$ canonical coordinate patch and $(D,(z,\theta),\phi)$ BRT trivialization.
%\end{thm}
%
In a BRT trivialization $(D,(z,\theta),\phi)$,
we have a useful formula for the operator $T$ on $\Omega^{0,q}(X)$
defined by \eqref{e-gue150508faII}.
It is clear that
$$\{d\overline{z}_{j_1}\wedge\ldots\wedge d\overline{z}_{j_q}, 1\leq j_1<\ldots<j_q\leq n-1\}$$
is a rigid frame of $T^{\ast0,q}X$ on $D$ so for $u\in\Omega^{0,q}(X)$ we write
\begin{equation*}\label{e-gue150524fb}
u=\sum\limits_{j_1<\ldots<j_q}u_{j_1\ldots j_q}d\overline{z}_{j_1}\wedge\ldots\wedge d\overline{z}_{j_q}
\:\:\text{on $D$}.
\end{equation*}
Then we can check that
\begin{equation}\label{lI}
Tu=\sum\limits_{j_1<\ldots<j_q}(Tu_{j_1\ldots j_q})d\overline{z}_{j_1}\wedge\ldots\wedge d\overline{z}_{j_q}
\:\:\text{on $D$}.
\end{equation}
%and $Tu$ is independent of the choice of BRT trivializations.
Note that on BRT trivialization $(D,(z,\theta),\phi)$, we have
\begin{equation}\label{e-gue150514f}
\ddbar_b=\sum^{n-1}_{j=1}d\ol z_j\wedge\Big(\frac{\partial}{\partial\ol z_j}+i\frac{\partial\phi}{\partial\ol z_j}(z)\frac{\partial}{\partial\theta}\Big).
\end{equation}

\section{Szeg\H{o} kernel asymptotics}\label{e-gue150811}

In this section, we will prove Theorem~\ref{t-gue150807}. We first introduce some notations. Let
\[\ol{\pr}^{*}_b:\Omega^{0,q+1}(X,L^k)\To\Omega^{0,q}(X,L^k)\]
be the formal adjoint of $\ddbar_b$ with respect to $(\,\cdot\,|\,\cdot\,)$. Since $\langle\,\cdot\,|\,\cdot\,\rangle$ and $h$ are rigid, we can check that
\begin{equation}\label{e-gue150517}
\begin{split}
&T\ddbar^{*}_b=\ddbar^{*}_bT\ \ \mbox{on $\Omega^{0,q}(X,L^k)$, $q=1,2,\ldots,n-1$},\\
&\ddbar^{*}_b:\Omega^{0,q+1}_m(X,L^k)\To\Omega^{0,q}_m(X,L^k),\ \ \forall m\in\mathbb Z.
\end{split}
\end{equation}
Put
\begin{equation}\label{e-gue150517I}
\Box^{(q)}_{b,k}:=\ddbar_b\ol{\pr}^{*}_b+\ol{\pr}^{*}_{b}\ddbar_b:\Omega^{0,q}(X,L^k)\To\Omega^{0,q}(X,L^k).
\end{equation}
From \eqref{e-gue150508d} and \eqref{e-gue150517}, we have
\begin{equation}\label{e-gue150517II}
\begin{split}
&T\Box^{(q)}_{b,k}=\Box^{(q)}_{b,k}T\ \ \mbox{on $\Omega^{0,q}(X,L^k)$, $q=0,1,\ldots,n-1$},\\
&\Box^{(q)}_{b,k}:\Omega^{0,q}_m(X,L^k)\To\Omega^{0,q}_m(X,L^k),\ \ \forall m\in\mathbb Z.
\end{split}
\end{equation}
Let $\Pi_k:L^2(X)\To{\rm Ker\,}\Box^{(0)}_{b,k}$ be the orthogonal projection
(the Szeg\H{o} projector).
\begin{defn}\label{d-gue130820m}
Let $A_k:L^2(X,L^k)\To L^2(X,L^k)$ be a continuous operator.
Let $D\Subset X$. We say that $\Box^{(0)}_{b,k}$ has $O(k^{-n_0})$
small spectral gap on $D$ with respect to $A_k$ if for every $D'\Subset D$,
there exist constants $C_{D'}>0$,  $n_0, p\in\mathbb N$, $k_0\in\mathbb N$,
such that for all $k\geq k_0$ and $u\in C^\infty_0(D',L^k)$,
we have
\[\norm{A_k(I-\Pi_k)u}\leq
C_{D'}\,k^{n_0}\sqrt{(\,(\Box^{(0)}_{b,k})^pu\,|\,u\,)}\,.\]
\end{defn}

Fix $\lambda>0$ and let $\Pi_{k,\leq\lambda}$ be as in \eqref{e-gue150806V}.

\begin{defn}\label{d-gue131205m}
Let $A_k:L^2(X,L^k)\To L^2(X,L^k)$ be a continuous operator.
We say that $\Pi_{k,\leq\lambda}$ is $k$-negligible away the diagonal with respect to $A_k$
on $D\Subset X$ if for any $\chi, \chi_1\in C^\infty_0(D)$ with $\chi_1=1$
on some neighborhood of ${\rm Supp\,}\chi$, we have
\[\big(\chi A_k(1-\chi_1)\big)\Pi_{k,\leq\lambda}\big(\chi A_k(1-\chi_1)\big)^*=
O(k^{-\infty})\ \ \mbox{ on $D$},\]
where $\big(\chi A_k(1-\chi_1)\big)^*:L^2(X,L^k)\To L^2(X,L^k)$
is the Hilbert space adjoint of $\chi A_k(1-\chi_1)$ with respect to $(\,\cdot\,|\,\cdot\,)$.
\end{defn}

Fix  $\delta>0$ and let $F_{k,\delta}$ be as in \eqref{e-gue150807I}.

\begin{thm}[{\cite[Theorem 1.13]{Hsiao14}}]\label{t-gue150811}
With the notations and assumptions used above,
let $s$ be a local rigid CR frame of $L$ on a canonical coordinate patch $D\Subset X$
with canonical coordinates $x=(z,\theta)=(x_1,\ldots,x_{2n-1})$,
$\abs{s}^2_{h}=e^{-2\Phi}$. Let $\delta>0$ be a small constant so that
$R^L_{x}-2s\mathcal{L}_{x}$ is positive definite, for every $x\in X$ and
$\abs{s}\leq\delta$. Let $F_{k,\delta}$ be as in \eqref{e-gue150807I} and
let $F_{k, \delta,s}$ be the localized operator of $F_{k,\delta}$ given by \eqref{e-gue140824II}.
Assume that:

{\rm (I)\,} $\Box^{(0)}_{b,k}$ has $O(k^{-n_0})$ small spectral gap on $D$ with respect to $F_{k,\delta}$.

{\rm (II)\,} $\Pi_{k,\leq\delta k}$ is $k$-negligible away the diagonal with
respect to $F_{k,\delta}$ on $D$.

{\rm (III)\,} $F_{k, \delta,s}-B_k=O(k^{-\infty}):
H^s_{{\rm comp\,}}(D)\To H^s_{{\rm loc\,}}(D)$, $\forall s\in\mathbb N_0$, where
\[B_k=\frac{k^{2n-1}}{(2\pi)^{2n-1}}\int e^{ik\langle x-y,\eta\rangle }\alpha(x,\eta,k)d\eta
+ O(k^{-\infty})\]
is a classical semi-classical pseudodifferential operator on $D$ of order $0$ with
\[\begin{split}&\mbox{$\alpha(x,\eta,k)\sim\sum_{j=0}^\infty\alpha_j(x,\eta)k^{-j}$
in $S^0_{{\rm loc\,}}(1;T^*D)$},\\
&\alpha_j(x,\eta)\in C^\infty(T^*D),\ \ j=0,1,\ldots,
\end{split}\]
and for every $(x,\eta)\in T^*D$, $\alpha(x,\eta,k)=0$
if $\big|\langle\,\eta\,|\,\omega_0(x)\,\rangle\big|> \delta$. Fix $D_0\Subset D$. Then
\begin{equation}\label{c1}
P_{k,\delta,s}(x,y)=\int e^{ik\varphi(x,y,t)}g(x,y,t,k)dt+O(k^{-\infty})\:\:
\text{on $D_0\times D_0$},
\end{equation}
where $\varphi(x,y,t)\in C^\infty( D\times D\times(-\delta,\delta))$ is as in \eqref{e-gue150807b} and
\[\begin{split}
&g(x,y,t,k)\in S^{n}_{{\rm loc\,}}(1;D\times D\times(-\delta,\delta))\cap C^\infty_0(D\times D\times(-\delta,\delta)),\\
&g(x,y,t,k)\sim\sum^\infty_{j=0}g_j(x,y,t)k^{n-j}\text{ in }S^{n}_{{\rm loc\,}}(1;D\times D\times(-\delta,\delta))
\end{split}\]
is as in \eqref{e-gue150807bI}, where $P_{k,\delta,s}$ is given by \eqref{e-gue150806VII}.
\end{thm}

In view of Theorem~\ref{t-gue150811}, we see that to prove
Theorem~\ref{t-gue150807}, we only need to prove  that
{\rm (I)\,},  {\rm (II)\,} and  {\rm (III)\,} in Theorem~\ref{t-gue150811}
hold if $\delta>0$ is small enough. Until further notice,
we fix $\delta>0$ small enough so that $R^L_{x}-2s\mathcal{L}_{x}$
is positive definite for every $x\in X$ and $\abs{s}\leq\delta$.
We first prove {\rm (I)\,} in Theorem~\ref{t-gue150811} holds.

\subsection{Small spectral gap of the Kohn Laplacian}\label{s-gue150811h}

For $m\in\mathbb Z$, let
\begin{equation}\label{e-gue131207III}
Q^{(q)}_{m,k}:L^2(X,T^{*0,q}X\otimes L^k)\To L^2_m(X,T^{*0,q}X\otimes L^k)
\end{equation}
be the orthogonal projection with respect to $(\,\cdot\,|\,\cdot\,)$.
Let $\tau_\delta\in C^\infty_0((-\delta,\delta))$ be as in \eqref{e-gue150807}.
Similar to \eqref{e-gue150807I}, let
$F^{(q)}_{k, \delta}$
be the continuous operator given by
\begin{equation}\label{e-gue131207IV}
%\begin{split}
F^{(q)}_{k, \delta}:L^2(X,T^{*0,q}X\otimes L^k)\To L^2(X,T^{*0,q}X\otimes L^k),\quad
u\longmapsto\sum_{m\in\mathbb Z}\tau_\delta\Big(\frac{m}{k}\Big)Q^{(q)}_{m,k}u\,.
%\end{split}
\end{equation}
Note that $F_{k,\delta}=F^{(0)}_{k, \delta}$. It is not difficult to see that for every $m\in\mathbb Z$, we have
\begin{equation}\label{e-gue131207V}
\begin{split}
&\norm{TQ^{(q)}_{m,k}u}=\abs{m}\norm{Q^{(q)}_{m,k}u},\ \ \forall u\in L^2(X,T^{*0,q}X\otimes L^k),\\
&\norm{TF^{(q)}_{k, \delta}u}\leq k\delta\norm{F^{(q)}_{k, \delta}u},\ \ \forall u\in L^2(X,T^{*0,q}X\otimes L^k),
\end{split}
\end{equation}
and
\begin{equation}\label{e-gue131207V-I}
\begin{split}
&Q^{(q)}_{m,k}:\Omega^{0,q}(X,L^k)\To\Omega^{0,q}_m(X,L^k),\\
&F^{(q)}_{k, \delta}:\Omega^{0,q}(X,L^k)\To
\bigoplus_{m\in\mathbb Z\cap[-k\delta,k\delta]}\Omega^{0,q}_m(X,L^k).
\end{split}
\end{equation}
Since the Hermitian metrics $\langle\,\cdot\,|\,\cdot\,\rangle$ and $h^{k}$ are all rigid,
it follows as in \cite[Section 5]{Hsiao12}:
\begin{equation}\label{e-gue131207VI}
\begin{split}
&\mbox{$\Box^{(q)}_{b,k}Q^{(q)}_{m,k}=Q^{(q)}_{m,k}\Box^{(q)}_{b,k}$ on $\Omega^{0,q}(X,L^k)$, $\forall m\in\mathbb Z$},\\
&\mbox{$\Box^{(q)}_{b,k}F^{(q)}_{k, \delta}=F^{(q)}_{k, \delta}\Box^{(q)}_{b,k}$ on $\Omega^{0,q}(X,L^k)$},\\
&\mbox{$\ddbar_{b}Q^{(q)}_{m,k}=Q^{(q+1)}_{m,k}\ddbar_{b}$ on $\Omega^{0,q}(X,L^k)$, $\forall m\in\mathbb Z$, $q=0,1,\ldots,n-2$},\\
&\mbox{$\ddbar_{b}F^{(q)}_{k, \delta}=F^{(q+1)}_{k, \delta}\ddbar_{b}$ on $\Omega^{0,q}(X,L^k)$, $q=0,1,\ldots,n-2$},\\
&\mbox{$\ddbar^*_{b}Q^{(q)}_{m,k}=Q^{(q-1)}_{m,k}\ddbar^*_{b}$ on $\Omega^{0,q}(X,L^k)$, $\forall m\in\mathbb Z$, $q=1,\ldots,n-1$},\\
&\mbox{$\ddbar^*_{b}F^{(q)}_{k, \delta}=F^{(q-1)}_{k, \delta}\ddbar^*_{b}$ on $\Omega^{0,q}(X,L^k)$, $q=1,\ldots,n-1$}.
\end{split}
\end{equation}
By elementary Fourier analysis, it is straightforward to see that for every $u\in\Omega^{0,q}(X,L^k)$,
\begin{equation}\label{e-gue131207VII}
\begin{split}
&\mbox{$\lim\limits_{N\To\infty}\sum\limits^N_{m=-N}Q^{(q)}_{m,k}u\To u$ in $C^\infty$ Topology},\\
&\sum^N_{m=-N}\norm{Q^{(q)}_{m,k}u}^2\leq\norm{u}^2,\ \ \forall N\in\mathbb N_0.
\end{split}
\end{equation}
Thus, for every $u\in L^2(X,T^{*0,q}X\otimes L^k)$,
\begin{equation}\label{e-gue131207VIII}
\begin{split}
&\mbox{$\lim\limits_{N\To\infty}\sum\limits^N_{m=-N}Q^{(q)}_{m,k}u\To u$ in $L^2(X,T^{*0,q}X\otimes L^k)$},\\
&\sum^N_{m=-N}\norm{Q^{(q)}_{m,k}u}^2\leq\norm{u}^2,\ \ \forall N\in\mathbb N_0.
\end{split}
\end{equation}
We will use the following result.
\begin{thm}[{\cite[Theorem 9.4]{Hsiao14}}]\label{t-gue131207I}
With the assumptions and notations above, let $q\geq1$. If $\delta>0$ is small enough, then for every $u\in\Omega^{0,q}(X,L^k)$, we have
\begin{equation}\label{e-gue131207aI}
\big\|\Box^{(q)}_{b,k}F^{(q)}_{k, \delta}u\big\|^2\geq
c_1k^2\big\|F^{(q)}_{k, \delta} u\big\|^2,
\end{equation}
where $c_1>0$ is a constant independent of $k$ and $u$.
\end{thm}

Now, we assume that $\delta>0$ is small enough so that \eqref{e-gue131207aI} holds. For $q=0,1,\ldots,n-1$, put
\begin{equation}\label{e-gue150811y}
\begin{split}
\Omega^{0,q}_{\leq k\delta}(X, L^k)&:=
\bigoplus_{\substack{m\in\mathbb Z\\\abs{m}\leq k\delta}}\Omega^{0,q}_m(X,L^k),\\
L^2_{\leq k\delta}(X,T^{*0,q}X\otimes L^k)&:=
\bigoplus_{\substack{m\in\mathbb Z\\\abs{m}\leq k\delta}}L^2_m(X,T^{*0,q}X\otimes L^k).
\end{split}
\end{equation}
We write $C^\infty_{\leq k\delta}(X,L^k):=\Omega^{0,0}_{\leq k\delta}(X, L^k)$, $L^2_{\leq k\delta}(X,L^k):=L^2_{\leq k\delta}(X, T^{*0,0}X\otimes L^k)$. It is clear that
\[\Box^{(q)}_{b,k}:\Omega^{0,q}_{\leq k\delta}(X,T^{*0,q}X\otimes L^k)\To \Omega^{0,q}_{\leq k\delta}(X,T^{*0,q}X\otimes L^k).\]
We will write $\Box^{(q)}_{b,\leq k\delta}$ to denote the restriction of $\Box^{(q)}_{b,k}$ on the space $\Omega^{0,q}_{\leq k\delta}(X, L^k)$. We extend $\Box^{(q)}_{b,\leq k\delta}$ to $L^2_{\leq k\delta}(X,T^{*0,q}X\otimes L^k)$ by
\begin{equation}\label{e-gue150517III}
\Box^{(q)}_{b,\leq k\delta}:{\rm Dom\,}\Box^{(q)}_{b,\leq k\delta}\subset L^2_{\leq k\delta}(X,T^{*0,q}X\otimes L^k)\To L^2_{\leq k\delta}(X,T^{*0,q}X\otimes L^k)\,,
\end{equation}
with ${\rm Dom\,}\Box^{(q)}_{b,\leq k\delta}:=\{u\in L^2_{\leq k\delta}(X,T^{*0,q}X\otimes L^k);\, \Box^{(q)}_{b,\leq k\delta}u\in L^2_{\leq k\delta}(X,T^{*0,q}X\otimes L^k)\}$, where for any $u\in L^2_{\leq k\delta}(X,T^{*0,q}X\otimes L^k)$, $\Box^{(q)}_{b,\leq k\delta}u$ is defined in the sense of distributions.

\begin{lem}\label{l-gue150517}
We have ${\rm Dom\,}\Box^{(q)}_{b,\leq k\delta}=L^2_{\leq k\delta}(X,T^{*0,q}X\otimes L^k)\bigcap H^2(X,T^{*0,q}X\otimes L^k)$.
\end{lem}

\begin{proof}
It is clearly that $L^2_{\leq k\delta}(X,T^{*0,q}X\otimes L^k)
\bigcap H^2(X,T^{*0,q}X\otimes L^k)\subset {\rm Dom\,}\Box^{(q)}_{b,\leq k\delta}$.
We only need to prove that ${\rm Dom\,}\Box^{(q)}_{b,\leq k\delta}\subset
L^2_{\leq k\delta}(X,T^{*0,q}X\otimes L^k)\bigcap H^2(X,T^{*0,q}X\otimes L^k)$.
Let $u\in{\rm Dom\,}\Box^{(q)}_{b,\leq k\delta}$.
Put $v=\Box^{(q)}_{b,\leq k\delta}u\in L^2_{\leq k\delta}(X,T^{*0,q}X\otimes L^k)$.
We have $(\Box^{(q)}_{b,\leq k\delta}-T^2)u=v-T^2u\in L^2(X,T^{*0,q}X\otimes L^k)$
since $\norm{T^2u}\leq k^2\delta^2\norm{u}$.
Since $(\Box^{(q)}_{b,\leq k\delta}-T^2)$ is elliptic, we have $u\in H^2(X,T^{*0,q}X\otimes L^k)$.
The lemma follows.
\end{proof}
\begin{thm}\label{t-gue150517a}
The operator $\Box^{(q)}_{b,\leq k\delta}$ defined in \eqref{e-gue150517III}
%$\Box^{(q)}_{b,\leq k\delta}:{\rm Dom\,}\Box^{(q)}_{b,\leq k\delta}
%\subset L^2_{\leq k\delta}(X,T^{*0,q}X\otimes L^k)\To L^2_{\leq k\delta}(X,T^{*0,q}X\otimes L^k)$
is self-adjoint.
\end{thm}

\begin{proof}
Let $(\Box^{(q)}_{b,\leq k\delta})^*:{\rm Dom\,}(\Box^{(q)}_{b,\leq k\delta})^*\subset L^2_{\leq k\delta}(X,T^{*0,q}X\otimes L^k)\To L^2_{\leq k\delta}(X,T^{*0,q}X\otimes L^k)$ be the Hilbert space adjoint of $\Box^{(q)}_{b,\leq k\delta}$.
Let $v\in{\rm Dom\,}(\Box^{(q)}_{b,\leq k\delta})^*$. Then, by definition of the Hilbert space adjoint of $\Box^{(q)}_{b,\leq k\delta}$, it is easy to see that $\Box^{(q)}_{b,\leq k\delta}v\in L^2_{\leq k\delta}(X,T^{*0,q}X\otimes L^k)$ and hence $v\in{\rm Dom\,}\Box^{(q)}_{b,\leq k\delta}$ and $\Box^{(q)}_{b,\leq k\delta}v=(\Box^{(q)}_{b,\leq k\delta})^*v$.

From Lemma~\ref{l-gue150517}, we can check that
\begin{equation}\label{e-gue150517fa}
\big(\,\Box^{(q)}_{b,\leq k\delta}g\,|\,f\,\big)=\big(\,g\,|\,\Box^{(q)}_{b,\leq k\delta}f\,\big),
\ \ \forall g,f\in{\rm Dom\,}\Box^{(q)}_{b,\leq k\delta}.
\end{equation}
From \eqref{e-gue150517fa}, we deduce that ${\rm Dom\,}\Box^{(q)}_{b,\leq k\delta}\subset{\rm Dom\,}(\Box^{(q)}_{b,\leq k\delta})^*$ and $\Box^{(q)}_{b,\leq k\delta}u=(\Box^{(q)}_{b,\leq k\delta})^*u$, for all $u\in{\rm Dom\,}\Box^{(q)}_{b,\leq k\delta}$. The theorem follows.
\end{proof}

Let ${\rm Spec\,}\Box^{(q)}_{b,\leq k\delta}\subset[0,\infty[$ denote the spectrum of $\Box^{(q)}_{b,\leq k\delta}$. For any $\lambda>0$, put
\[\begin{split}
&\Pi^{(q)}_{k,\leq k\delta,\leq\lambda}:=E^{(q)}_{\leq k\delta}([0,\lambda]),\\
&\Pi^{(q)}_{k,\leq k\delta,>\lambda}:=E^{(q)}_{\leq k\delta}(]\lambda,\infty[),
\end{split}\]
where $E^{(q)}_{\leq k\delta}$ denotes the spectral measure for
$\Box^{(q)}_{b,\leq k\delta}$.
We write $\Pi_{k,\leq k\delta,\leq\lambda}:=\Pi^{(0)}_{k,\leq k\delta,\leq\lambda}$,
$\Pi_{k,\leq k\delta,>\lambda}:=\Pi^{(0)}_{k,\leq k\delta,>\lambda}$.
\begin{thm}\label{t-gue150517aI}
${\rm Spec\,}\Box^{(q)}_{b,\leq k\delta}$ is a discrete subset of
$[0,\infty[$ , for any $\nu\in{\rm Spec\,}\Box^{(q)}_{b,\leq k\delta}$, $\nu$
is an eigenvalue of $\Box^{(q)}_{b,\leq k\delta}$ and the eigenspace
\[\mathcal{E}^q_{\leq k\delta,\nu}(X,L^k)
:=\big\{u\in{\rm Dom\,}\Box^{(q)}_{b,\leq k\delta};\, \Box^{(q)}_{b,\leq k\delta}u=\nu u\big\}\]
is finite dimensional with
$\mathcal{E}^q_{\leq k\delta,\nu}(X,L^k)\subset\Omega^{0,q}_{\leq k\delta}(X,L^k)$.
\end{thm}

\begin{proof}
Fix $\lambda>0$. We claim that ${\rm Spec\,}\Box^{(q)}_{b,\leq k\delta}\bigcap\,[0,\lambda]$
is discrete. If not, we can find an orthonormal system
$\{f_j\in{\rm Range\,}E^{(q)}_{\leq k\delta}([0,\lambda]); j\in\N\}$,
i.\ e.\ $(\,f_j\,|\,f_\ell\,)=\delta_{j,\ell}$ for all $j,\ell\in\N$. Note that
\begin{equation}\label{e-gue150517d}
\big\|\Box^{(q)}_{b,\leq k\delta}f_j\big\|\leq\lambda\big\|f_j\big\|,\ \ j=1,2,\ldots.
\end{equation}
From \eqref{e-gue150517d}, we have
\begin{equation}\label{e-gue150520}
\big\|(\Box^{(q)}_{b,\leq k\delta}-T^2)f_j\big\|\leq(\lambda+k^2\delta^2)\big\|f_j\big\|,\ \ j=1,2,\ldots.
\end{equation}
Since $\Box^{(q)}_{b,\leq k\delta}-T^2$ is a second order elliptic operator, there is a constant
$C_\delta>0$ independent of $j$ such that
\begin{equation}\label{e-gue150520I}
\norm{f_j}_{2}\leq C_\delta,\ \ j=1,2,\ldots,
\end{equation}
where $\norm{\cdot}_{2}$ denotes the usual Sobolev norm of order $2$.
From \eqref{e-gue150520I}, we can apply Rellich's theorem and find
subsequence $\set{f_{j_s}}^\infty_{s=1}$, such that
$f_{j_s}\To f$ in $L^2_{\leq k\delta}(X,T^{*0,q}X\otimes L^k)$.
This is a contradiction to the fact that $\{f_j; j\in\N\}$ is orthonormal.
%$(\,f_j\,|\,f_\ell\,)=\delta_{j,\ell}$ for $j,\ell=1,2,\ldots$,
Thus, ${\rm Spec\,}\Box^{(q)}_{b,\leq k\delta}\bigcap[0,\lambda]$ is discrete
and therefore ${\rm Spec\,}\Box^{(q)}_{b,\leq k\delta}$ is a discrete subset of $[0,\infty[$.

Let $r\in{\rm Spec\,}\Box^{(q)}_{b,\leq k\delta}$.
Since ${\rm Spec\,}\Box^{(q)}_{b,\leq k\delta}$ is discrete,
$\Box^{(q)}_{b,\leq k\delta}-r$ has $L^2$ closed range.
If $\Box^{(q)}_{b,\leq k\delta}-r$ is injective,
then ${\rm Range\,}(\Box^{(q)}_{b,\leq k\delta}-r)=
L^2_{\leq k\delta}(X,T^{*0,q}X\otimes L^k)$ and
 \[(\Box^{(q)}_{b,\leq k\delta}-r)^{-1}:
 L^2_{\leq k\delta}(X,T^{*0,q}X\otimes L^k)\To
 L^2_{\leq k\delta}(X,T^{*0,q}X\otimes L^k)\]
 is continuous. We get a contradiction.
 Hence $r$ is an eigenvalue of $\Box^{(q)}_{b,\leq k\delta}$.

For any $\nu\in{\rm Spec\,}\Box^{(q)}_{b,\leq k\delta}$, put
\[\mathcal{E}^q_{\leq k\delta,\nu}(X,L^k):=\big\{u\in{\rm Dom\,}\Box^{(q)}_{b,\leq k\delta};\, \Box^{(q)}_{b,\leq k\delta}u=\nu u\big\}.\]
We can repeat the argument before and conclude that $\mathcal{E}^q_{\leq k\delta,\nu}(X,E)$
is finite dimensional. Let $u\in\mathcal{E}^q_{\leq k\delta,\nu}(X,L^k)$.
Then, $\Box^{(q)}_{b,\leq k\delta}u=\nu u$. For $m\in\mathbb Z$,
put $u_m:=Q^{(q)}_{m,k}u\in L^2_m(X,T^{*0,q}X\otimes L^k)$.
We have $u=\sum_{m\in\mathbb Z,\abs{m}\leq k\delta}u_m$. We can check that
\[\Box^{(q)}_{b,k}u_m=\nu u_m,\ \ \forall m\in\mathbb Z.\]
Hence
\begin{equation}\label{e-gue150520II}
(\Box^{(q)}_{b,k}-T^2)u_m=(\nu+m^2)u_m,\ \ \forall m\in\mathbb Z.
\end{equation}
From \eqref{e-gue150520II}, we can apply some standard argument in partial differential operator and deduce that $u_m\in\Omega^{0,q}_m(X,L^k)$.  Thus, $u\in\Omega^{0,q}_{\leq k\delta}(X,L^k)$ and hence $\mathcal{E}^q_{\leq k\delta,\nu}(X,L^k)\subset\Omega^{0,q}_{\leq k\delta}(X,L^k)$. The theorem follows.
\end{proof}

For every $\mu\in{\rm Spec\,}\Box^{(0)}_{b,\leq k\delta}$, let
\[\Pi_{k,\leq k\delta,\mu}:L^2(X,L^k)\To\mathcal{E}^0_{\leq k\delta,\mu}(X,L^k)\]
be the orthogonal projection. For $\mu=0$, it is clearly that $\Pi_{k,\leq k\delta,0}=\Pi_{k,\leq k\delta}$, where $\Pi_{k,\leq k\delta}$ is given by \eqref{e-gue150806V}. We have

\begin{thm}\label{t-gue131208}
With the assumptions and notations above, if $\epsilon_0>0$ is small enough, then for every $u\in C^\infty(X,L^k)$, we have
\begin{equation}\label{e-gue131216}
F_{k,\delta}\Pi_{k,\leq k\delta,\mu}u=0,\ \ \forall\mu\in{\rm Spec\,}\Box^{(0)}_{b,\leq k\delta},\ 0<\mu\leq k\epsilon_0,
\end{equation}
and
\begin{equation}\label{e-gue131208}
\big\|F_{k,\delta}(I-\Pi_{k,\leq k\delta})u\big\|\leq\frac{1}{k\epsilon_0}\big\|\Box^{(0)}_{b,k}u\big\|.
\end{equation}
\end{thm}

\begin{proof}
Let $\epsilon_0>0$ be a small constant. For $u\in L^2_{\leq k\delta}(X,L^k)$, we have
\begin{equation}\label{e-gue131208I}
(I-\Pi_{k,\leq k\delta})u=\sum_{\substack{\mu\in{\rm Spec\,}\Box^{(0)}_{b,\leq k\delta}\\
0<\mu\leq k\epsilon_0}}\Pi_{k,\leq k\delta,\mu}u+\Pi_{k,\leq k\delta,>k\epsilon_0}u.
\end{equation}
We claim that for every $\mu\in{\rm Spec\,}\Box^{(0)}_{b,\leq k\delta},0<\mu\leq k\epsilon_0$ and every $u\in C^\infty(X,L^k)$,
\begin{equation}\label{e-gue131208II}
F_{k,\delta}\Pi_{k,\leq k\delta,\mu}u=0
\end{equation}
if $\epsilon_0>0$ is small enough. Fix $\mu\in{\rm Spec\,}\Box^{(0)}_{b,\leq k\delta}\bigcap\,(0,k\delta]$
and $u\in C^\infty(X,L^k)$. From \eqref{e-gue131207VI} and \eqref{e-gue131207aI}, we have
\begin{equation}\label{e-gue131208III}
\norm{\Box^{(1)}_{b,k}F^{(1)}_{k, \delta}\ddbar_{b}\Pi_{k,\leq k\delta,\mu}u}^2\geq c_1k^2\norm{F^{(1)}_{k, \delta}\ddbar_{b}\Pi_{k,\leq k\delta,\mu}u}^2,\end{equation}
where $c_1>0$ is a constant independent of $k$ and $u$. It is easy to see that
\[\Box^{(1)}_{b,k}F^{(1)}_{k, \delta}\ddbar_{b}\Pi_{k,\leq k\delta,\mu}u=\mu F^{(1)}_{k, \delta}\ddbar_{b}\Pi_{k,\leq k\delta,\mu}u.\]
Thus,
\begin{equation}\label{e-gue131208IV}
\norm{\Box^{(1)}_{b,k}F^{(1)}_{k, \delta}\ddbar_{b}\Pi_{k,\leq k\delta,\mu}u}^2\leq k^2\epsilon_0^2\norm{F^{(1)}_{k, \delta}\ddbar_{b}\Pi_{k,\leq k\delta,\mu}u}^2.\end{equation}
From \eqref{e-gue131208III} and \eqref{e-gue131208IV}, we conclude that if $\epsilon_0>0$ is small enough then
\[F^{(1)}_{k, \delta}\ddbar_{b}\Pi_{k,\leq k\delta,\mu}u=\ddbar_{b}F_{k,\delta}\Pi_{k,\leq k\delta,\mu}u=0.\]
%Similarly, we have
%\[F^{(q-1)}_{\delta,k}\ddbar^*_{b}\Pi^{(q)}_{k,\mu}u=\ddbar^*_{b}F^{(q)}_{\delta,k}\Pi^{(q)}_{k,\mu}u=0.\]
Hence,
\begin{equation}\label{e-gue131208V}
F_{k,\delta}\Pi_{k,\leq k\delta,\mu}u=\frac{1}{\mu}\Box^{(0)}_{b,k}F_{k,\delta}\Pi_{k,\leq k\delta,\mu}u=0.
\end{equation}
From \eqref{e-gue131208V}, the claim \eqref{e-gue131208II} follows. We get \eqref{e-gue131216}.

Let $Q^{(0)}_{\leq k\delta}:L^2(X,L^k)\To L^2_{\leq k\delta}(X,L^k)$ be the orthogonal projection.
From \eqref{e-gue131208I} and \eqref{e-gue131208II}, if $\epsilon_0>0$ is small enough, then
\begin{equation}\label{e-gue131208VI}
\begin{split}
&\norm{F_{k,\delta}(I-\Pi_{k,\leq k\delta})u)}\\
&=\norm{F_{k,\delta}(I-\Pi_{k,\leq k\delta})(Q^{(0)}_{\leq k\delta}u)}\\
&=\norm{F_{k,\delta}\Pi_{k,\leq k\delta,>k\epsilon_0}(Q^{(0)}_{\leq k\delta}u)}\\
&\leq\norm{\Pi_{k,\leq k\delta,>k\epsilon_0}(Q^{(0)}_{\leq k\delta}u)}\\
&\leq\frac{1}{k\epsilon_0}\norm{\Box^{(0)}_{b,k}\Pi^{(0)}_{k,\leq k\delta,>k\epsilon_0}(Q^{(0)}_{\leq k\delta}u)}\\
&=\frac{1}{k\epsilon_0}\norm{\Pi^{(0)}_{k,\leq k\delta,>k\epsilon_0}\Box^{(0)}_{b,k}(Q^{(0)}_{\leq k\delta}u)}\leq\frac{1}{k\epsilon_0}\norm{\Box^{(0)}_{b,k}u},
\end{split}
\end{equation}
for every $u\in C^\infty(X,L^k)$. From \eqref{e-gue131208VI}, \eqref{e-gue131208} follows.
\end{proof}

\begin{thm}\label{t-gue150816}
$\Box^{(0)}_{b,k}$ has a $O(k^{-n_0})$ small spectral gap on $X$ with respect to $F_{k,\delta}$.
\end{thm}

\begin{proof}
Let $u\in C^\infty(X,L^k)$. It is easy to see that
\begin{equation}\label{e-gue150816}
F_{k,\delta}(I-\Pi_k)u=F_{k,\delta}(I-\Pi_{k,\leq k\delta})u.
\end{equation}
From \eqref{e-gue150816} and \eqref{e-gue131208}, the theorem follows.
\end{proof}

\subsection{The operator $F_{k,\delta}$ on canonical coordinate patch}\label{s-gue150816}

Let $D\subset X$ be a canonical coordinate patch and let $x=(x_1,\ldots,x_{2n-1})$ be canonical coordinates on $D$. We identify $D$ with $W\times]-\varepsilon,\varepsilon[\subset\Real^{2n-1}$, where $W$ is some open set in $\Real^{2n-2}$ and $\varepsilon>0$. Until further notice, we work with canonical coordinates $x=(x_1,\ldots,x_{2n-1})$. Let $\eta=(\eta_1,\ldots,\eta_{2n-1})$ be the dual coordinates of $x$.
Let $s$ be a local rigid CR frame of $L$ on $D$, $\abs{s}^2_{h}=e^{-2\Phi}$. Let $F_{k, \delta,s}$ be the localized operator of $F_{k,\delta}$ given by \eqref{e-gue140824II}. Put
\begin{equation}\label{e-gue131209}
B_{k}=\frac{k^{2n-1}}{(2\pi)^{2n-1}}\int e^{ik\langle x-y,\eta\rangle}\tau_\delta(\eta_{2n-1})d\eta.
\end{equation}

\begin{lem}\label{l-gue131209}
We have
\[F_{k, \delta,s}-B_{k}=O(k^{-\infty}):H^s_{{\rm comp\,}}(D)\To H^s_{{\rm loc\,}}(D),\ \ \forall s\in\mathbb N_0.\]
\end{lem}

\begin{proof}
We also write $y=(y_1,\ldots,y_{2n-1})$ to denote the canonical coordinates $x$. It is easy to see that on $D$,
\begin{equation}\label{e-gue131217}
F_{k, \delta,s}u(y)=\frac{1}{2\pi}
\sum_{m\in\mathbb Z}\tau_\delta\Big(\frac{m}{k}\Big)e^{imy_{2n-1}}\int^{\pi}_{-\pi}e^{-imt}
u(e^{it} y')dt,\ \ \forall u\in C^\infty_0(D),
\end{equation}
where $y'=(y_1,\ldots,y_{2n-2})$. Fix $D'\Subset D$ and let $\chi(y_{2n-1})\in C^\infty_0(]-\varepsilon,\varepsilon[)$ such that $\chi(y_{2n-1})=1$ for every $(y',y_{2n-1})\in D'$. Let $R_{k}:C^\infty_0(D')\To C^\infty(D')$ be the continuous operator given by
\begin{equation}\label{e-gue131217I}
\begin{split}
%R_{k}:C^\infty_0(D')&\To C^\infty(D'),\\
&(2\pi)^2R_ku=\\
&\sum_{m\in\mathbb Z}\:
\int\limits_{\abs{t}\leq\pi}e^{i\langle x_{2n-1}-y_{2n-1},\eta_{2n-1}\rangle+im(y_{2n-1}-t)}\tau_\delta\Big(\frac{\eta_{2n-1}}{k}\Big) (1-\chi(y_{2n-1}))u(e^{it}x')dtd\eta_{2n-1}dy_{2n-1}.
\end{split}
\end{equation}
By using integration by parts with respect to $\eta_{2n-1}$, it is easy to see that the integral \eqref{e-gue131217I} is well-defined. Moreover, we can integrate by parts with respect to $\eta_{2n-1}$ and $y_{2n-1}$ several times and conclude that
\begin{equation}\label{e-gue131217II}
R_{k}=O(k^{-\infty}):H^s_{{\rm comp\,}}(D')\To H^s_{{\rm loc\,}}(D'),\ \ \forall s\in\mathbb N_0.
\end{equation}
Now, we claim that
\begin{equation}\label{e-gue131217III}
B_{k}+R_k=F_{k, \delta,s}\ \ \mbox{on $C^\infty_0(D')$}.
\end{equation}
Let $u\in C^\infty_0(D')$. From \eqref{e-gue131209} and Fourier inversion formula, it is straightforward to see that
\begin{equation}\label{e-gue131217IV}
\begin{split}
B_{k}u(x)&=\frac{1}{(2\pi)^2}\sum_{m\in\mathbb Z}
\int_{\abs{t}\leq\pi}e^{i\langle x_{2n-1}-y_{2n-1},\eta_{2n-1}\rangle}
\tau_\delta\Big(\frac{\eta_{2n-1}}{k}\Big)\chi(y_{2n-1})e^{im(y_{2n-1}-t)}
%e^{-imt}
u(e^{it}  x')dtd\eta_{2n-1}dy_{2n-1}.
\end{split}
\end{equation}
From \eqref{e-gue131217IV} and \eqref{e-gue131217I}, we have
\begin{equation}\label{e-gue131217V}
\begin{split}
&(B_{k}+R_{k})u(x)\\
&=\frac{1}{(2\pi)^2}\sum_{m\in\mathbb Z}
\int_{\abs{t}\leq\pi}e^{i\langle x_{2n-1}-y_{2n-1},\eta_{2n-1}\rangle}\tau_\delta\Big(\frac{\eta_{2n-1}}{k}\Big)e^{imy_{2n-1}}e^{-imt}
u(e^{it}x')dtd\eta_{2n-1}dy_{2n-1}.
\end{split}\end{equation}
Note that the following formula holds for every $m\in\mathbb Z$,
\begin{equation}\label{dm}
\int e^{imy_{2n-1}}e^{-iy_{2n-1}\eta_{2n-1}}dy_{2n-1}=2\pi\delta_m(\eta_{2n-1}),
\end{equation}
where the integral is defined as an oscillatory integral and $\delta_m$ is the Dirac measure at $m$.
Using \eqref{e-gue131217}, \eqref{dm} and the Fourier inversion formula,
 \eqref{e-gue131217V} becomes
\begin{equation}\label{e-gue131217VI}
\begin{split}
(B_{k}+R_k)u(x)
=\frac{1}{2\pi}\sum_{m\in\mathbb Z}\tau_\delta\Big(\frac{m}{k}\Big)e^{ix_{2n-1}m}
\int_{\abs{t}\leq\pi}e^{-imt}u(e^{it} x')dt
=F_{k, \delta,s}u(x).
\end{split}\end{equation}
From \eqref{e-gue131217VI}, the claim \eqref{e-gue131217III} follows.
From \eqref{e-gue131217III} and \eqref{e-gue131217II}, the lemma follows.
\end{proof}

From Lemma~\ref{l-gue131209}, we see that the condition {\rm (III)\,} in Theorem~\ref{t-gue150811} holds.

\begin{lem}\label{l-gue150813}
Let $D\subset X$ be a canonical coordinate patch of $X$.
Then, $\Pi_{k,\leq k\delta}$ is $k$-negligible away the diagonal with respect to $F_{k,\delta}$ on $D$.
\end{lem}

\begin{proof}
Let $\chi, \chi_1\in C^\infty_0(D)$, $\chi_1=1$ on some neighbourhood of
${\rm Supp\,}\chi$. Let $u\in\mathcal{H}^0_{b,\leq k\delta}(X,L^k)$ with
$\norm{u}=1$. It is well-known that (see Theorem 2.4 in~\cite{HL15})
there is a constant $C>0$ independent of $k$ and $u$ such that
\begin{equation}\label{e-gue150813}
\abs{u(x)}^2_{h^k}\leq Ck^n,\ \ \forall x\in X.
\end{equation}
Let $x=(x_1,\ldots,x_{2n-1})=(x',x_{2n-1})$ be canonical coordinates on $D$. Put $v=(1-\chi_1)u$. It is straightforward to see that on $D$,
\begin{equation}\label{e-gue150813I}
\begin{split}
&(2\pi)^2\chi F_{k,\delta}(1-\chi_1)u(x)\\
&=\sum_{\substack{m\in\mathbb Z\\ \abs{m}\leq 2k\delta}}
\int_{\abs{t}\leq\pi}
e^{i\langle x_{2n-1}-y_{2n-1},\eta_{2n-1}\rangle}
\chi(x)\tau_\delta\Big(\frac{\eta_{2n-1}}{k}\Big)e^{im(y_{2n-1}-t)}
v(e^{it}  x')dt\,d\eta_{2n-1}\,dy_{2n-1}.
\end{split}
\end{equation}
%----------
Let $\varepsilon>0$ be a small constant so that for every $(x_1,\ldots,x_{2n-1})\in{\rm Supp\,}\chi$, we have
\begin{equation}\label{e-gue150813II}
(x_1,\ldots,x_{2n-2},y_{2n-1})\in\set{x\in D;\, \chi_1(x)=1},\ \ \forall\abs{y_{2n-1}-x_{2n-1}}<\varepsilon.
\end{equation}
Let $\psi\in C^\infty_0((-1,1))$, $\psi=1$ on $\big[-\frac{1}{2},\frac{1}{2}\big]$. Put
\begin{equation}\label{e-gue150813III}
\begin{split}
I_0(x)=\frac{1}{(2\pi)^2}\sum_{m\in\mathbb Z,\abs{m}\leq 2k\delta}&\int_{\abs{t}\leq\pi}e^{i\langle x_{2n-1}-y_{2n-1},\eta_{2n-1}\rangle}\Big(1-\psi\Big(\frac{x_{2n-1}-y_{2n-1}}{\varepsilon}\Big)\Big)\chi(x)
\tau_\delta\Big(\frac{\eta_{2n-1}}{k}\Big)e^{imy_{2n-1}}\\
&\quad\quad\times e^{-imt}v(e^{it}  x')dtd\eta_{2n-1}dy_{2n-1},
\end{split}
\end{equation}
\begin{equation}\label{e-gue150813IV}
\begin{split}
I_1(x)=\frac{1}{(2\pi)^2}\sum_{m\in\mathbb Z}&\int_{\abs{t}\leq\pi}e^{i\langle x_{2n-1}-y_{2n-1},\eta_{2n-1}>}\psi\Big(\frac{x_{2n-1}-y_{2n-1}}{\varepsilon}\Big)\chi(x)
\tau_\delta\Big(\frac{\eta_{2n-1}}{k}\Big)e^{imy_{2n-1}}\\
&\quad\quad\times e^{-imt}v(e^{it}x')dtd\eta_{2n-1}dy_{2n-1},
\end{split}
\end{equation}
and
\begin{equation}\label{e-gue150813V}
\begin{split}
I_2(x)=\frac{1}{(2\pi)^2}\sum_{m\in\mathbb Z,\abs{m}>2k\delta}&\int_{\abs{t}\leq\pi}e^{i\langle x_{2n-1}-y_{2n-1},\eta_{2n-1}\rangle}\psi\Big(\frac{x_{2n-1}-y_{2n-1}}{\varepsilon}\Big)\chi(x)
\tau_\delta\Big(\frac{\eta_{2n-1}}{k}\Big)e^{imy_{2n-1}}\\
&\quad\quad\times e^{-imt}v(e^{it}  x')dtd\eta_{2n-1}dy_{2n-1}.
\end{split}
\end{equation}
It is clear that on $D$,
\begin{equation}\label{e-gue150813VI}
\chi F_{k,\delta}(1-\chi_1)u(x)=I_0(x)+I_1(x)-I_2(x).
\end{equation}
By using integration by parts with respect to $\eta_{2n-1}$ several times and \eqref{e-gue150813}, we conclude that for every $N>0$ and $m\in\mathbb N$, there is a constant $C_{N,m}>0$ independent of $u$ and $k$ such that
%We can integrate by parts with respect to $\eta_{2n-1}$ several times and conclude that
\begin{equation}\label{e-gue150813VII}
\norm{I_0(x)}_{C^m(D)}\leq C_{N,m}k^{-N}.
\end{equation}
%\begin{equation}\label{e-gue150813VII}
%\mbox{$I_0(x)\equiv0\mod O(k^{-\infty})$ on $D$}.
%\end{equation}
Similarly, by using integration by parts with respect to $y_{2n-1}$ several times and \eqref{e-gue150813}, we conclude that for every $N>0$ and $m\in\mathbb N$, there is a constant $\Td C_{N,m}>0$ independent of $u$ and $k$ such that
\begin{equation}\label{e-gue150813VIII}
\norm{I_2(x)}_{C^m(D)}\leq\Td C_{N,m}k^{-N}.
\end{equation}
%Similarly, we can integrate by parts with respect to $y_{2n-1}$ several times and conclude that
%\begin{equation}\label{e-gue150813VIII}
%\mbox{$I_2(x)\equiv0\mod O(k^{-\infty})$ on $D$}.
%\end{equation}
We can check that
\begin{equation}\label{e-gue150813aVIII}
\begin{split}
I_1(x)=\frac{1}{2\pi}\int e^{i\langle x_{2n-1}-y_{2n-1},\eta_{2n-1}\rangle}\psi\Big(\frac{x_{2n-1}-y_{2n-1}}{\varepsilon}\Big)
\chi(x)\tau_\delta\Big(\frac{\eta_{2n-1}}{k}\Big)v(x',y_{2n-1})d\eta_{2n-1}dy_{2n-1}.
\end{split}
\end{equation}
From \eqref{e-gue150813II} and \eqref{e-gue150813aVIII}, we deduce that
\begin{equation}\label{e-gue150813bVIII}
\mbox{$I_1(x)=0$ on $D$.}
\end{equation}
From \eqref{e-gue150813VI}, \eqref{e-gue150813VII}, \eqref{e-gue150813VIII} and \eqref{e-gue150813bVIII}, we conclude that for every $N>0$ and $m\in\mathbb N$, there is a constant $\hat C_{N,m}>0$ independent of $u$ and $k$ such that
\begin{equation}\label{e-gue150813gVIII}
\norm{\chi F_{k,\delta}(1-\chi_1)u(x)}_{C^m(D)}\leq\hat C_{N,m}k^{-N}.
\end{equation}
%\begin{equation}\label{e-gue150813gVIII}
%\mbox{$\chi F^{(q)}_{\delta,k}(1-\chi_1)u(x)\equiv0\mod O(k^{-\infty})$ on $D$}.
%\end{equation}
From \eqref{e-gue150813} and \eqref{e-gue150813gVIII}, it is not difficult to see that
\begin{equation}\label{e-gue150813vgVIII}
\mbox{$\sum\limits^{d_k}_{j=1}\abs{\chi F_{k,\delta}(1-\chi_1)f_j(x)}^2_{h^k}=O(k^{-\infty})$ on $D$},
\end{equation}
where $\set{f_1,\ldots,f_{d_k}}$ is an orthonormal basis for 
$\mathcal{H}^0_{b,\leq k\delta}(X,L^k)$. From \eqref{e-gue150813vgVIII}, the lemma follows.
\end{proof}

From Theorem~\ref{t-gue150816}, Lemma~\ref{l-gue131209} and Lemma~\ref{l-gue150813},
we see that the conditions {\rm (I)\,}, {\rm (II)\,} and {\rm (III)\,}
in Theorem~\ref{t-gue150811} holds. The proof of Theorem~\ref{t-gue150807} is completed.
%----------
\begin{proof}[Proof of Corollary \ref{c-gue150811}] We use the same notations as in Theorem \ref{t-gue150807}.
On the diagonal $x=y$, by (\ref{e-gue150807b}) we have $\varphi(x, x, t)=0$. From (\ref{sk2}) we have
\begin{equation}
P_{k, \delta}(x)=P_{k, \delta, s}(x, x)=\int g(x, x, t, k)dt+O(k^{-\infty})~\text{on}~D_0.
\end{equation}
Combining with \eqref{e-gue150807bI}, there exist $b_j(x)\in C^\infty(D_0)$,  $j\in\N_0$, such that
\begin{equation}\label{ee-160318}
P_{k, \delta}(x)=P_{k, \delta, s}(x, x)\sim\sum_{j=0}^{\infty}k^{n-j}b_j(x)
\:\:\text{in}~S^n_{\rm loc}(1; D_0).
\end{equation}
Let $D_1$ be another canonical coordinate neighborhood and $s_1$ 
be another local rigid CR frame of $L$ on $D_1$. 
Then from Theorem \ref{t-gue150807} and the above argument, on $D_1$ we have
\begin{equation}\label{ee1-160318}
P_{k, \delta}(x)=P_{k, \delta, s_1}(x, x)\sim\sum_{j=0}^{\infty}k^{n-j}\hat b_j(x)
\:\:\text{in}~S^n_{\rm loc}(1; D_1),
\end{equation}
where $\hat b_j(x)\in C^\infty(D_1)$,  $j\in\N_0$.
Since on $D\cap D_1$, we have $P_{k, \delta, s}(x, x)=P_{k, \delta, s_1}(x, x)=P_{k, \delta}(x)$, 
\eqref{ee-160318} and \eqref{ee1-160318} yield $b_j(x)=\hat b_j(x)$ on $D\cap D_1$, for all $j\in\N_0$. 
Hence, $b_j(x)\in C^\infty(X)$, for all $j\in\N_0$, and we get the conclusion of Corollary \ref{c-gue150811}.
\end{proof}

%---------

\subsection{Properties of the phase function}\label{s-gue150808y}

In this section, we collect some properties of the phase $\varphi$ in Theorem~\ref{t-gue150807}. We will use the same notations as in Theorem~\ref{t-gue150807}.

In view of \eqref{e-gue150807b}, we see that ${\rm Im\,}\varphi(x,y,t)\geq0$. Moreover, we can estimate ${\rm Im\,}\varphi(x,y,t)$ in some local coordinates.
%--------------
\begin{thm}\label{t-gue140121I}
With the assumptions and notations used in Theorem~\ref{t-gue150807}, fix $p\in D$. We take canonical coordinates $x=(x_1,\ldots,x_{2n-1})$ defined in a small neighbourhood of $p$ so that $x(p)=0$, $\omega_0(p)=-dx_{2n-1}$ and $T^{1,0}_pX\oplus T^{0,1}_pX=\set{\sum^{2n-2}_{j=1}a_j\frac{\pr}{\pr x_j};\, a_j\in\mathbb C, j=1,\ldots,2n-2}$. If $D$ is small enough, then there is a constant $c>0$ such that for $(x,y,t)\in D\times D\times(-\delta,\delta)$,
\begin{equation}\label{e-gue140121II}\begin{split}
&{\rm Im\,}\varphi(x,y,t)\geq c\abs{x'-y'}^2,\ \ \forall (x,y,t)\in D\times D\times(-\delta,\delta),\\
&{\rm Im\,}\varphi(x,y,t)+\abs{\frac{\pr\varphi}{\pr t}(x,y,t)}^2\geq c\bigr(\abs{x_{2n-1}-y_{2n-1}}^2+\abs{x'-y'}^2\bigr),
\end{split}\end{equation}
where $x'=(x_1,\ldots,x_{2n-2})$, $y'=(y_1,\ldots,y_{2n-2})$, $\abs{x'-y'}^2=\sum^{2n-2}_{j=1}\abs{x_j-y_j}^2$.
\end{thm}
%--------------
For the proof of Theorem~\ref{t-gue140121I}, we refer the reader to the proof of Theorem 4.24 in~\cite{Hsiao14}.

In Section 4.4 of~\cite{Hsiao14}, the first author determined the tangential Hessian of $\varphi(x,y,t)$.
We denote as usual $x=(x_1,\ldots,x_{2n-1})=(z,\theta)$, $z_j=x_{2j-1}+ix_{2j}$, $j=1,\ldots,n-1$,
canonical local coordinates of $X$. We also use
$y=(y_1,\ldots,y_{2n-1})$, $w_j=y_{2j-1}+iy_{2j}$, $j=1,\ldots,n-1$.
%--------------
\begin{thm}\label{t-gue140121II}
With the assumptions and notations used in Theorem~\ref{t-gue150807}, fix $(p,p,t_0)\in D\times D\times(-\delta,\delta)$, and let $\ol Z_{1,t_0},\ldots,\ol Z_{n-1,t_0}$ be an orthonormal rigid frame of $T^{1,0}_xX$ varying smoothly with $x$ in a neighbourhood of $p$, for which the Hermitian quadratic form $R^L_x-2t_0\mathcal{L}_x$ is diagonalized at $p$.
That is,
\[R^L_p\bigr(\ol Z_{j,t_0}(p),Z_{k,t_0}(p)\bigr)-2t_0\mathcal{L}_p\bigr(\ol Z_{j,t_0}(p),Z_{k,t_0}(p)\bigr)
=\lambda_j(t_0)\delta_{j,k},\ \ j,k=1,\ldots,n-1,\]
where $\lambda_j(t_0)>0$, $j=1,\ldots,n-1$.  Let $s$ be a local rigid CR frame of $L$ defined in
some small neighbourhood of $p$ such that
\begin{equation}\label{e-geusw13623}
\begin{split}
&x(p)=0,\ \ \omega_0(p)=-dx_{2n-1},\ \ T=\frac{\pr}{\pr x_{2n-1}}=\frac{\pr}{\pr\theta},\\
&\Big\langle\,\frac{\pr}{\pr x_j}(p)\,|\,\frac{\pr}{\pr x_k}(p)\,\Big\rangle=2\delta_{j,k},\ \ j,k=1,\ldots,2n-2,\\
&\ol Z_{j,t_0}(p)=\frac{\pr}{\pr z_j}+i\sum^{n-1}_{k=1}\tau_{j,k}\ol z_k\frac{\pr}{\pr x_{2n-1}}+O(\abs{z}^2),\ \ j=1,\ldots,n-1,\\
&\Phi(x)=\frac{1}{2}\sum^{n-1}_{l,k=1}\mu_{k,l}z_k\ol z_l+\sum^{n-1}_{l,k=1}\bigr(a_{l,k}z_lz_k+\ol a_{l,k}\ol z_l\ol z_k\bigr)+O(\abs{z}^3),
\end{split}
\end{equation}
where $\tau_{j,k}, \mu_{j,k}, a_{j,k}\in\Complex$, $\mu_{j,k}=\ol\mu_{k,j}$, $j, k=1,\ldots,n-1$.
Then there exists a neighbourhood of $(p,p)$
such that
\begin{equation}\label{e-guew13627}
\begin{split}
\varphi(x,y,t_0)&=t_0(-x_{2n-1}+y_{2n-1})
-\frac{i}{2}\sum^{n-1}_{j,l=1}(a_{l,j}+a_{j,l})(z_jz_l-w_jw_l)\\
&\quad+\frac{i}{2}\sum^{n-1}_{j,l=1}(\ol a_{l,j}+\ol a_{j,l})(\ol z_j\ol z_l-\ol w_j\ol w_l)+\frac{it_0}{2}\sum^{n-1}_{j,l=1}(\ol\tau_{l,j}-\tau_{j,l})(z_j\ol z_l-w_j\ol w_l)\\
&\quad-\frac{i}{2}\sum^{n-1}_{j=1}\lambda_j(t_0)(z_j\ol w_j-\ol z_jw_j)+\frac{i}{2}\sum^{n-1}_{j=1}\lambda_j(t_0)\abs{z_j-w_j}^2\\
&\quad+(-x_{2n-1}+y_{2n-1})f(x,y,t_0)+O(\abs{(x, y)}^3),
%&\quad\quad\quad f\in C^\infty,\ \  f(0,0)=0.
\end{split}
\end{equation}
where $f$ is smooth in a neighborhood of $(p,p,t_0)$ and $f(0,0,t_0)=0$
\end{thm}
%--------------

\section{Equivariant Kodaira embedding}\label{s-gue150819}

In this section, we will prove Theorem~\ref{t-gue150807I}.
Let
\[f_1\in\mathcal{H}^0_{b,\leq k\delta}(X,L^k),\ldots,
f_{d_k}\in\mathcal{H}^0_{b,\leq k\delta}(X,L^k)\]
be an orthonormal basis for $\mathcal{H}^0_{b,\leq k\delta}(X,L^k)$
with respect to $(\,\cdot\,|\,\cdot\,)$. On $D$, we write
\[F_{k,\delta}f_j=s^k\otimes\Td g_j,\ \ \Td g_j\in C^\infty(D),\ \ j=1,2,\ldots,d_{k}.\]
\begin{lem}\label{L:sk}
We have
\begin{equation}\label{e-gue150806VIIa}
\begin{split}
&P_{k,\delta,s}(x,y)=\sum^{d_{k}}_{j=1}e^{-k\Phi(x)}\Td g_j(x)\ol{\Td g_j(y)}e^{-k\Phi(y)},\\
&P_{k,\delta,s}(x,x)=\sum^{d_{k}}_{j=1}\abs{\Td g_j(x)}^2e^{-2k\Phi(x)}
=\sum^{d_{k}}_{j=1}\big|(F_{k,\delta}f_j)(x)\big|^2_{h^k}.
\end{split}
\end{equation}
In particular, \eqref{sk1} holds.
\end{lem}

%In this section, we will prove Theorem~\ref{t-gue150807I}. Let
%$f_1,\ldots,f_{d_k}$
%\[f_1\in\mathcal{H}^0_{b,\leq k\delta}(X,L^k),\ldots,f_{d_k}\in\mathcal{H}^0_{b,\leq k\delta}(X,L^k)\]
%be an orthonormal basis of $\mathcal{H}^0_{b,\leq k\delta}(X,L^k)$ with respect to $(\,\cdot\,|\,\cdot\,)$. 
From Theorem~\ref{t-gue150807}, we have

\begin{lem}\label{l-gue131007I}
Let $\delta>0$ be a small constant.
Then there exist constants $C_0>0$ and $k_0\in\N$ such that
for all $k\geq k_0$ and all $x\in X$ we have
\begin{equation}\label{e-gue131007aI}
\sum^{d_k}_{j=1}\big|F_{k,\delta}f_j(x)\big|^2_{h^k}\geq C_0k^n.
\end{equation}
Moreover, there is a constant $c_0>0$ and $k_0\in\N$ such that for $k\geq k_0$ and
every $x\in X$, there exists $j_0\in\set{1,2,\ldots,d_k}$ with
\begin{equation}\label{e-gue131019}
\big|F_{k,\delta}f_{j_0}(x)\big|^2_{h^k}\geq c_0.
\end{equation}
\end{lem}

\begin{proof}
Theorem~\ref{t-gue150807} immediately implies the first assertion. %\eqref{e-gue131007aI}.
We only need to prove the second. %\eqref{e-gue131019}.
It is well known~\cite[Theorem 1.4]{HL15} that there is a constant $C_1>0$ such that
\begin{equation}\label{e-gue131019I}
{\rm dim\,}\mathcal{H}^0_{b,\leq k\delta}(X,L^k)=d_k\leq C_1k^n,
\end{equation}
where $C_1>0$ is a constant independent of $k$.
From \eqref{e-gue131019I} and \eqref{e-gue131007aI}, we have for every $x\in X$,
\[\begin{split}
&C_1k^n{\sup\,}\set{\abs{F_{k,\delta}f_j(x)}^2_{h^k};\, j=1,2,\ldots,d_k}\\
&\geq d_k{\sup\,}\set{\abs{F_{k,\delta}f_j(x)}^2_{h^k};\, j=1,2,\ldots,d_k}\\
&\geq\sum^{d_k}_{j=1}\abs{F_{k,\delta}f_j(x)}^2_{h^k}\geq C_0k^n.\end{split}\]
From this, \eqref{e-gue131019} follows.
\end{proof}

The (modified) Kodaira map $\Phi_{k,\delta}:X\To\Complex\mathbb P^{d_k-1}$ introduced in
\eqref{e-gue150807h} is explicitly defined as follows.
For $x_0\in X$, let $s$ be a local rigid CR frame of $L$ on an open neighbourhood $D\subset X$ of $x_0$, $\abs{s(x)}^2_{h}=e^{-2\Phi}$. On $D$, put $F_{k,\delta}f_j(x)=s^k\Td f_j(x)$, $\Td f_j(x)\in C^\infty(D)$, $j=1,\ldots,d_k$. Then,
\begin{equation}\label{e-gue131019II}
\Phi_{k,\delta}(x_0)=[\Td f_1(x_0),\ldots,\Td f_{d_k}(x_0)]\in\Complex\mathbb P^{d_k-1}.
\end{equation}
In view of \eqref{e-gue131019}, we see that $\Phi_{k,\delta}$ is well-defined as a
smooth CR map from $X$ to $\Complex\mathbb P^{d_k-1}$.
We wish to prove that $\Phi_{k,\delta}$ is an embedding for $k$ large enough.
Since $X$ is compact, a smooth map is an embedding if and only if
it is an injective immersion.
\begin{thm}\label{t-gue131019}
The map $\Phi_{k,\delta}$ is an immersion for $k$ large enough.
%\[d\Phi_{k,\delta}(x):T_xX\To T_{\Phi_{k,\delta}(x)}\Complex\mathbb P^{d_k-1}\]
%is injective, for every $x\in X$.
\end{thm}

To prove Theorem~\ref{t-gue131019}, we need some preparations. Fix $p\in X$ and let $s$ be a local rigid CR frame of $L$ on a canonical coordinate patch $D$, $p\in D$, $\abs{s}^2_{h}=e^{-2\Phi}$, with canonical local coordinates $x=(x_1,\ldots,x_{2n-1})=(z,\theta)$. We take canonical coordinates $x$ and $s$ so that \eqref{e-geusw13623} hold.
We identify $D$ with an open set in $\Real^{2n-1}$. For $r>0$, put
\[
D_r:=\set{x\in\Real^{2n-1};\, \abs{x_j}<r, j=1,2,\ldots,2n-1}.
\]
For $x=(x_1,\ldots,x_{2n-1})$, put
\[F^*_kx:=\left(\frac{x_1}{\sqrt{k}},\ldots,\frac{x_{2n-2}}{\sqrt{k}},\frac{x_{2n-1}}{k}\right).\]
%Put
%\begin{equation}\label{e-gue150821}
%R=R(x)=R(z)=\sum^{n-1}_{j=1}\Bigr(\alpha_jz_j-\ol\alpha_j\ol z_j\Bigr),
%\end{equation}
%where $\alpha_1,\ldots,\alpha_{n-1}\in\Complex$ are as in \eqref{e-geusw13623}.
From \eqref{e-guew13627}, we can check that
\begin{equation}\label{e-gue151009f}
ki\varphi(0,F^*_ky,t)=iy_{2n-1}t+i\varphi_0(w,t)+r_k(y,t)\ \ \mbox{on $D_{\log k}$},
\end{equation}
where $r_k(y,t)\in C^\infty(D_{\log k}\times(-\delta,\delta))$ and
$\varphi_0(w,t)\in C^\infty(\Real^{2n-2}\times(-\delta,\delta))$ independent of $k$.
Moreover,
for every $\alpha\in\mathbb N^{2n-1}$, we have
\begin{equation}\label{e-gue151009fI}
\lim_{k\To\infty}\sup_{(y,t)\in D_{\log k}\times(-\delta,\delta)}\abs{\pr^\alpha_yr_k(y,t)}=0,
\end{equation}
%$w=(w_1,\ldots,w_{n-1})=(y_1,\ldots,y_{2n-2})$, $w_j=y_{2j-1}+iy_{2j}$, $j=1,\ldots,n-1$,
%$\varphi_0(w,t)\in C^\infty(\Real^{2n-2}\times(-\delta,\delta))$,
and there exists a constant $C>0$ such that
\begin{equation}\label{e-gue151009fII}
\begin{split}
&\varphi_0(0,t)=0,\ \ \forall t\in(-\delta,\delta),\\
&\varphi_0(w,t)=\varphi_0(-w,t),\ \ \forall (w,t)\in\Real^{2n-2}\times(-\delta,\delta),\\
&\int e^{-{\rm Im\,}\varphi_0(w,t)}dy_1\ldots dy_{2n-2}\leq C<\infty,\ \ \forall t\in(-\delta,\delta).
\end{split}
\end{equation}
Take $\chi\in C^\infty_0((-1,1))$ with $0\leq\chi\leq1$, $\chi(x)=1$ on $[-\frac{1}{2},\frac{1}{2}]$
and $\chi(t)=\chi(-t)$, for every $t\in\Real$. For $\epsilon>0$,
put $\chi_\epsilon(x):=\epsilon^{-1}\chi(\epsilon^{-1}x)$.
Let $g_0(x,y,t)\in C^\infty(D\times D\times(-\delta,\delta))$ be as \eqref{e-gue150807bI}. We can check that
\begin{equation}\label{e-gue151016}
\begin{split}
&\lim_{\epsilon\To0}\int e^{i\varphi_0(w,t)+iy_{2n-1}t}g_0(0,0,t)\chi_\epsilon(y_1)\ldots\chi_\epsilon(y_{2n-2})dydt\\
&=\int e^{iy_{2n-1}t}g_0(0,0,t)dtdy_{2n-1}\int_{\Real^{2n-2}}\chi(y_1)\ldots\chi(y_{2n-2})dy_1\ldots dy_{2n-2}\\
&=(2\pi)g_0(0,0,0)\int_{\Real^{2n-2}}\chi(y_1)\ldots\chi(y_{2n-2})dy_1\ldots dy_{2n-2}\neq0,
\end{split}
\end{equation}
\begin{equation}\label{e-gue151016I}
\begin{split}
&\lim_{\epsilon\To0}\int e^{i\varphi_0(w,t)+iy_{2n-1}t}g_0(0,0,t)\epsilon^{-2}\abs{y_j}^2\chi_\epsilon(y_1)\ldots\chi_\epsilon(y_{2n-1})dydt\\
&=\int e^{iy_{2n-1}t}g_0(0,0,t)dtdy_{2n-1}\int_{\Real^{2n-2}}\abs{y_j}^2\chi(y_1)\ldots\chi(y_{2n-2})dy_1\ldots dy_{2n-2}\\
&=(2\pi)g_0(0,0,0)\int_{\Real^{2n-2}}\abs{y_j}^2\chi(y_1)\ldots\chi(y_{2n-2})dy_1\ldots dy_{2n-2}\neq0,\ \ j=1,2,\ldots,2n-2,
\end{split}
\end{equation}
and
\begin{equation}\label{e-gue151016II}
\begin{split}
&\lim_{\epsilon\To0}\int e^{i\varphi_0(w,t)+iy_{2n-1}t}(-it)g_0(0,0,t)\chi_\epsilon(y_1)\ldots\chi_\epsilon(y_{2n-1})y_{2n-1}dydt\\
&=\int e^{iy_{2n-1}t}(-ity_{2n-1})g_0(0,0,t)dtdy_{2n-1}\int_{\Real^{2n-2}}\chi(y_1)\ldots\chi(y_{2n-2})dy_1\ldots dy_{2n-2}\\
&=(2\pi)g_0(0,0,0)\int_{\Real^{2n-2}}\chi(y_1)\ldots\chi(y_{2n-2})dy_1\ldots dy_{2n-2}\neq0.
\end{split}
\end{equation}
%\begin{equation}\label{e-gue151013}
%\begin{split}
%&\lim_{\epsilon\To0}\int e^{i\varphi_0(w,t)+iy_{2n-1}t}g_0(0,0,t)\chi_\epsilon(y_1)\ldots\chi_\epsilon(y_{2n-1})dydt=\int g_0(0,0,t)\chi(y_1)\ldots\chi(y_{2n-1})dydt\neq0,\\
%&\lim_{\epsilon\To0}\int e^{i\varphi_0(w,t)+iy_{2n-1}t}g_0(0,0,t)\epsilon^{-2}\abs{y_j}^2\chi_\epsilon(y_1)\ldots\chi_\epsilon(y_{2n-1})dydt\\
%&=\int g_0(0,0,t)\chi(y_1)\ldots\chi(y_{2n-1})dydt\neq0,\ \ j=1,2,\ldots,2n-2
%\end{split}
%\end{equation}
%and
%\begin{equation}\label{e-gue151013I}
%\begin{split}
%&\lim_{M\To\infty}\lim_{\epsilon\To0}\int e^{i\varphi_0(w,t)+iy_{2n-1}t}-itg_0(0,0,t)y_{2n-1}\chi_\epsilon(y_1)\ldots\chi_\epsilon(y_{2n-2})\chi(\frac{y_{2n-1}}{M})dydt\\
%&=\lim_{M\To\infty}\int e^{iy_{2n-1}t}-itg_0(0,0,t)y_{2n-1}\chi(\frac{y_{2n-1}}{M})dydt\\
%&=\lim_{M\To\infty}\int e^{iy_{2n-1}t}g_0(0,0,t)\chi(\frac{y_{2n-1}}{M})dydt+\lim_{M\To\infty}\int e^{iy_{2n-1}t}g_0(0,y,t)\frac{1}{M}\chi'(\frac{y_{2n-1}}{M})dydt\\
%&=\int_{\Real}e^{iy_{2n-1}t}g_0(0,0,t)dtdy_{2n-1}=(2\pi)g_0(0,0,0)\neq0.
%\end{split}
%\end{equation}
From \eqref{e-gue151016}, \eqref{e-gue151016I} and \eqref{e-gue151016II}, we deduce that there is a $\epsilon_0>0$, $\epsilon_0$ small, such that
\begin{equation}\label{e-gue151016III}
\begin{split}
&\int e^{i\varphi_0(w,t)+iy_{2n-1}t}g_0(0,0,t)\chi_{\epsilon_0}(y_1)\ldots\chi_{\epsilon_0}(y_{2n-2})dydt=:V_0\neq0,\\
&\int e^{i\varphi_0(w,t)+iy_{2n-1}t}g_0(0,0,t)\abs{y_j}^2\chi_{\epsilon_0}(y_1)\ldots\chi_{\epsilon_0}(y_{2n-2})dydt=:V_j\neq0,\ \ j=1,2,\ldots,2n-2,\\
&\int e^{i\varphi_0(w,t)+iy_{2n-1}t}(-ity_{2n-1})g_0(0,0,t)\chi_{\epsilon_0}(y_1)\ldots\chi_{\epsilon_0}(y_{2n-2})dydt=:V\neq0.
\end{split}
\end{equation}

%Put
%\begin{equation}\label{e-gue151009b}
%\alpha_k(y):=\int e^{ik\varphi(0,y,t)+kR(w)}g(0,y,t,k)dt.
%\end{equation}

Assume that $p\in D_0\Subset D$. We need

\begin{lem}\label{l-gue151009}
With the notations above, there is a $K_0>0$ independent of the point $p$ such that for all $k\geq K_0$, we can find
\[g^j_k\in\mathcal{H}^0_{b,\leq k\delta}(X,L^k),\ \ j=1,\ldots,n,\]
such that if we put $F_{k,\delta}g^j_k=s^k\Td g^j_k$ on $D$, $j=1,\ldots,n$, then
\begin{equation}\label{e-gue151009}
\begin{split}
&\sum^n_{j=1}\abs{(e^{-k\Phi}\Td g^j_k)(p)}=0,\\
&\abs{\frac{1}{\sqrt{k}}\pr_{\ol z_t}\bigr(e^{-k\Phi}\Td g^j_k\bigr)(0)}\leq\epsilon_k,\ \ j=1,2,\ldots,n,\ \ t=1,2,\ldots,n-1,\\
%&\abs{\frac{1}{\sqrt{k}}\pr_{x_{2n-1}}\bigr(e^{-kR}e^{-k\Phi}\Td g^j_k\bigr)(0)}\leq\epsilon_k,\ \ j=1,2,\ldots,n-1,\\
&\abs{\frac{1}{\sqrt{k}}\pr_{z_t}\bigr(e^{-k\Phi}\Td g^j_k\bigr)(0)}\leq\epsilon_k,\ \ j=1,2,\ldots,n,\ \ t=1,2,\ldots,n-1,\ \  j\neq t,\\
&\abs{\frac{1}{\sqrt{k}}\pr_{z_j}\bigr(e^{-k\Phi}\Td g^j_k\bigr)(0)}\geq C_0,\ \ j=1,\ldots,n-1,\\
&\abs{\frac{1}{k}\pr_{x_{2n-1}}\bigr(e^{-k\Phi}\Td g^n_k\bigr)(0)}\geq C_0,
\end{split}
\end{equation}
where $C_0>0$ is a constant independent of $k$ and the point $p$, $\epsilon_k$ is a sequence independent of $p$ with $\lim_{k\To\infty}\epsilon_k=0$.
%and for every $N>0$, there is a constant $C_N>0$ independent of $k$ and the point $p$ such that
%\begin{equation}\label{e-gue151009I}
%\sum^n_{j=1}\abs{(e^{-kR}e^{-k\Phi}\Td g^j_k)(p)}\leq C_Nk^{-N}.
%\end{equation}
\end{lem}

\begin{proof}
Let $\chi\in C^\infty_0(\Real)$ and $\epsilon_0>0$ be as in \eqref{e-gue151016III}.  Put
\begin{equation}\label{e-gue151016f}
u_k:=\Pi_{k,\leq k\delta}F_{k,\delta}\left(e^{k\Phi(w)}\chi_{\epsilon_0}(\sqrt{k}y_1)\ldots
\chi_{\epsilon_0}(\sqrt{k}y_{2n-2})\chi\Big(\frac{k}{\log k}y_{2n-1}\Big)\frac{1}{m(y)}\right),
\end{equation}
where $w=(w_1,\ldots,w_{n-1})$, $w_j=y_{2j-1}+iy_{2j}$, $j=1,\ldots,n-1$, and $m(y)dy=dv_X(y)$ on $D$. We put $F_{k,\delta}u_k=s^k\Td u_k$ on $D$.
In view of Theorem~\ref{t-gue150807}, we see that on $D_0$,
\begin{equation}\label{e-gue151009bI}
%\begin{split}
e^{-k\Phi(z)}\Td u_k(x)=\int e^{ik\varphi(x,y,t)}g(x,y,t,k)\chi_{\epsilon_0}(\sqrt{k}y_1)\ldots\chi_{\epsilon_0}(\sqrt{k}y_{2n-2})\chi(\frac{k}{\log k}y_{2n-1})dydt+O(k^{-\infty}).
%\end{split}
\end{equation}
From \eqref{e-gue151009f}, \eqref{e-gue151009fI}, \eqref{e-gue151009bI} and \eqref{e-gue151016III}, we can check that
\begin{equation}\label{e-gue151009bII}
\begin{split}
&\lim_{k\To\infty}\int e^{ik\varphi(0,y,t)}g(0,y,t,k)\chi_{\epsilon_0}(\sqrt{k}y_1)\ldots\chi_{\epsilon_0}(\sqrt{k}y_{2n-2})\chi(\frac{k}{\log k}y_{2n-1})dydt\\
&=\lim_{k\To\infty}\int e^{ik\varphi(0,F^*_ky,t)}k^{-n}g(0,F^*_ky,t,k)\chi_{\epsilon_0}(y_1)\ldots\chi_{\epsilon_0}(y_{2n-2})\chi(\frac{y_{2n-1}}{\log k})dtdy\\
&=\int e^{i\varphi_0(w,t)+iy_{2n-1}t}g_0(0,0,t)\chi_{\epsilon_0}(y_1)\ldots\chi_{\epsilon_0}(y_{2n-2})dydt=V_0\neq0.
\end{split}
\end{equation}
From \eqref{e-gue151009bII} and notice that $X$ is compact, it is easy to see that there is a $k_0>0$ independent of the point $p$ such that for all $k\geq k_0$, we have
\begin{equation}\label{e-gue151009bIII}
\begin{split}
\frac{1}{A_0}\leq\abs{e^{-k\Phi(0)}\Td u_k(0)}\leq A_0,
\end{split}
\end{equation}
where $A_0>1$ is a constant independent of $k$ and the point $p$. From now on, we assume that $k\geq k_0$. From \eqref{e-guew13627} and \eqref{e-gue151009bI}, we can check that
\begin{equation}\label{e-gue151016fI}
\begin{split}
&\lim_{k\To\infty}\frac{1}{k}\pr_{x_{2n-1}}(e^{-k\Phi}\Td u_k)(0)=\int e^{i\varphi_0(w,t)+ity_{2n-1}}(-it)g_0(0,0,t)\chi_{\epsilon_0}(y_1)\ldots\chi_{\epsilon_0}(y_{2n-2})dydt=0,
\end{split}
\end{equation}
and for $j=1,\ldots,n-1$,
\begin{equation}\label{e-gue151017}
\begin{split}
&\lim_{k\To\infty}\frac{1}{\sqrt{k}}\pr_{z_j}(e^{-k\Phi}\Td u_k)(0)\\
&=\int e^{i\varphi_0(w,t)+ity_{2n-1}}(-i)\lambda_j(t)(y_{2j-1}-iy_{2j})g_0(0,0,t)\chi_{\epsilon_0}(y_1)\ldots\chi_{\epsilon_0}(y_{2n-2})dydt=0,\ \
\end{split}
\end{equation}
\begin{equation}\label{e-gue151017I}
\lim_{k\To\infty}\frac{1}{\sqrt{k}}\pr_{\ol{z}_j}(e^{-k\Phi}\Td u_k)(0)=0.
\end{equation}
From \eqref{e-gue151016fI}, \eqref{e-gue151017} and \eqref{e-gue151017I}, it is easy to see that
\begin{equation}\label{e-gue151017II}
\begin{split}
&\abs{\frac{1}{k}\pr_{x_{2n-1}}(e^{-k\Phi}\Td u_k)(0)}\leq\delta_k,\\
&\abs{\frac{1}{\sqrt{k}}\pr_{z_j}(e^{-k\Phi}\Td u_k)(0)}\leq\delta_k,\ \ j=1,\ldots,n-1,\\
&\abs{\frac{1}{\sqrt{k}}\pr_{\ol z_j}(e^{-k\Phi}\Td u_k)(0)}\leq\delta_k,\ \ j=1,\ldots,n-1,
\end{split}
\end{equation}
where $\delta_k$ is a sequence independent of the point $p$ with $\lim_{k\To\infty}\delta_k=0$.

Put
\begin{equation}\label{e-gue151017III}
v^n_k:=\Pi_{k,\leq k\delta}F_{k,\delta}\Bigr(e^{k\Phi(w)}ky_{2n-1}\chi_{\epsilon_0}(\sqrt{k}y_1)\ldots\chi_{\epsilon_0}(\sqrt{k}y_{2n-2})\chi(\frac{k}{\log k}y_{2n-1})\frac{1}{m(y)}\Bigr).
\end{equation}
We put $F_{k,\delta}v^n_k=s^k\Td v^n_k$ on $D$. In view of Theorem~\ref{t-gue150807}, we see that
\begin{equation}\label{e-gue151017IV}
\begin{split}
&e^{-k\Phi(z)}\Td v^n_k(x)\\
&=\int e^{ik\varphi(x,y,t)}g(x,y,t,k)\\
&\quad\quad\quad\times ky_{2n-1}\chi_{\epsilon_0}(\sqrt{k}y_1)\ldots\chi_{\epsilon_0}(\sqrt{k}y_{2n-2})\chi(\frac{k}{\log k}y_{2n-1})dydt+O(k^{-\infty})\ \ \mbox{on $D_0$}.
\end{split}
\end{equation}
From \eqref{e-gue151017IV}, it is easy to see that there is a constant $E_0>0$ independent of $k$ and the point $p$ such that
\begin{equation}\label{e-gue151017V}
\abs{e^{-k\Phi(0)}\Td v^n_k(0)}\leq E_0.
\end{equation}
From \eqref{e-guew13627} and \eqref{e-gue151016III}, we can check that
\begin{equation}\label{e-gue151017VI}
\begin{split}
&\lim_{k\To\infty}\frac{1}{k}\pr_{x_{2n-1}}(e^{-k\Phi}\Td v^n_k)(0)\\
&=\int e^{i\varphi_0(w,t)+ity_{2n-1}}(-ity_{2n-1})g_0(0,0,t)\chi_{\epsilon_0}(y_1)\ldots\chi_{\epsilon_0}(y_{2n-2})dydt=V\neq0,
\end{split}
\end{equation}
and for $j=1,\ldots,n-1$,
\begin{equation}\label{e-gue151017VII}
\begin{split}
&\lim_{k\To\infty}\frac{1}{\sqrt{k}}\pr_{z_j}(e^{-k\Phi}\Td v^n_k)(0)\\
&=\int e^{i\varphi_0(w,t)+ity_{2n-1}}(-i)\lambda_j(t)(y_{2j-1}-iy_{2j})y_{2n-1}g_0(0,0,t)
\chi_{\epsilon_0}(y_1)\ldots\chi_{\epsilon_0}(y_{2n-2})dydt=0,
\end{split}
\end{equation}
\begin{equation}\label{e-gue151017VIII}
\lim_{k\To\infty}\frac{1}{\sqrt{k}}\pr_{\ol z_j}(e^{-k\Phi}\Td v^n_k)(0)=0.
\end{equation}
From \eqref{e-gue151017VI}, \eqref{e-gue151017VII} and \eqref{e-gue151017VIII}, it is easy to see that there is a $k_1>k_0$ independent of the point $p$ such that for all $k\geq k_1$, we have
\begin{equation}\label{e-gue151017IX}
\begin{split}
&\abs{\frac{1}{k}\pr_{x_{2n-1}}(e^{-k\Phi}\Td v^n_k)(0)}\geq B_0,\\
&\abs{\frac{1}{\sqrt{k}}\pr_{z_j}(e^{-k\Phi}\Td v^n_k)(0)}\leq\hat\delta_k,\ \ j=1,\ldots,n-1,\\
&\abs{\frac{1}{\sqrt{k}}\pr_{\ol z_j}(e^{-k\Phi}\Td v^n_k)(0)}\leq\hat\delta_k,\ \ j=1,\ldots,n-1,
\end{split}
\end{equation}
where $B_0>0$ is a constant independent of $k$ and the point $p$ and $\hat\delta_k$ is a sequence independent of the point $p$ with $\lim_{k\To\infty}\hat\delta_k=0$. From now on, we assume that $k\geq k_1$. Put
\begin{equation*}%\label{e-gue151017g}
g^n_k:=v^n_k-\frac{(e^{-k\Phi}\Td v^n_k)(0)}{(e^{-k\Phi}\Td u_k)(0)}u_k\in\mathcal{H}^0_{b,\leq k\delta}(X,L^k).
\end{equation*}
Put $F_{k,\delta}g^n_k=s^k\Td g^n_k$ on $D$. From \eqref{e-gue151009bIII}, \eqref{e-gue151017II}, \eqref{e-gue151017V} and \eqref{e-gue151017IX}, we see that there is a constant $k_1>0$ independent of $k$ and the point $p$ such that
\begin{equation}\label{e-gue151018}
\begin{split}
&\abs{\bigr(e^{-k\Phi}\Td g^n_k\bigr)(0)}=0,\\
&\abs{\frac{1}{\sqrt{k}}\pr_{\ol z_t}\bigr(e^{-k\Phi}\Td g^n_k\bigr)(0)}\leq\epsilon_k,\ \ t=1,\ldots,n-1,\\
%&\abs{\frac{1}{\sqrt{k}}\pr_{x_{2n-1}}\bigr(e^{-kR}e^{-k\Phi}\Td g^j_k\bigr)(0)}\leq\epsilon_k,\ \ j=1,2,\ldots,n-1,\\
&\abs{\frac{1}{\sqrt{k}}\pr_{z_t}\bigr(e^{-k\Phi}\Td g^n_k\bigr)(0)}\leq\epsilon_k,\ \ t=1,\ldots,n-1,\\
%&\abs{\frac{1}{\sqrt{k}}\pr_{z_j}\bigr(e^{-kR}e^{-k\Phi}\Td g^j_k\bigr)(0)}^2\geq C_0,\ \ j=1,\ldots,n-1,\\
&\abs{\frac{1}{k}\pr_{x_{2n-1}}\bigr(e^{-k\Phi}\Td g^n_k\bigr)(0)}\geq C_0,
\end{split}
\end{equation}
where $C_0>0$ is a constant independent of $k$ and the point $p$, $\epsilon_k$ is a sequence independent of $p$ with $\lim_{k\To\infty}\epsilon_k=0$.

Fix $j\in\{1,2,\ldots,n-1\}$. Put
\begin{equation*}%\label{e-gue150821II}
v^j_k:=\Pi_{k,\leq k\delta}F_{k,\delta}\left(e^{k\Phi(w)}\sqrt{k}(y_{2j-1}+iy_{2j})\chi_{\epsilon_0}(\sqrt{k}y_1)
\ldots\chi_{\epsilon_0}(\sqrt{k}y_{2n-2})\chi\Big(\frac{k}{\log k}y_{2n-1}\Big)\frac{1}{m(y)}\right).
\end{equation*}
We put $F_{k,\delta}v^j_k=s^k\Td v^j_k$ on $D$. In view of Theorem~\ref{t-gue150807}, we see that
\begin{equation}\label{e-gue151009bIg}
\begin{split}
e^{-k\Phi(z)}\Td v^j_k(x)&=\int e^{ik\varphi(x,y,t)}g(x,y,t,k)\sqrt{k}(y_{2j-1}+iy_{2j})\\
&\quad\quad\quad\times\chi_{\epsilon_0}(\sqrt{k}y_1)\ldots
\chi_{\epsilon_0}(\sqrt{k}y_{2n-2})\chi\Big(\frac{k}{\log k}y_{2n-1}\Big)dydt
+ O(k^{-\infty})\ \ \mbox{on $D_0$}.
\end{split}
\end{equation}
From \eqref{e-gue151009bIg}, it is easy to see that there is a constant $E_1>0$
independent of $k$ and the point $p$ such that
\begin{equation}\label{e-gue151018a}
\abs{e^{-k\Phi(0)}\Td v^j_k(0)}\leq E_1.
\end{equation}
Moreover, from \eqref{e-guew13627}, \eqref{e-gue151016III} and
\eqref{e-gue151009bIg}, we can repeat the proof of \eqref{e-gue151017IX}
with minor changes and deduce that there is a $\hat k_0>0$ independent of the point
$p$ such that for all $k\geq \hat k_0$, we have
\begin{equation}\label{e-gue151018aI}
\begin{split}
&\abs{\frac{1}{k}\pr_{x_{2n-1}}(e^{-k\Phi}\Td v^j_k)(0)}\leq\Td\delta_k,\\
&\abs{\frac{1}{\sqrt{k}}\pr_{z_j}(e^{-k\Phi}\Td v^j_k)(0)}\geq B_1,\\
&\abs{\frac{1}{\sqrt{k}}\pr_{z_t}(e^{-k\Phi}\Td v^j_k)(0)}\leq\Td\delta_k,\ \ j,t=1,\ldots,n-1,\ \ j\neq t,\\
&\abs{\frac{1}{\sqrt{k}}\pr_{\ol z_t}(e^{-k\Phi}\Td v^j_k)(0)}\leq\Td\delta_k,\ \ j,t=1,\ldots,n-1,
\end{split}
\end{equation}
where $B_1>0$ is a constant independent of $k$ and the point $p$ and
$\Td\delta_k$ is a sequence independent of the point $p$ with $\lim_{k\To\infty}\Td\delta_k=0$. Put
\begin{equation}\label{e-gue151017g}
g^j_k:=v^j_k-\frac{(e^{-k\Phi}\Td v^j_k)(0)}{(e^{-k\Phi}\Td u_k)(0)}u_k
\in\mathcal{H}^0_{b,\leq k\delta}(X,L^k).
\end{equation}
Put $F_{k,\delta}g^j_k=s^k\Td g^j_k$ on $D$. From \eqref{e-gue151009bIII},
\eqref{e-gue151017II}, \eqref{e-gue151018a} and \eqref{e-gue151018aI},
we see that there is a constant $\hat k_1>0$ independent of $k$ and the point $p$ such that
\begin{equation}\label{e-gue151018aII}
\begin{split}
&\abs{\bigr(e^{-k\Phi}\Td g^j_k\bigr)(0)}=0,\\
&\abs{\frac{1}{\sqrt{k}}\pr_{\ol z_t}\bigr(e^{-k\Phi}\Td g^j_k\bigr)(0)}\leq\epsilon_k,\ \ t=1,\ldots,n-1,\\
%&\abs{\frac{1}{\sqrt{k}}\pr_{x_{2n-1}}\bigr(e^{-kR}e^{-k\Phi}\Td g^j_k\bigr)(0)}\leq\epsilon_k,\ \ j=1,2,\ldots,n-1,\\
&\abs{\frac{1}{\sqrt{k}}\pr_{z_t}\bigr(e^{-k\Phi}\Td g^j_k\bigr)(0)}\leq\epsilon_k,\ \ t=1,\ldots,n-1,\ \ t\neq j,\\
&\abs{\frac{1}{\sqrt{k}}\pr_{z_j}\bigr(e^{-k\Phi}\Td g^j_k\bigr)(0)}\geq C_0,\\
&\abs{\frac{1}{k}\pr_{x_{2n-1}}\bigr(e^{-k\Phi}\Td g^j_k\bigr)(0)}\leq\epsilon_k,
\end{split}
\end{equation}
where $C_0>0$ is a constant independent of $k$ and the point $p$, $\epsilon_k$ is a sequence independent of $p$ with $\lim_{k\To\infty}\epsilon_k=0$.

From \eqref{e-gue151018} and \eqref{e-gue151018aII}, the lemma follows.
\end{proof}

%We need the following which is essentially well-known (see Lemma 7.9 in~\cite{Hsiao14})

%\begin{lem}\label{l-gue131023}
%With the assumptions and notations above. Let
%\[x_k=(x^1_k,\ldots,x^{2n-1}_k)\in\Real^{2n-1},\ \ y_k=(y^1_k,\ldots,y^{2n-1}_k)\in\Real^{2n-1}\]
%with $\lim_{k\To\infty}(\sqrt{k}\sum^{2n-2}_{j=1}\abs{x^j_k}+k\abs{x^{2n-1}_k})=0$, $\lim_{k\To\infty}(\sqrt{k}\sum^{2n-2}_{j=1}\abs{y^j_k}+k\abs{y^{2n-1}_k})=0$. Then, for every $\alpha=(\alpha_1,\ldots,\alpha_{2n-1})\in\mathbb N_0^{2n-1}$, $\beta=(\beta_1,\ldots,\beta_{2n-1})$, there are constants $C_\alpha>0$, $C_{\alpha,\beta}>0$ independent of $k$ and the point $p$ such that for every $u_k\in \mathcal{H}^0_{b,\leq k\delta}(X,L^k)$ with $\norm{u_k}=1$. Put $F_{k,\delta}u_k=s^k\Td u_k$ on $D$. Then,
%\begin{equation}\label{e-gue131023II}
%\abs{\pr^\alpha_x\bigr(e^{-kR}e^{-k\Phi}\Td u\bigr)(x_k)}^2\leq C_{\alpha}k^{n+\abs{\alpha'}+2\alpha_{2n-1}},
%\end{equation}
%where $C_{\alpha}>0$ is a consatnt independent of $k$ and the point $p$ and $\abs{\alpha'}=\sum^{2n-2}_{j=1}\alpha_j$, $\abs{\beta'}=\sum^{2n-2}_{j=1}\beta_j$
%\end{lem}

\begin{proof}[Proof of Theorem~\ref{t-gue131019}]
We are going to prove that if $k$ is large enough then the map
\[d\Phi_{k,\delta}(x):T_xX\To T_{\Phi_{k,\delta}(x)}\Complex\mathbb P^{d_k-1},\]
is injective. Fix $p\in X$ and let $s$ be a local rigid CR frame of $L$
on a canonical coordinate patch $D$, $p\in D$, $\abs{s}^2_{h}=e^{-2\Phi}$,
with canonical local coordinates $x=(x_1,\ldots,x_{2n-1})=(z,\theta)$.
We take local coordinate $x$ and $s$ so that \eqref{e-geusw13623} hold. From Lemma~\ref{l-gue131007I}, we may assume that
\begin{equation}\label{e-gue131023II-I}
\abs{\bigr(e^{-k\Phi}\Td f_1\bigr)(p)}^2\geq c_0,
\end{equation}
where $F_{k,\delta}f_j=s^k\Td f_j$ on $D$, $j=1,\ldots,d_k$ and $c_0>0$
is a constant independent of $k$ and the point $p$.
Let $g^1_k,\ldots,g^n_k\in \mathcal{H}^0_{b,\leq k\delta}(X,L^k)$ be as in
Lemma~\ref{l-gue151009}. From \eqref{e-gue151009}, it is not difficult to see that there is a
$\hat K_0>0$ independent of the point $p$ such that $f_1,g^1_k,\ldots,g^n_k$
are linearly independent over $\Complex$. Put
\begin{equation}\label{e-gue131023IV}
\begin{split}
&p^j_k=\frac{e^{-k\Phi}\Td g^j_k}{e^{-k\Phi}\Td f_1},\ \ j=1,\ldots,n,\\
&p^j_k=\alpha^{2j-1}_k+i\alpha^{2j}_k,\ \ \alpha^{2j-1}_k={\rm Re\,}p^j_k,\ \ \alpha^{2j}_k={\rm Im\,}p^j_k,\ \ j=1,\ldots,n-1,
\end{split}
\end{equation}
where $F_{k,\delta}g^j_k=s^k\Td g^j_k$ on $D$, $j=1,\ldots,n$.
From \eqref{e-gue151009} and \eqref{e-gue131023II-I}, it is not difficult to see that there is a $\Td K_0>0$ independent of the point $p$ such that for all $k\geq\Td K_0$, we have
\begin{equation}\label{e-gue131024}
\abs{\pr_{z_t}p^t_k(p)}\geq c_1\sqrt{k},\ \ t=1,\ldots,n-1,\ \ \abs{\pr_{x_{2n-1}}p^n_k(p)}\geq c_1k,
\end{equation}
and
\begin{equation}\label{e-gue131024I}
\begin{split}
&{\sup\,}\{\abs{\pr_{x_{2n-1}}p^t_k(p)}, \abs{\pr_{z_s}p^t_k(p)}, \abs{\pr_{z_s}p^n_k(p)};\, s,t=1,\ldots,n-1, s\neq t\}\\
&\quad+\sup\,\{\abs{p^t_k(p)}, \abs{\pr_{\ol z_s}p^t_k(p)};\, s=1,\ldots,n-1, t=1,\ldots,n\}\leq\varepsilon_k,
\end{split}
\end{equation}
where $\varepsilon_k$ is a sequence independent of the point $p$ with $\lim_{k\To\infty}\varepsilon_k=0$.
From \eqref{e-gue131024}, \eqref{e-gue131024I} and some elementary linear algebra argument, we conclude that
there is a $K_1>0$ independent of the point $p$ such that for every $k\geq K_1$, the linear map
$A_k:\Real^{2n-1}\To\Real^{2n}$ represented by the matrix
\[
\left[
\begin{array}[c]{cccc}
  \pr_{x_1}\bigr(e^{-k\Phi}\alpha^1_k\bigr)(p)&\pr_{x_2}\bigr(e^{-k\Phi}\alpha^1_k\bigr)(p)&\ldots&\pr_{x_{2n-1}}\bigr(e^{-k\Phi}\alpha^{1}_k\bigr)(p)\\
  \pr_{x_1}\bigr(e^{-k\Phi}\alpha^2_k\bigr)(p)&\pr_{x_2}\bigr(e^{-k\Phi}\alpha^2_k\bigr)(p)&\ldots&\pr_{x_{2n-1}}\bigr(e^{-k\Phi}\alpha^{2}_k\bigr)(p)\\
 \vdots&\vdots&\vdots&\vdots\\
\pr_{x_1}\bigr(e^{-k\Phi}\alpha^{2n}_k\bigr)(p)&\pr_{x_2}\bigr(e^{-k\Phi}\alpha^{2n}_k\bigr)(p)&\ldots&\pr_{x_{2n-1}}\bigr(e^{-k\Phi}\alpha^{2n}_k\bigr)(p)
\end{array}\right],
\]
is injective. Hence the differential of the map
\[X\ni x\longmapsto\left(\frac{\Td g^1_k}{\Td f_1}(x),\ldots,\frac{\Td g^n_k}{\Td f_1}(x)\right)
\in\Complex^n\]
at $p$ is injective if $k\geq K_1$.
From this and some elementary linear algebra arguments, we conclude that
the differential of the map
\[X\ni x\longmapsto\left(\frac{\Td f_2}{\Td f_1}(x),\ldots,\frac{\Td f_{d_k}}{\Td f_1}(x)\right)
\in\Complex^{d_k}\]
at $p$ is injective if $k\geq K_1$. Theorem~\ref{t-gue131019} follows.
\end{proof}

In the rest of this section, we will prove that for $k$ large enough,
the map $\Phi_{k,\delta}:X\To\Complex\mathbb P^{d_k-1}$ is injective.
We need some preparations. Let $(D,(z,\theta),\phi)$ be a BRT trivialization.
We write $x=(z,\theta)=(x_1,\ldots,x_{2n-1})$, $x'=(x_1,\ldots,x_{2n-2},0)$,
$z_j=x_{2j-1}+ix_{2j}$, $j=1,\ldots,n-1$. We need

\begin{lem}\label{l-gue151031}
With the notations above, for every $u_k\in C^\infty(X,L^k)$, we have
\begin{equation}\label{e-gue151031}
\begin{split}
&(F_{k,\delta}u_k)(x)\\
&=\frac{1}{(2\pi)^2}\sum_{m\in\mathbb Z}\int e^{i(x_{2n-1}-y_{2n-1})\eta_{2n-1}}\tau_\delta\Big(\frac{\eta_{2n-1}}{k}\Big)e^{imy_{2n-1}}e^{-im\theta}u_k(e^{i\theta}  x')d\theta d\eta_{2n-1}dy_{2n-1}\ \ \mbox{on $D$}.
\end{split}
\end{equation}
\end{lem}

\begin{proof}
Put $\tau_{\delta,k}(\eta_{2n-1}):=\tau_\delta\Big(\frac{\eta_{2n-1}}{k}\Big)$. By Fourier inversion formula, we have
\begin{equation}\label{e-gue151031I}
\begin{split}
&\frac{1}{(2\pi)^2}\sum_{m\in\mathbb Z}\int e^{i(x_{2n-1}-y_{2n-1})\eta_{2n-1}}\tau_\delta\Big(\frac{\eta_{2n-1}}{k}\Big)e^{imy_{2n-1}}e^{-im\theta}u_k(e^{i\theta}  x')d\theta d\eta_{2n-1}dy_{2n-1}\\
&=\frac{1}{(2\pi)^2}\sum_{m\in\mathbb Z}\int\hat\tau_{\delta,k}(y_{2n-1}-x_{2n-1})e^{imy_{2n-1}}e^{-im\theta}u_k(e^{i\theta}  x')d\theta dy_{2n-1}\\
&=\frac{1}{(2\pi)^2}\sum_{m\in\mathbb Z}\int\hat\tau_{\delta,k}(y_{2n-1})e^{imy_{2n-1}+imx_{2n-1}}e^{-im\theta}u_k(e^{i\theta}  x')d\theta dy_{2n-1}\\
&=\frac{1}{2\pi}\sum_{m\in\mathbb Z}\int\tau_{\delta,k}(m)e^{imx_{2n-1}}e^{-im\theta}u(e^{i\theta}  x')d\theta\\
&=\frac{1}{2\pi}\sum_{m\in\mathbb Z}\int\tau_{\delta}\Big(\frac{m}{k}\Big)e^{-im\theta}u(e^{i\theta}  x)d\theta\\
&=F_{k,\delta}u_k,
\end{split}
\end{equation}
where $\hat\tau_{\delta,k}$ denotes the Fourier transform of $\tau_{\delta,k}$. From \eqref{e-gue151031I}, the lemma follows.
\end{proof}

%We also need

\begin{lem}\label{l-gue151031I}
With the notations above, let $u_k\in C^\infty(X,L^k)$. Assume that there are constants
$C>0$ and $M>0$ independent of $k$ such that $\abs{u_k(x)}^2_{h^k}\leq Ck^M$, for every $x\in X$.
If ${\rm Supp\,}u_k\bigcap D=\emptyset$ for every $k$,
then $F_{k,\delta}u_k=O(k^{-\infty})$ on $D$.
\end{lem}

\begin{proof}
Assume that $D=U\times(-\epsilon_0,\epsilon_0)$, where $U$ is an open set in
$\Complex^{n-1}$ and $\epsilon_0>0$.
Fix $D'\Subset D$ and let $\chi(y_{2n-1})\in C^\infty_0((-\epsilon_0,\epsilon_0))$
such that $\chi(y_{2n-1})=1$ for every $(y',y_{2n-1})\in D'$. Let
\begin{equation}\label{e-gue151031a}
\begin{split}
&R_{k}u_k(x)=\frac{1}{(2\pi)^2}\sum_{m\in\mathbb Z}
\int_{\abs{\theta}\leq\pi}e^{i(x_{2n-1}-y_{2n-1})\eta_{2n-1}}\tau_\delta\Big(\frac{\eta_{2n-1}}{k}\Big)\\
&\quad\quad\times (1-\chi(y_{2n-1}))e^{imy_{2n-1}}e^{-im\theta}u_k(e^{i\theta}  x')d\theta d\eta_{2n-1}dy_{2n-1},
\end{split}
\end{equation}
where $x\in D'$. Since $\chi(y_{2n-1})=1$ for every $(y',y_{2n-1})\in D'$, we can integrate by parts with respect to $\eta_{2n-1}$ several times and deduce that
\begin{equation}\label{e-gue151031aI}
R_{k}u_k(x)=O(k^{-\infty})\ \ \mbox{on $D'$}.
\end{equation}
From \eqref{e-gue151031} and \eqref{e-gue151031a}, we have
\begin{equation}\label{e-gue151031aII}
\begin{split}
&(F_{k,\delta}u_k-R_ku_k)(x)\\
&=\frac{1}{(2\pi)^2}\sum_{m\in\mathbb Z}
\int_{\abs{\theta}\leq\pi}e^{i(x_{2n-1}-y_{2n-1})\eta_{2n-1}}\tau_\delta\Big(\frac{\eta_{2n-1}}{k}\Big)\chi(y_{2n-1})e^{imy_{2n-1}}e^{-im\theta}u_k(e^{i\theta}  x')d\theta d\eta_{2n-1}dy_{2n-1}\\
&=\frac{1}{(2\pi)}
\int e^{i(x_{2n-1}-y_{2n-1})\eta_{2n-1}}\tau_\delta\Big(\frac{\eta_{2n-1}}{k}\Big)\chi(y_{2n-1})u_k(x_1,\ldots,x_{2n-2},y_{2n-1})d\eta_{2n-1}dy_{2n-1}=0
\end{split}
\end{equation}
since  ${\rm Supp\,}u_k\bigcap D=\emptyset$. From \eqref{e-gue151031aI} and \eqref{e-gue151031aII}, the lemma follows.
\end{proof}

We need the following CR peak sections lemma.

\begin{lem}\label{l-gue151031II}
Let $p\neq q$ two points in $X$ and let $\set{x_k}$, $\set{y_k}$ be two sequences in $X$ with $x_k\To p$, $y_k\To q$.
Then there exist $v_k\in\mathcal{H}^0_{b,\leq k\delta}(X,L^k)$
such that $u_k=F_{k,\delta}v_k$ satisfies for $k$ large enough,
\begin{equation}\label{e-gue151031b}
\abs{u_k(x_k)}^2_{h^k}\geq1,\ \ \abs{u_k(y_k)}^2_{h^k}\leq\frac{1}{2}\,\cdot
\end{equation}
\end{lem}
\begin{proof}
Let $(D,(z,\theta),\phi)$ be a BRT trivialization with $p\in D$ and $q\notin D$.
We may assume that $p=(0,0)$. %and $x_k\in D$, for every $k$.
%We write $x=(z,\theta)=(x_1,\ldots,x_{2n-1})$, $z_j=x_{2j-1}+ix_{2j}$, $j=1,\ldots,n-1$, $x_{2n-1}=\theta$.
As before, let
\[f_1\in\mathcal{H}^0_{b,\leq k\delta}(X,L^k),\ldots,f_{d_k}\in\mathcal{H}^0_{b,\leq k\delta}(X,L^k)\]
be an orthonormal basis for $\mathcal{H}^0_{b,\leq k\delta}(X,L^k)$ with respect
to $(\,\cdot\,|\,\cdot\,)$. Let $s$ be a local rigid CR frame of $L$ on an open neighbourhood
$D\subset X$ of $p$, $\abs{s(x)}^2_{h}=e^{-2\Phi}$.
Let $\chi\in C^\infty_0(D)$, $\chi=1$ on $D_0$, where $D_0\subset D$ is an open set of $p$.
%We may assume that $x_k\in D_0$, for every $k$.
On $D$, put $F_{k,\delta}f_j(x)=s^k\Td f_j(x)$, $\Td f_j(x)\in C^\infty(D)$, $j=1,\ldots,d_k$.
Assume that $\set{x_k}\Subset D_0$. 
Let
\begin{equation}\label{e-gue151031g}
\Td v_k(x)=s^k(x)\otimes\sum^{d_k}_{j=1}\chi(x)(F_{k,\delta}f_j)(x)\ol{\Td f_j(x_k)}e^{-k\Phi(x_k)}
\in C^\infty_0(D,L^k)\subset C^\infty(X,L^k).
\end{equation}
In view of Theorem~\ref{t-gue150807}, we see that
\begin{equation}\label{e-gue151031gI}
\Td v_k(x)=s^k(x)\otimes\chi(x)\int e^{ik\varphi(x,x_k,t)+k\Phi(x)}g(x,x_k,t,k)dt+(k^{-\infty})\ \ \mbox{on $D$}.
\end{equation}
Since $\int e^{ik\varphi(x,x_k,t)}g(x,x_k,t,k)dt=O(k^{-\infty})$ on $D\setminus D_0$ and
\[\mbox{$\Box^{(0)}_{b,k}(s^k(x)\otimes\int e^{ik\varphi(x,x_k,t)+k\Phi(x)}g(x,x_k,t,k)dt)=O(k^{-\infty})$ on $D$},\]
we conclude that
\begin{equation}\label{e-gue151031gII}
\Box^{(0)}_{b,k}\Td v_k=O(k^{-\infty})\ \ \mbox{on $D$}.
\end{equation}
Let $v_k=\Pi_{k,\leq k\delta}\Td v_k\in\mathcal{H}^0_{b,\leq k\delta}(X,L^k)$ and let $u_k=F_{k,\delta}v_k=F_{k,\delta}\Pi_{k,\leq k\delta}\Td v_k$. From \eqref{e-gue131208} and \eqref{e-gue151031gII}, we can check that
\begin{equation}\label{e-gue151031gIII}
\norm{F_{k,\delta}(I-\Pi_{k,\leq k\delta})\Td v_k}=O(k^{-\infty}).
\end{equation}
Form Kohn's estimates, it is straightforward to see that for every $s\in\mathbb N_0$, there are $N_s\in\mathbb N$, $C_s>0$ independent of $k$ such that
\begin{equation}\label{e-gue151031r}
\|F_{k,\delta}(I-\Pi_{k,\leq k\delta})\Td v_k\|_{s+1,k}
\leq C_sk^{N_s}\Bigr(\|\Box^{(0)}_{b,k}F_{k,\delta}(I-\Pi_{k,\leq k\delta})\Td v_k\|_{s,k}+
\|TF_{k,\delta}(I-\Pi_{k,\leq k\delta})\Td v_k\|_{s,k}\Bigr),
\end{equation}
where $\norm{\cdot}_{s,k}$ denotes the standard Sobolev norm of order $s$ on the Sobolev
space $H^s(X,L^k)$.
From \eqref{e-gue151031r}, by using induction and the estimate
\[\norm{TF_{k,\delta}(I-\Pi_{k,\leq k\delta})\Td v_k}\leq k\delta\norm{F_{k,\delta}(I-\Pi_{k,\leq k\delta})\Td v_k},\]
it is straightforward to see that for every $s\in\mathbb N$, there are $\Td N_s\in\mathbb N$, $\Td C_s>0$ independent of $k$ such that
\begin{equation}\label{e-gue151031rI}
\norm{F_{k,\delta}(I-\Pi_{k,\leq k\delta})\Td v_k}_{s,k}\leq \Td C_sk^{\Td N_s}\Bigr(\sum^s_{j=0}\norm{(\Box^{(0)}_{b,k})^jF_{k,\delta}(I-\Pi_{k,\leq k\delta})\Td v_k}\Bigr).
\end{equation}
From \eqref{e-gue151031gII}, \eqref{e-gue151031gIII} and \eqref{e-gue151031rI}, we deduce that
\begin{equation}\label{e-gue151031rII}
F_{k,\delta}(I-\Pi_{k,\leq k\delta})\Td v_k= O(k^{-\infty})
\end{equation}
and thus
\begin{equation}\label{e-gue151031rIII}
u_k= F_{k,\delta}\Td v_k+O(k^{-\infty}).
\end{equation}

Let $\Td\chi\in C^\infty_0(D)$, $\Td\chi(x_k)=1$ for each $k$ and $\chi=1$ on ${\rm Supp\,}\Td\chi$.
We can repeat the proof of \eqref{e-gue150813vgVIII} with minor change and deduce that
\begin{equation}\label{e-gue151031m}
\mbox{$\sum\limits^{d_k}_{j=1}\abs{\Td\chi F_{k,\delta}(1-\chi)F_{k,\delta}f_j(x)}^2_{_{h^k}}
=O(k^{-\infty})$ on $D$}
\end{equation}
and hence
\begin{equation}\label{e-gue151031mI}
\begin{split}
\abs{F_{k,\delta}\Td v_k(x_k)}^2_{h^k}&=
\Big|\sum^{d_k}_{j=1}(F_{k,\delta}f_j)(x_k)\ol{\Td f_j(x_k)}e^{-k\Phi(x_k)}\Big|^2_{h^k}+
O(k^{-\infty})\\
&=\int e^{ik\varphi(x_k,x_k,t)}g(x_k,x_k,t,k)dt+O(k^{-\infty})\\
&\geq Ck^n,
\end{split}
\end{equation}
where $C>0$ is a constant independent of $k$.

Note that ${\rm Supp\,}\Td v_k\subset D$ and $q\notin D$. From this observation and
Lemma~\ref{l-gue151031I}, we deduce that
%\begin{equation}\label{e-gue151031mII}
%\abs{F_{k,\delta}\Td v_k(y_k)}^2_{h^k}\equiv0\mod O(k^{-\infty}).
%\end{equation}
\begin{equation}\label{e-gue151031mII}
\abs{F_{k,\delta}\Td v_k(y_k)}^2_{h^k}=O(k^{-\infty}).
\end{equation}
From \eqref{e-gue151031rIII}, \eqref{e-gue151031mI} and \eqref{e-gue151031mII}, the lemma follows.
\end{proof}

\begin{thm}\label{t-gue151031}
%With the assumptions and notations above, for $k$ large,
The map $\Phi_{k,\delta}$
is injective for $k$ large enough.
\end{thm}

\begin{proof}
We assume that the claim of the theorem is not true.
We can find $x_{k_j}, y_{k_j}\in X$, $x_{k_j}\neq y_{k_j}$, $0<k_1<k_2<\ldots$,
$\lim_{j\To\infty}k_j=\infty$, such that $\Phi_{k_j,\delta}(x_{k_j})=\Phi_{k_j,\delta}(y_{k_j})$,
for each $j$. We may suppose that there are $x_{k}, y_k\in X$, $x_k\neq y_{k}$,
such that $\Phi_{k,\delta}(x_k)=\Phi_{k,\delta}(y_k)$, for each $k$.
We may assume that $x_k\To p\in X$, $y_k\To q\in X$, as $k\To\infty$.
If $p\neq q$. From Lemma~\ref{l-gue151031II},
we can find $u_k=F_{k,\delta}v_k$,
$v_k=F_{k,\delta}g_k$, $v_k, g_k\in\mathcal{H}^0_{b,\leq k\delta}(X,L^k)$
such that for $k$ large, we have
\begin{equation}\label{e-gue131103IV}
\abs{u_k(x_k)}^2_{h^k}\geq 1,\ \ \abs{u_k(y_k)}^2_{h^k}\leq\frac{1}{2},
\end{equation}
and
\begin{equation}\label{e-gue131103V}
\abs{v_k(y_k)}^2_{h^k}\geq 1,\ \ \abs{v_k(x_k)}^2_{h^k}\leq\frac{1}{2}.
\end{equation}
Now,  $\Phi_{k,\delta}(x_k)=\Phi_{k,\delta}(y_k)$ implies that
\[\abs{u_k(x_k)}^2_{h^k}=r_k\abs{u_k(y_k)}^2_{h^k},\ \ \abs{v_k(x_k)}^2_{h^k}=
r_k\abs{v_k(y_k)}^2_{h^k},\]
where $r_k\in\Real_+$, for each $k$. \eqref{e-gue131103IV} implies that $r_k\geq 2$, for $k$ large. But  \eqref{e-gue131103V} implies that $r_k\leq\frac{1}{2}$, for $k$ large. We get a contradiction. Thus, we must have $p=q$.

Let $\{f_j\}_{j=1}^{d_k}$ be an orthonormal basis of $\mathcal{H}^0_{b, \leq k\delta}(X, L^k)$. Let $s$ be a local rigid CR frame of $L$ on a BRT trivialization $(D,(z,\theta),\phi)$, $p\in D$, $\abs{s}^2_{h}=e^{-2\Phi}$. Write $F_{k,\delta}f_j=s^k\otimes\widetilde f_j, j=1, \ldots, d_k.$ We assume that
\begin{equation}
e^{-k\Phi(x_k)}\widetilde f_j(x_k)=\lambda_k e^{-k\Phi(y_k)}\widetilde f_j(y_k), \lambda_k\in\mathbb C, |\lambda_k|\geq 1.
\end{equation}
This implies that
\begin{equation}\label{c2}
P_{k, \delta, s}(x_k, y_k)=\lambda_k P_{k, \delta, s}(y_k, y_k), \lambda_k\in\mathbb C,   |\lambda_k|\geq 1.
\end{equation}
We will show that (\ref{c2}) is impossible.  Write $x_k=(z^k, x_{2n-1}^k)=(x^k_1,\ldots,x^k_{2n-1})$, $y_k=(w^k, y_{2n-1}^k)=(y^k_1,\ldots,y^k_{2n-1})$ and
\[z^k=(z_1^k, \ldots, z_{n-1}^k),\ \  w^k=(w_1^k, \ldots, w_{n-1}^k).\]
Without loss of generality, we assume that
\begin{equation*}
\lim\limits_{k\rightarrow\infty }k|z^k-w^k|^2=M, M\in [0, \infty].
\end{equation*}
Case I: $M\in (0, \infty].$ First we assume that $M=\infty$. From (\ref{c1}) we have
\begin{equation}\label{d6}
\begin{split}
\limsup\limits_{k\rightarrow\infty}k^{-n}|P_{k, \delta, s}(x_k, y_k)|&\leq\limsup\limits_{k\rightarrow\infty}\int e^{-k{\rm Im}\varphi(x_k, y_k, t)}\left|g_0(x_k, y_k, t)\right|dt.
\end{split}
\end{equation}
Combining with (\ref{d6}) and the fact ${\rm Im}\varphi(x_y, y_k, s)\geq c|z^k-w^k|^2$ 
in (\ref{e-gue150807b}) we have
$$\limsup\limits_{k\rightarrow\infty}k^{-n}|P_{k, \delta, s}(x_k, y_k)|=0.$$ 
This is a contradiction with 
$\lim\limits_{k\rightarrow\infty}k^{-n}P_{k, \delta, s}(y_k, y_k)=
\int g_0(p, p, t)dt\neq 0$ and the assumption (\ref{c2}). 
Thus we have $M<\infty$. From (\ref{c1}) we have
\begin{equation}
\lim_{k\rightarrow\infty}k^{-n}|P_{k, \delta, s}(x_k, y_k)|\leq e^{-c M}\int g_0(p, p, t)dt
\end{equation}
for some positive constant $c$. On the other hand 
$\lim\limits_{k\rightarrow\infty} k^{-n}|P_{k, \delta, s}(y_k, y_k)|=
\int g_0(p, p, t)dt.$ This is a contradiction with (\ref{c2}). Thus we have $M=0$, that is
\begin{equation}
\lim_{k\rightarrow\infty}k|z^k-w^k|^2=0.
\end{equation}
Set
\begin{equation}
\widehat{\widehat{\alpha_k}}=\sqrt{-1}\sum_{j=1}^{n-1}
\left[\frac{\partial\phi(z^k)}{\partial\overline z_j}(\overline z_j^k-\overline w_j^k)-
\frac{\partial\phi(z^k)}{\partial z_j}(z_j^k-w_j^k)\right]\in \mathbb R.
\end{equation}
Recall that
\[\omega_0(x)=-dx_{2n-1}+i\sum^{n-1}_{j=1}
\Big(\frac{\pr\phi}{\pr\ol z_j}(z)d\ol z_j-\frac{\pr\phi}{\pr z_j}(z)dz_j\Big),\ \ x=(z,\theta).\]
Without loss of generality, we assume that
\begin{equation}
\lim_{k\rightarrow\infty}k(y^k_{2n-1}-x_{2n-1}^k+\widehat{\widehat{\alpha_k}})=N\in [0, \infty].
\end{equation}

Case II:
\begin{equation}
\lim_{k\rightarrow\infty}k|z^k-w^k|^2=0; 
\lim_{k\rightarrow\infty}k(y^k_{2n-1}-x_{2n-1}^k+\widehat{\widehat{\alpha_k}})=N\in (0, \infty].
\end{equation}
First we assume $N=\infty$. From (\ref{c1}) we have
\begin{equation}\label{c4}
k^{-n}P_{k, \delta, s}(x_k, y_k)=k^{-n}\int e^{ik\varphi(x_k, y_k, t)}g(x_k, y_k, t, k)dt+r_k,
\end{equation}
where $|r_k|=O(k^{-\infty})$. By the second property in (\ref{e-gue150807b}) we have
\begin{equation}\label{c3}
\begin{split}
\varphi(x, y, t)=&t(y_{2n-1}-x_{2n-1})+
ti\sum_{j=1}^{n-1}\left[\frac{\partial\phi}{\partial\overline z_j}(z)(\overline z_j-\overline w_j)
-\frac{\partial\phi}{\partial z_j}(z)(z_j-w_j)\right]\\
&+i\sum_{j=1}^{n-1}\left[\frac{\partial\Phi}{\partial\overline z_j}(z)
(\overline z_j-\overline w_j)-\frac{\partial\Phi}{\partial z_j}(z)(z_j-w_j)\right]+O(|x-y|^2).
\end{split}
\end{equation}
Note that
\begin{equation}\label{d7}
k|x_k-y_k|^2\lesssim k|z^k-w^k|^2+k|y_{2n-1}^k-x_{2n-1}^k+
\widehat{\widehat{\alpha_k}}|^2\lesssim k|z^k-w^k|^2+k|y_{2n-1}^k-x_{2n-1}^k+
\widehat{\widehat{\alpha_k}}|\epsilon_k,
\end{equation}
where $\epsilon_k\rightarrow0$. From (\ref{c3}), \eqref{d7} and the assumption $N=\infty$ we have
\begin{equation}\label{d2}
\lim_{k\rightarrow\infty}k\frac{\partial\varphi(x_k, y_k, t)}{\partial t}=\infty.
\end{equation}
Substituting (\ref{c3}) to (\ref{c4}), by (\ref{d2}) and integrating by parts with respect to $t$ we have
\begin{equation}
\limsup\limits_{k\rightarrow\infty}k^{-n}|P_{k, \delta, s}(x_k, y_k)|=0.
\end{equation}
This is a contradiction with (\ref{c2}) since 
$\lim_{k\To\infty}k^{-n}P_{k,\delta,s}(y_k,y_k)=
\int g_0(p,p,t)dt\neq0$. Second, we assume $N<\infty$. 
Since $\lim\limits_{k\rightarrow\infty}k|z^k-w^k|^2=0$ and 
$\lim\limits_{k\rightarrow\infty}k(y_{2n-1}^k-x_{2n-1}^k+
\widehat{\widehat{\alpha_k}})=N<\infty$, by (\ref{d7}) 
we have that $\lim\limits_{k\rightarrow\infty}k|x_k-y_k|^2=0.$  
Substituting (\ref{c3}) to (\ref{c4}) we have that
\begin{equation}\label{c5}
\begin{split}
\limsup_{k\rightarrow\infty}k^{-n}|P_{k, \delta, s}(x_k, y_k)&\leq 
\limsup_{k\rightarrow\infty}k^{-n}\left|\int e^{ik\left[t(y_{2n-1}-x_{2n-1}+
\widehat{\widehat{\alpha_k}})+O(|x_k-y_k|^2)\right]}g(x_k, y_k, t, k)dt\right|\\
&\leq \left|\int e^{iNt}g_0(p, p, t)dt\right|.
\end{split}
\end{equation}
Since $\left|\int e^{iNt}g_0(p, p, t)dt\right|<\int g_0(p, p, t)dt=
\lim_{k\To\infty}k^{-n}P_{k,\delta,s}(y_k,y_k)$, combining this with 
(\ref{c2}) and (\ref{c5}) we get a contradiction. Thus we have $N=0$.

Case III:
\begin{equation}\label{d4}
\lim_{k\rightarrow\infty}k|z^k-w^k|^2=0; 
\lim_{k\rightarrow\infty}k(y^k_{2n-1}-x_{2n-1}^k+\widehat{\widehat{\alpha_k}})=0.
\end{equation}
Define $$A_k(u)=\left|P_{k, \delta, s}(ux_k+(1-u)y_k, y_k)\right|^2, B_k(u)=
P_{k, \delta, s}(ux_k+(1-u)y_k,ux_k+(1-u)y_k)\cdot P_{k, \delta, s}(y_k,y_k).$$
Set $H_k(u)=\frac{A_k(u)}{B_k(u)}$. By Schwartz inequality, we have 
$0\leq H_k(u)\leq 1.$ Since $H_k(0)=H_k(1)=1$, then there exists a 
$u_k\in(0, 1)$ such that $H_k^{''}(u_k)\geq0.$ By direct calculation,
\begin{equation}
H_k^{''}(u_k)=\frac{A_k^{''}(u_k)}{B_k(u_k)}-2\frac{A_k^{'}(u_k)B_k^{'}(u_k)}{B_k^2(u_k)}
-\frac{A_k(u_k)B_k^{''}(u_k)}{B_k^2(u_k)}+2\frac{A_k(u_k)B_k^{'2}(u_k)}{B_k^3(u_k)}.
\end{equation}
Write $\alpha_k(u)=P_{k, \delta, s}(ux_k+(1-u)y_k, y_k)$. 
Then $A_k(u)=|\alpha_k(u)|^2$, $A_k^\prime(u)=\alpha_k^\prime(u)\overline{\alpha_k(u)}+
\alpha_k(u)\overline{\alpha_k^\prime(u)}$ and 
\begin{equation}A_k^{\prime\prime}(u_k)=\alpha_k^{\prime\prime}(u_k)\overline{\alpha_k(u_k)}+
2|\alpha_k^\prime(u_k)|^2+\alpha_k(u_k)\overline{\alpha_k^{\prime\prime}(u_k)}.\end{equation}
By Theorem \ref{t-gue150811}, we have
\begin{equation}
\alpha_k(u)=\int e^{ik\varphi(ux_k+(1-u)y_k, y_k, t)}g(ux_k+(1-u)y_k, y_k, t, k)dt+\gamma_k(u),
\end{equation}
where $\gamma_k(u)=O(k^{-\infty}).$ Write $\beta_k(u)=
\int e^{ik\varphi(ux_k+(1-u)y_k, y_k, t)}g(ux_k+(1-u)y_k, y_k, t, k)dt$ and then
\begin{equation}\label{d111}
\begin{split}
A_k^{\prime\prime}(u_k)&=2|\beta_k^\prime(u_k)|^2+\beta_k^{\prime\prime}(u_k)
\overline{\beta_k(u_k)}+\overline{\beta^{\prime\prime}(u_k)}\beta_k(u_k)+2\beta^\prime_k(u_k)
\overline{\gamma_k^\prime(u_k)}+2\gamma_k^\prime(u_k)\overline{\beta_k^\prime(u_k)}\\
&\quad+\beta_k^{\prime\prime}(u_k)\overline{\gamma_k(u_k)}+\gamma_k(u_k)\overline{\beta_k^{\prime\prime}
(u_k)}+\overline{\beta_k(u_k)}\gamma_k^{\prime\prime}(u_k)+\beta_k(u_k)
\overline{\gamma_k^{\prime\prime}(u_k)}\\
&\quad+\gamma_k^{\prime\prime}(u_k)\overline{\gamma_k(u_k)}+
\gamma_k(u_k)\overline{\gamma^{\prime\prime}(u_k)}
+2|\gamma_k^\prime(u_k)|^2.
\end{split}
\end{equation}
Set
\begin{equation}
\widehat\alpha_k(u)=i\sum_{j=1}^{n-1}\left[\frac{\partial\phi}
{\partial\overline z_j}\Big|_{w^k+u(z^k-w^k)}\cdot(\overline z_j^k-
\overline w_j^k)-\frac{\partial\phi}{\partial z_j}\Big|_{w^k+u(z^k-w^k)}\cdot(z_j^k-w_j^k)\right].
\end{equation}
By the mean value theorem,
\begin{equation}\label{d5}
|k(y_{2n-1}^k-x_{2n-1}^k+\widehat{\alpha}_k(u_k))-k(y^k_{2n-1}-x_{2n-1}^k+\widehat{\widehat{\alpha_k}})|
=k|\widehat{\widehat{\alpha_k}}-\widehat{\alpha}_k(u_k)|\lesssim k|z^k-w^k|^2.
\end{equation}
Then (\ref{d4}) and (\ref{d5}) implies that
\begin{equation}
\lim\limits_{k\rightarrow\infty}k|y_{2n-1}^k-x_{2n-1}^k+\widehat{\alpha}_k(u_k)|=0.
\end{equation}
By direct calculation we have that
\begin{equation}\label{d8}
\begin{split}
&2|\beta_k^\prime(u_k)|^2+\beta_k^{\prime\prime}(u_k)\overline{\beta_k(u_k)}
+\overline{\beta_k^{\prime\prime}(u_k)}\beta_k(u_k)\\
&=2k^{2n+2}\left[\left|\int tg_0(p, p, t)dt\right|^2-
\int g_0(p, p, t)t^2dt\cdot\int g_0(p, p, t)dt\right](y_{2n-1}^k-x_{2n-1}^k+\widehat\alpha_k(u_k))^2\\
&-2k^{2n+1}\int\left[\sum_{j,l=1}^{2n-2}
\frac{\partial^2{\rm Im}\varphi(p, p, t)}{\partial x_j\partial x_l}(x^k_j-y^k_j)(x_l^k-y_l^k)\right]
g_0(p, p, t)dt\\
&+o(k^{2n})O(k|z^k-w^k|^2+k^2|y_{2n-1}^k-x_{2n-1}^k+\widehat{\alpha}_k(u_k)|^2).
\end{split}
\end{equation}
By \eqref{e-guew13627} there exists $c>0$ such that for $\delta$ sufficiently small the following holds,
\begin{equation}\label{d9}
\int\left[\sum_{j,l=1}^{2n-2}\frac{\partial^2{\rm Im}\varphi(p, p, t)}{\partial x_j\partial x_l}
(x^k_j-y^k_j)(x_l^k-y_l^k)\right]g_0(p, p, t)dt\geq c|z^k-w^k|^2.
\end{equation}
By H\"older's inequality, $\left|\int tg_0(p, p, t)dt\right|^2<\int t^2g_0(p, p, t)dt
\cdot\int g_0(p, p, t)dt$, so by combining (\ref{d8}) and (\ref{d9}) there exists $c_1>0$
such that
\begin{equation}\label{c7}
\begin{split}
\limsup_{k\rightarrow\infty}k^{-2n}&\left[k|z^k-w^k|^2+k^2(y_{2n-1}^k-x_{2n-1}^k
+\widehat\alpha_k(u_k))^2\right]^{-1}\times\\
&\left[2|\beta_k^\prime(u_k)|^2+\beta_k^{\prime\prime}(u_k)\overline{\beta_k(u_k)}
+\overline{\beta_k^{\prime\prime}(u_k)}\beta_k(u_k)\right]<-c_1<0.
\end{split}
\end{equation}
By direct calculation we have that
\begin{equation}\label{c8}
\begin{split}
&\limsup_{k\rightarrow\infty}k^{-2n}\left[k|z^k-w^k|^2+k^2(y_{2n-1}^k-x_{2n-1}^k
+\widehat\alpha_k(u_k))^2\right]^{-1}\times C_k=0,
\end{split}
\end{equation}
where
\begin{equation*}\label{c6}
\begin{split}
C_k=&2\beta^\prime_k(u_k)
\overline{\gamma_k^\prime(u_k)}+2\gamma_k^\prime(u_k)\overline{\beta_k^\prime(u_k)}
+\beta_k^{\prime\prime}(u_k)\overline{\gamma_k(u_k)}+\gamma_k(u_k)\overline{\beta_k^{\prime\prime}
(u_k)}+\\
&\overline{\beta_k(u_k)}\gamma_k^{\prime\prime}(u_k)+\beta_k(u_k)
\overline{\gamma_k^{\prime\prime}(u_k)}
+\gamma_k^{\prime\prime}(u_k)\overline{\gamma_k(u_k)}+\gamma_k(u_k)\overline{\gamma^{\prime\prime}(u_k)}
+2|\gamma_k^\prime(u_k)|^2.
\end{split}
\end{equation*}
Combining (\ref{c7}), (\ref{c8}) and (\ref{d111}) there exists $c_2>0$ such that
\begin{equation}\label{d1}
\limsup_{k\rightarrow\infty}\left[k|z^k-w^k|^2+k^2(y_{2n-1}^k-x_{2n-1}^k
+\widehat\alpha_k(u_k))^2\right]^{-1}\frac{A_k^{\prime\prime}(u_k)}{B_k(u_k)}<-c_2<0.
\end{equation}
It is straightforward to see that
\begin{equation}\label{d2b}
\begin{split}
&\limsup_{k\rightarrow\infty}\left[k|z^k-w^k|^2+k^2(y_{2n-1}^k-x_{2n-1}^k
+\widehat\alpha_k(u_k))^2\right]^{-1}\times\\
& \left\{2\frac{|A_k^{\prime}(u_k)|\cdot|B_k^{\prime}(u_k)|}{B_k^2(u_k)}
+\frac{|A_k(u_k)|\cdot|B_k^{\prime\prime}(u_k)|}{B_k^2(u_k)}+2\frac{|A_k(u_k)|\cdot|B_k^{\prime 2}(u_k)|}{B_k^3(u_k)}\right\}=0.
\end{split}
\end{equation}
From (\ref{d1}) and (\ref{d2b}) we have
\begin{equation}
\limsup_{k\rightarrow\infty}\left[k|z^k-w^k|^2+k^2(y_{2n-1}^k-x_{2n-1}^k
+\widehat\alpha_k(u_k))^2\right]^{-1}H_k^{\prime\prime}(u_k)<0.
\end{equation}
This is a contradiction with $H^{\prime\prime}_k(u_k)\geq 0.$
\end{proof}
Since $X$ is compact Theorems \ref{t-gue131019} and \ref{t-gue151031} 
implies the modified Kodaira map  $\Phi_{k, \delta}$ defined in (\ref{e-gue131019II}) 
is an embedding. For different $m_1, m_2\in\mathbb Z$, 
$\mathcal H^0_{b, m_1}(X, L^k)\perp \mathcal H^0_{b, m_2}(X, L^k)$, 
thus we can choose an orthonormal basis $\{f_j\}_{j=1}^{d_k}$ of 
$\mathcal H^0_{b, \leq k\delta}(X, L^k)$ such that 
$f_j\in \mathcal H^0_{b, m_j}(X, L^k)$ with $m_j\in\mathbb Z$ and 
$|m_j|\leq k\delta$ for each $1\leq j\leq d_k$. 
Then $F_{k, \delta}f_j\in \mathcal H^0_{b, m_j}(X)$ for each $j$. 
For any $p\in X$, from the argument in the proof of  \cite[Lemma 1.20]{HL15}, 
we can find a local trivialization $W$  which is an $S^1$ invariant neighborhood of 
$p$ and local trivializing rigid CR section $s$ of $L$ on $W$. 
Then on $W$, $F_{k, \delta}f_j=s^k\otimes\widetilde f_j $ with $\widetilde f_j\in C^\infty(W), 
1\leq j\leq d_k$. Since $F_{k, \delta}f_j\in \mathcal H^0_{b, m_j}(X, L^k)$, 
we have $T\widetilde f_j=im_j \widetilde f_j$. Then for any $\theta\in[0, 2\pi)$ 
we have $\widetilde f_j(e^{i\theta}p)=e^{im_j\theta}\widetilde f_j(p)$. 
Thus, $\Phi_{k, \delta}(e^{i\theta}p)=
[\widetilde f_1(e^{i\theta}p), \ldots, \widetilde f_{d_k}(e^{i\theta}p)]=
[e^{im_1\theta}\widetilde f_1(p), \ldots, e^{im_{d_k}\theta}\widetilde f_{d_k}(p)]=
e^{i\theta}\Phi_{k, \delta}(p)$. 
We get the conclusion of Theorem \ref{t-gue150807I}.

Corollaries \ref{C:Lf} and \ref{t-gue151230I} are immediate consequences of 
Theorem \ref{t-gue150807I}. We close with an application of Corollary \ref{t-gue151230I}.

\begin{ex}\label{ex_t-gue150807I}
Let $(X, T^{1,0}X)$ be a compact CR manifold 
of dimension $3$ with a transversal CR locally free $S^1$-action. 
Assume that $X$ admits a rigid positive CR line bundle $L$. 
For example, if $X$ is strongly pseudoconvex, there is a rigid positive CR 
line bundle over $X$. 
Take $Z\in C^\infty(X,T^{1,0}X)$ such that $Z_x$ is a basis for $T^{1,0}_xX$, 
for every $x\in X$. Let $h$ be a distribution on $X$ with $Th=0$ and $Zh$ smooth 
(note that it is possible that there is a non-smooth function $h$ such that $Zh$ is smooth). 
Hence, $Zh\in C^\infty(X)$. 
Consider $\hat T^{1,0}X:={\rm span\,}\langle Z+(Zh)T\rangle$. 
Then, $(X, \hat T^{1,0}X)$ is a compact CR manifold of dimension $3$ 
with a transversal CR locally free $S^1$-action.
Moreover, $L$ is still a rigid positive CR line bundle over $(X, \hat T^{1,0}X)$. 
To see this, let $s$ be a rigid CR frame with respect to $T^{1, 0}X$ 
and $|s|^2=e^{-2\phi}.$ Then $s$ is still a rigid CR frame with respect to 
$\hat T^{1, 0}X.$ Let $\hat{\overline\partial}_b$ be the tangential Cauchy-Riemann 
operator with respect to $\hat T^{1, 0}X$ and $\hat{\partial}_b$ its conjugate.
Then the curvature of $L$ is given by
$\hat R^L=2\hat\partial_b\hat{\overline\partial}_b\phi.$ 
Since $\hat{\overline\partial}_b\phi=(\overline Z+\overline Zh T)\phi=
\overline Z\phi d\overline z$ we have 
$\hat\partial_b\hat{\overline\partial}_b\phi=
Z\overline Z\phi\, dz\wedge d\overline z=
\partial_b\overline\partial_b\phi>0$.

From Theorem~\ref{t-gue150807I}, we deduce that there exists smooth CR 
embeddings $\Phi_{k,\delta}$ of $(X,\hat T^{1,0}X)$ in $\mathbb C\mathbb P^{d_k-1}$
which are $S^1$-equivariant with respect to weighted diagonal actions. 
\end{ex}

\end{document}